\documentclass[12pt]{amsart}
\usepackage{customstyle} 

\usepackage{customstyle} 

\definecolor{bluestroke}{HTML}{007BC2}
\definecolor{bluefill}{HTML}{C2F2FF}
\definecolor{redstroke}{HTML}{CC0000}
\definecolor{redfill}{HTML}{FF8080}

\newcommand{\ToCollideOrNot}{

    \begin{tikzpicture}[scale = 0.8]

    \begin{scope}[xshift=-4cm]

        \draw[very thick] (-1.5,0) .. controls (-0.6,-0.9) and (0.6,-0.2) .. (1.2,1); 

        \begin{scope}[xshift=-3cm,yshift=1cm]
            \draw[very thick] (-1.5,-1) .. controls (-0.6,-0.1) and (0.6,-0.2) .. (1.5,-1); 
        \end{scope}
    
        \draw[very thick, fill, bluestroke] (0.75,0.3) circle (0.13);
        \node at (0.3, 0.5) {\small $q$};
        \fill (-3,0.65) circle (0.15);
        \node at (-3, 0.1) {\small $p_{j_1}$};
        \fill (-1, -0.35) circle (0.15);
        \node at (-1, -0.9) {\small $p_{j_k}$};
        \node at (-2.05, -0.2) {\large $\ddots$};

    \end{scope}

    \begin{scope}[xshift=3cm]

        \draw[very thick] (-2.0,1.0) .. controls (-1.0,-0.85) and (1.0,0.85) .. (2,-1.0); 

        \draw[very thick, bluestroke] (1.25, -0.75) .. controls (1.75,0) and (3,0.75) .. (4, 1);

        \fill (1.52, -0.4) circle (0.1);
        \node at (1.4, -1) {\small $q'$};

        \draw[very thick, fill, bluestroke] (2.1, 0.1) circle (0.13);
        \node at (2.0, 0.6) {\small $p_{i_1}$};
        \draw[very thick, fill, bluestroke] (3.6, 0.85) circle (0.13);
        \node at (3.6, 1.4) {\small $p_{i_\ell}$};
        \node[rotate=60] at (2.6, 1.1) {\large $\ddots$};

        \fill (0.5,-0.05) circle (0.15);
        \node at (0.7, 0.45) {\small $p_{j_k}$};
        \fill (-1.35, 0.3) circle (0.15);
        \node at (-1.1, 0.8) {\small $p_{j_1}$};
        \node[rotate=15] at (-0.3, 0.75) {\large $\ddots$};

        \node at (-2.5, 1) {\small $W$};
        \node at (4.4, 1) {\small \color{bluestroke} $T$};
        
    \end{scope}
    
    \end{tikzpicture}

}

\newcommand{\ExtremalIsSimplicial}{

    \begin{tikzpicture}[scale = 0.6]

    \begin{scope}[xshift = -1.5cm]
    
        \draw[very thick]  (0,0) circle (0.20);
        \draw[very thick] (-0.1, -0.21) -- (-0.55, -0.8);
        \draw[very thick] (-0.1, 0.21) -- (-0.55, 0.8);
        \node at (-0.65, 0.17) {\large $\vdots$};
        \node at (-2.1,0) {\small $[n]-J$};
        
        \draw[very thick] (0.20,0) -- (1.5,0);
        \draw[very thick, fill] (1.5,0) circle (0.2);
        \draw[very thick] (1.4,0.21) -- (0.95,0.8);
        \draw[very thick] (1.6,0.21) -- (2.05,0.8);
        \node at (1.5,0.85) {\large $...$};
        \node at (1.5,1.5) {\small $J-I$};
        
        \draw[very thick] (1.5,0) -- (3,0); 
        \draw[very thick, fill] (3,0) circle (0.2);
        
        \draw[very thick] (3.1,0.21) -- (3.55, 0.8);
        \draw[very thick] (3.1,-0.21) -- (3.55,-0.8);
        \node at (3.65, 0.17) {\large $\vdots$};
        \node at (4.3,0) {\small $I$};

    \end{scope}

    \draw[very thick, -to] (-2.5,-1.5) -- (-4.75,-3.8);
    \node at (-4, -2.2) {\small $\pi_I$};

    \draw[very thick, -to] (2.5,-1.5) -- (4.75,-3.8);
    \node at (4, -2.2) {\small $\pi_J$};

    \begin{scope}[xshift=-6cm,yshift=-5cm]

        \draw[very thick]  (0,0) circle (0.20);
        \draw[very thick] (-0.1, -0.21) -- (-0.55, -0.8);
        \draw[very thick] (-0.1, 0.21) -- (-0.55, 0.8);
        \node at (-0.65, 0.17) {\large $\vdots$};
        \node at (-2.1,0) {\small $[n]-I$};
        
        \draw[very thick] (0.20,0) -- (1.5,0);
        \draw[very thick, fill] (1.5,0) circle (0.2);
        
        \draw[very thick] (1.6,0.21) -- (2.05, 0.8);
        \draw[very thick] (1.6,-0.21) -- (2.05,-0.8);
        \node at (2.15, 0.17) {\large $\vdots$};
        \node at (2.8, 0) {\small $I$};

    \end{scope}

    \begin{scope}[xshift=6cm,yshift=-5cm]

        \draw[very thick]  (0,0) circle (0.20);
        \draw[very thick] (-0.1, -0.21) -- (-0.55, -0.8);
        \draw[very thick] (-0.1, 0.21) -- (-0.55, 0.8);
        \node at (-0.65, 0.17) {\large $\vdots$};
        \node at (-2.1,0) {\small $[n]-J$};
        
        \draw[very thick] (0.20,0) -- (1.5,0);
        \draw[very thick, fill] (1.5,0) circle (0.2);
        
        \draw[very thick] (1.6,0.21) -- (2.05, 0.8);
        \draw[very thick] (1.6,-0.21) -- (2.05,-0.8);
        \node at (2.15, 0.17) {\large $\vdots$};
        \node at (2.8, 0) {\small $J$};
    
    \end{scope}
    
    \end{tikzpicture}

}

\newcommand{\DifferentCollisionLimits}{

    \begin{tikzpicture}[scale = 0.6]

    \begin{scope}[xshift=-2.5cm,yshift=4.5cm, scale=5/6]

        \draw[very thick] (0,0) ellipse (2cm and 1.0cm);
        \draw[very thick] (-0.8,0) .. controls (-0.6,0.35) and (0.6,0.35) .. (0.8,0); 
        \draw[very thick] (-1,0.15) .. controls (-0.7,-0.25) and (0.7,-0.25) .. (1,0.15); 
    
        \draw[very thick] (1.7,0.5) -- (4,1);
    
        \draw[very thick] (4,1) -- (6.3, 1.5);
            \draw[very thick] (6.3, 1.5) -- (7.5, 2);
            \node at (8.0, 2.1) {\small $p_3$};
            \draw[very thick] (6.3, 1.5) -- (7.5, 1);
            \node at (8.0, 1.1) {\small $p_4$};
        \draw[very thick, fill] (6.3, 1.5) circle (0.2);
    
        \draw[very thick, fill] (4,1) circle (0.2);
    
        \draw[very thick] (1.7,-0.5) -- (4, -1);
        \draw[very thick, fill] (4,-1) circle (0.2);
            \draw[very thick] (4,-1) -- (5.2, -0.5);
            \node at (5.7, -0.4) {\small $p_1$};
            \draw[very thick] (4,-1) -- (5.2, -1.5);
            \node at (5.7, -1.4) {\small $p_2$};

    \end{scope}

    \begin{scope}[xshift=-6cm]

        \draw[very thick] (-1.5,0) .. controls (-0.6,-0.9) and (0.6,-0.2) .. (1.2,1); 
        \draw[very thick] (-1.5,-1) .. controls (-0.6,-0.1) and (0.6,-0.2) .. (1.5,-1); 
    
        \fill (0.7,0.3) circle (0.15);
        \node at (2.2, 0.3) {\small $p_1 = p_2$};
    
        \fill (0.5,-0.45) circle (0.15);
        \node at (0.4, -0.95) {\small $p_4$};
    
        \fill (1.25, -0.8) circle (0.15);
        \node at (1.9, -0.6) {\small $p_3$};

    \end{scope}

    \begin{scope}[xshift=6cm]

        \draw[very thick] (-1.5,1) -- (1.5,-1);
        \draw[very thick] (-1.5,-1) -- (1.5,1);
        \draw[very thick] (0,-1.5) -- (0,1.5);
    
        \fill (-1.05,0.7) circle (0.15);
        \node at (-1.05, 0.2) {\small $p_2$};
    
        \fill (1.05,0.7) circle (0.15);
        \node at (1.15, 0.2) {\small $p_1$};
    
        \fill (0,1.2) circle (0.15);
        \node at (0.55,1.25) {\small $p_3$};
    
        \fill (0,-1.2) circle (0.15);
        \node at (-0.5,-1.2) {\small $p_4$};
    
    \end{scope}
    
    \end{tikzpicture}

}

\newcommand{\TailFunction}{

    \begin{tikzpicture}[scale = 0.8]

    \draw[very thick] (0,0) ellipse (2cm and 1.0cm);
    \draw[very thick] (-0.8,0) .. controls (-0.6,0.35) and (0.6,0.35) .. (0.8,0); 
    \draw[very thick] (-1,0.15) .. controls (-0.7,-0.25) and (0.7,-0.25) .. (1,0.15); 

    \draw[very thick] (1.7,0.5) -- (4,1);
    \draw[very thick, -to] (1.7,0.5) -- (2.85,0.75);
        \node at (2.75, 1.15) {\small \color{bluestroke}$4$};
        \draw[very thick, bluestroke] (4,1) -- (5.0, 2);
        \node at (5.2,2.4) {\small $p_2$};

    \draw[very thick, bluestroke] (4,1) -- (6.3, 1.5);
    \draw[very thick, -to, bluestroke] (4,1) -- (5.15,1.25);
        \node at (5.1,1.6) {\small \color{bluestroke}$1$};
        \draw[very thick, bluestroke] (6.3, 1.5) -- (7.5, 2);
        \node at (8.0, 2.1) {\small $p_1$};
        \draw[very thick, bluestroke] (6.3, 1.5) -- (7.5, 1);
        \node at (8.0, 1.1) {\small $p_7$};
    \draw[very thick, color=bluestroke, fill=bluefill] (6.3, 1.5) circle (0.2);

    \draw[very thick, bluestroke] (4, 1) -- (6.3, 0);
    \draw[very thick, -to, bluestroke] (4,1) -- (5.15,0.5);
        \node at (5.1,0.1) {\small \color{bluestroke}$1$};
        \draw[very thick, bluestroke] (6.3, 0) -- (7.5, 0.5);
        \node at (8.0, 0.4) {\small $p_3$};
        \draw[very thick, bluestroke] (6.3, 0) -- (7.5, -0.5);
        \node at (8.0, -0.6) {\small $p_4$};
    \draw[very thick, color=bluestroke, fill=bluefill] (6.3, 0) circle (0.2);

    \draw[very thick, color=bluestroke, fill=bluefill] (4,1) circle (0.2);

    \draw[very thick] (1.7,-0.5) -- (4, -1);
    \draw[very thick, fill] (4,-1) circle (0.2);
        \draw[very thick] (4,-1) -- (5.2, -0.5);
        \node at (5.7, -0.4) {\small $p_5$};
        \draw[very thick] (4,-1) -- (5.2, -1.5);
        \node at (5.7, -1.4) {\small $p_6$};

    \end{tikzpicture}

}

\newcommand{\FirstLemma}{
    \begin{tikzpicture}[scale = 0.6]

    \begin{scope}[xshift = -15cm]
        \draw[very thick, redstroke] (1.7,0.5) -- (4,1);
    
        \draw[very thick, color=redstroke, fill=redfill] (0,0) ellipse (2cm and 1.0cm);
        \draw[very thick, color=redstroke, fill=white] (-0.8,0) .. controls (-0.6,0.35) and (0.6,0.35) .. (0.8,0); 
        \draw[very thick, color=redstroke, fill=white] (-0.8,0) .. controls (-0.6,-0.2) and (0.6,-0.2) .. (0.8,0);
        \draw[very thick, color=redstroke] (-1,0.15) .. controls (-0.7,-0.25) and (0.7,-0.25) .. (1,0.15); 
    
        \draw[very thick, fill] (4,1) circle (0.2);
            \draw[very thick] (4,1) -- (5.2, 1.5);
            \node at (5.7, 1.6) {\small $p_3$};
            \draw[very thick] (4,1) -- (5.2, 0.5);
            \node at (5.7, 0.4) {\small $p_1$};
    
        \draw[very thick,redstroke] (1.7,-0.5) -- (4, -1);
            \draw[very thick] (4,-1) -- (5.2,-1.25);
            \node at (5.6, -1.25) {\small $p_2$};
        \draw[very thick, fill] (4,-1) circle (0.2);
        
        \draw[thick, dashed] (4,2.1) -- (4,-2.1);
            \node at (4,-2.4) {\small $\rho_1$};

    \end{scope}

    \begin{scope}

        \draw[thick, dashed] (4,2.1) -- (4,-2.1);
            \node at (4,-2.4) {\small $\rho_1$};
        
        \draw[very thick] (1.7,0.5) -- (4,1);
        \draw[very thick, -to] (1.7,0.5) -- (5.7/2,0.75);
    
        \draw[very thick] (0,0) ellipse (2cm and 1.0cm);
        \draw[very thick, fill=white] (-0.8,0) .. controls (-0.6,0.35) and (0.6,0.35) .. (0.8,0); 
        \draw[very thick, fill=white] (-0.8,0) .. controls (-0.6,-0.2) and (0.6,-0.2) .. (0.8,0);
        \draw[very thick] (-1,0.15) .. controls (-0.7,-0.25) and (0.7,-0.25) .. (1,0.15); 
    
            \draw[very thick, bluestroke] (4,1) -- (5.2, 1.5);
            \node at (5.7, 1.6) {\small $p_3$};
            \draw[very thick, bluestroke] (4,1) -- (5.2, 0.5);
            \node at (5.7, 0.4) {\small $p_1$};
        \draw[very thick, color=bluestroke, fill=bluefill] (4,1) circle (0.2);
    
        \draw[very thick] (1.7,-0.5) -- (4, -1);
            \draw[very thick] (4,-1) -- (5.2,-1.25);
            \node at (5.6, -1.25) {\small $p_2$};

    \end{scope}
    
    \end{tikzpicture}
}

\newcommand{\SecondLemma}{
    \begin{tikzpicture}[scale = 0.5]

    \begin{scope}[xshift = -15cm]
    
        \draw[thick, dashed] (4,2.3) -- (4,-4.0);
            \node at (4,-4.3) {\small $\rho_1$};
        \draw[thick, dashed] (6.3,2.3) -- (6.3,-4.0);
            \node at (6.3,-4.3) {\small $\rho_2$};
    
        \draw[very thick, color=redstroke, fill=redfill] (0,0) ellipse (2cm and 1.0cm);
        \draw[very thick, color=redstroke, fill=white] (-0.8,-0.01) .. controls (-0.6,0.35) and (0.6,0.35) .. (0.8,-0.01); 
        \draw[very thick, color=redstroke, fill=white] (-0.8,0.01) .. controls (-0.6,-0.2) and (0.6,-0.2) .. (0.8,0.01);
        \draw[very thick, redstroke] (-1,0.15) .. controls (-0.7,-0.25) and (0.7,-0.25) .. (1,0.15); 
    
        \draw[very thick, redstroke] (1.7,0.5) -- (4,1);

        \draw[very thick, redstroke] (4,1) -- (6.3, 1.5);
            \draw[very thick] (6.3, 1.5) -- (7.5, 1.8);
            \node at (8.0, 1.8) {\small $p_1$};
        \draw[very thick, fill] (6.3, 1.5) circle (0.2);
    
        \draw[very thick, redstroke] (4, 1) -- (6.3, 0);
        \draw[very thick, fill] (6.3, 0) circle (0.2);
            \draw[very thick] (6.3, 0) -- (7.5, 0.5);
            \node at (8.0, 0.4) {\small $p_3$};
            \draw[very thick] (6.3, 0) -- (7.5, -0.5);
            \node at (8.0, -0.6) {\small $p_4$};
    
        \draw[very thick, color=redstroke, fill=redfill] (4,1) circle (0.2);
    
        \draw[very thick, redstroke] (1.7,-0.5) -- (6.3, -2);
        \draw[very thick, fill] (6.3,-2) circle (0.2);
            \draw[very thick] (6.3,-2) -- (7.5, -1.5);
            \node at (8, -1.4) {\small $p_5$};
            \draw[very thick] (6.3,-2) -- (7.5, -2.5);
            \node at (8, -2.5) {\small $p_6$};
    
        \draw[very thick, redstroke] (0.9,-0.9) -- (6.3,-3.05);
            \draw[very thick, fill] (6.3,-3.05) circle (0.2);
            \draw[very thick, fill] (6.3,-3.05) -- (7.5, -3.5);
            \node at (8, -3.5) {\small $p_2$};

        \draw[very thick, color=redstroke] (0,0) ellipse (2cm and 1.0cm);

    \end{scope}

    \begin{scope}

        \draw[thick, dashed] (4,2.3) -- (4,-4.0);
            \node at (4,-4.3) {\small $\rho_1$};
        \draw[thick, dashed] (6.3,2.3) -- (6.3,-4.0);
            \node at (6.3,-4.3) {\small $\rho_2$};
        
        \draw[very thick, color=redstroke, fill=redfill] (0,0) ellipse (2cm and 1.0cm);
        \draw[very thick, color=redstroke, fill=white] (-0.8,-0.01) .. controls (-0.6,0.35) and (0.6,0.35) .. (0.8,-0.01); 
        \draw[very thick, color=redstroke, fill=white] (-0.8,0.01) .. controls (-0.6,-0.2) and (0.6,-0.2) .. (0.8,0.01);
        \draw[very thick, redstroke] (-1,0.15) .. controls (-0.7,-0.25) and (0.7,-0.25) .. (1,0.15); 
    
        \draw[very thick, redstroke] (1.7,0.5) -- (4,1);
        \draw[very thick, fill] (4,1) circle (0.2);

        \draw[very thick] (4,1) -- (6.3, 1.5);
            \draw[very thick] (6.3, 1.5) -- (7.5, 1.8);
            \node at (8.0, 1.8) {\small $p_1$};
    
        \draw[very thick] (4, 1) -- (6.3, 0);
        \draw[very thick, -to] (4,1) -- (10.3/2,0.5);
    
        \draw[very thick, fill] (4,-1.25) circle (0.2);
            \draw[very thick, bluestroke] (6.3, 0) -- (7.5, 0.5);
            \node at (8.0, 0.4) {\small $p_3$};
            \draw[very thick, bluestroke] (6.3, 0) -- (7.5, -0.5);
            \node at (8.0, -0.6) {\small $p_4$};
        \draw[very thick, color=bluestroke, fill=bluefill] (6.3, 0) circle (0.2);

        \draw[very thick, redstroke] (1.7,-0.5) -- (4,-1.25);
        \draw[very thick] (4,-1.25) -- (6.3,-2);
        \draw[very thick, -to] (4,-1.25) -- (10.3/2,-3.25/2);
        
            \draw[very thick, bluestroke] (6.3,-2) -- (7.5, -1.5);
            \node at (8, -1.4) {\small $p_5$};
            \draw[very thick, bluestroke] (6.3,-2) -- (7.5, -2.5);
            \node at (8, -2.5) {\small $p_6$};
        \draw[very thick, color=bluestroke, fill=bluefill] (6.3,-2) circle (0.2);
    
        \draw[very thick, redstroke] (0.9,-0.9) -- (4,-2.1);
        \draw[very thick] (4,-2.1) -- (7.5,-3.5);
            \draw[very thick, fill] (4,-2.1) circle (0.2);
            \node at (8, -3.5) {\small $p_2$};

        \draw[very thick, color=redstroke] (0,0) ellipse (2cm and 1.0cm);
    
        \draw[very thick, fill] (4,-1.25) circle (0.2);
        
    \end{scope}

    \begin{scope}[xshift=-7.5cm, yshift = -8cm]

        \draw[thick, dashed] (4,2.3) -- (4,-4.0);
            \node at (4,-4.3) {\small $\rho_1$};
        \draw[thick, dashed] (6.3,2.3) -- (6.3,-4.0);
            \node at (6.3,-4.3) {\small $\rho_2$};
    
        \draw[very thick] (0,0) ellipse (2cm and 1.0cm);
        \draw[very thick] (-0.8,0) .. controls (-0.6,0.35) and (0.6,0.35) .. (0.8,0); 
        \draw[very thick] (-0.8,0) .. controls (-0.6,-0.2) and (0.6,-0.2) .. (0.8,0);
        \draw[very thick] (-1,0.15) .. controls (-0.7,-0.25) and (0.7,-0.25) .. (1,0.15); 
    
        \draw[very thick] (1.7,0.5) -- (4,1);
        \draw[very thick, fill] (4,1) circle (0.2);
    
        \draw[very thick] (4,1) -- (6.3, 1.5);
            \draw[very thick] (6.3, 1.5) -- (7.5, 1.8);
            \node at (8.0, 1.8) {\small $p_1$};
    
        \draw[very thick] (4, 1) -- (6.3, 0);
        
            \draw[very thick, bluestroke] (6.3, 0) -- (7.5, 0.5);
            \node at (8.0, 0.4) {\small $p_3$};
            \draw[very thick, bluestroke] (6.3, 0) -- (7.5, -0.5);
            \node at (8.0, -0.6) {\small $p_4$};
        \draw[very thick, color=bluestroke, fill=bluefill] (6.3, 0) circle (0.2);
    
        \draw[very thick] (1.7,-0.5) -- (6.3, -2);
        
            \draw[very thick, bluestroke] (6.3,-2) -- (7.5, -1.5);
            \node at (8, -1.4) {\small $p_5$};
            \draw[very thick, bluestroke] (6.3,-2) -- (7.5, -2.5);
            \node at (8, -2.5) {\small $p_6$};
        \draw[very thick, bluestroke, fill=bluefill] (6.3,-2) circle (0.2);
    
        \draw[very thick] (0.9,-0.9) -- (7.5, -3.5);
            \node at (8, -3.5) {\small $p_2$};
    
        \draw[very thick, -to] (4,1) -- (10.3/2,0.5);
        \draw[very thick, -to] (4,-1.25) -- (10.3/2,-3.25/2);

    \end{scope}

    \end{tikzpicture}
}

\newcommand{\ThirdLemma}{

    \begin{tikzpicture}[scale = 0.45]

    \begin{scope}

        \draw[very thick, -to] (-3.5,0) -- (-6,0);
        \draw[very thick, -to] (9,0) -- (11.5,0);
        
        \draw[thick, dashed] (4,2.3) -- (4,-4.0);
            \node at (4,-4.3) {\small $\rho_1$};
        \draw[thick, dashed] (6.3,2.3) -- (6.3,-4.0);
            \node at (6.3,-4.3) {\small $\rho_2$};
    
        \draw[very thick, color=redstroke, fill=redfill] (0,0) ellipse (2cm and 1.0cm);
        \draw[very thick, color=redstroke, fill=white] (-0.8,-0.01) .. controls (-0.6,0.35) and (0.6,0.35) .. (0.8,-0.01); 
        \draw[very thick, color=redstroke, fill=white] (-0.8,0.01) .. controls (-0.6,-0.2) and (0.6,-0.2) .. (0.8,0.01);
        \draw[very thick, redstroke] (-1,0.15) .. controls (-0.7,-0.25) and (0.7,-0.25) .. (1,0.15); 
    
        \draw[very thick, redstroke] (1.7,0.5) -- (4,1);
    
        \draw[very thick, redstroke] (4,1) -- (6.3, 1.5);
            \draw[very thick] (6.3, 1.5) -- (7.5, 1.8);
        \draw[very thick, fill] (6.3, 1.5) circle (0.2);
    
        \draw[very thick, redstroke] (4, 1) -- (6.3, 0);
        \draw[very thick, fill] (6.3, 0) circle (0.2);
            \draw[very thick] (6.3, 0) -- (7.5, 0.5);
            \draw[very thick] (6.3, 0) -- (7.5, -0.5);
    
        \draw[very thick, color=redstroke, fill=redfill] (4,1) circle (0.2);
    
        \draw[very thick, redstroke] (1.7,-0.5) -- (6.3, -2);
        \draw[very thick, fill] (6.3,-2) circle (0.2);
            \draw[very thick] (6.3,-2) -- (7.5, -1.5);
            \draw[very thick] (6.3,-2) -- (7.5, -2.5);
    
        \draw[very thick, redstroke] (0.9,-0.9) -- (6.3,-3.05);
            \draw[very thick, fill] (6.3,-3.05) circle (0.2);
            \draw[very thick, fill] (6.3,-3.05) -- (7.5, -3.5);
    
        \draw[very thick, color=redstroke] (0,0) ellipse (2cm and 1.0cm);

    \end{scope}

    \begin{scope}[xshift=-13cm]
    
        \draw[thick, dashed] (4,2.3) -- (4,-4.0);
            \node at (4,-4.3) {\small $\rho_1$};
        
        \draw[very thick, color=redstroke, fill=redfill] (0,0) ellipse (2cm and 1.0cm);
        \draw[very thick, color=redstroke, fill=white] (-0.8,-0.01) .. controls (-0.6,0.35) and (0.6,0.35) .. (0.8,-0.01); 
        \draw[very thick, color=redstroke, fill=white] (-0.8,0.01) .. controls (-0.6,-0.2) and (0.6,-0.2) .. (0.8,0.01);
        \draw[very thick, redstroke] (-1,0.15) .. controls (-0.7,-0.25) and (0.7,-0.25) .. (1,0.15); 
    
        \draw[very thick, redstroke] (1.7,0.5) -- (4,1);
        \draw[very thick, fill] (4,1) circle (0.2);
    
            \draw[very thick] (4, 1) -- (5.2, 1.75);
            \draw[very thick] (4, 1) -- (5.2, 0.25);
            \draw[very thick] (4, 1) -- (5.2, 1);
    
        \draw[very thick, fill] (4,-1.25) circle (0.2);
            \draw[very thick] (4,-1.25) -- (5.2, -0.75);
            \draw[very thick] (4,-1.25) -- (5.2, -1.75);
        
        \draw[very thick, redstroke] (1.7,-0.5) -- (4,-1.25);
    
        \draw[very thick, redstroke] (0.9,-0.9) -- (4,-2.1);
        \draw[very thick] (4,-2.1) -- (5.2, -2.6);
            \draw[very thick, fill] (4,-2.1) circle (0.2);
            
        \draw[very thick, color=redstroke] (0,0) ellipse (2cm and 1.0cm);
    
        \draw[very thick, fill] (4,-1.25) circle (0.2);

    \end{scope}

    \begin{scope}[xshift=15cm]

        \draw[thick, dashed] (4,2.3) -- (4,-4.0);
            \node at (4,-4.3) {\small $\rho_2$};
        
        \draw[very thick, color=redstroke, fill=redfill] (0,0) ellipse (2cm and 1.0cm);
        \draw[very thick, color=redstroke, fill=white] (-0.8,-0.01) .. controls (-0.6,0.35) and (0.6,0.35) .. (0.8,-0.01); 
        \draw[very thick, color=redstroke, fill=white] (-0.8,0.01) .. controls (-0.6,-0.2) and (0.6,-0.2) .. (0.8,0.01);
        \draw[very thick, redstroke] (-1,0.15) .. controls (-0.7,-0.25) and (0.7,-0.25) .. (1,0.15); 
    
        \draw[very thick, redstroke] (1.7,0.5) -- (4,1);
        \draw[very thick, fill] (4,1) circle (0.2);
            \draw[very thick] (4,1) -- (5.2, 1.3);
            
            \draw[very thick, redstroke] (2,0) -- (4,0.1);
            \draw[very thick] (4, 0.1) -- (5.2, 0.6);
            \draw[very thick] (4, 0.1) -- (5.2, -0.4);
        \draw[very thick, fill] (4, 0.1) circle (0.2);
    
        \draw[very thick, fill] (4,-1.25) circle (0.2);
            \draw[very thick] (4,-1.25) -- (5.2, -0.75);
            \draw[very thick] (4,-1.25) -- (5.2, -1.75);
        
        \draw[very thick, redstroke] (1.7,-0.5) -- (4,-1.25);
    
        \draw[very thick, redstroke] (0.9,-0.9) -- (4,-2.1);
        \draw[very thick] (4,-2.1) -- (5.2, -2.6);
            \draw[very thick, fill] (4,-2.1) circle (0.2);
            
        \draw[very thick, color=redstroke] (0,0) ellipse (2cm and 1.0cm);
    
        \draw[very thick, fill] (4,-1.25) circle (0.2);

    \end{scope}

    \begin{scope}[yshift=-8cm]

        \draw[very thick, -to] (-3.5,0) -- (-6,0);
        \draw[very thick, -to] (9,0) -- (11.5,0);
    
        \draw[thick, dashed] (4,2.3) -- (4,-4.0);
            \node at (4,-4.3) {\small $\rho_1$};
        \draw[thick, dashed] (6.3,2.3) -- (6.3,-4.0);
            \node at (6.3,-4.3) {\small $\rho_2$};
        
        \draw[very thick, color=redstroke, fill=redfill] (0,0) ellipse (2cm and 1.0cm);
        \draw[very thick, color=redstroke, fill=white] (-0.8,-0.01) .. controls (-0.6,0.35) and (0.6,0.35) .. (0.8,-0.01); 
        \draw[very thick, color=redstroke, fill=white] (-0.8,0.01) .. controls (-0.6,-0.2) and (0.6,-0.2) .. (0.8,0.01);
        \draw[very thick, redstroke] (-1,0.15) .. controls (-0.7,-0.25) and (0.7,-0.25) .. (1,0.15); 
    
        \draw[very thick, redstroke] (1.7,0.5) -- (4,1);
        \draw[very thick, fill] (4,1) circle (0.2);

        \draw[very thick] (4,1) -- (6.3, 1.5);
            \draw[very thick] (6.3, 1.5) -- (7.5, 1.8);
    
        \draw[very thick] (4, 1) -- (6.3, 0);
        \draw[very thick, -to] (4,1) -- (10.3/2, 0.5);
        
        \draw[very thick, fill] (4,-1.25) circle (0.2);
            \draw[very thick, bluestroke] (6.3, 0) -- (7.5, 0.5);
            \draw[very thick, bluestroke] (6.3, 0) -- (7.5, -0.5);
        \draw[very thick, color=bluestroke, fill=bluefill] (6.3, 0) circle (0.2);

        \draw[very thick, redstroke] (1.7,-0.5) -- (4,-1.25);
        \draw[very thick] (4,-1.25) -- (6.3,-2);
        \draw[very thick, -to] (4,-1.25) -- (10.3/2, -3.25/2);
        
            \draw[very thick, bluestroke] (6.3,-2) -- (7.5, -1.5);
            \draw[very thick, bluestroke] (6.3,-2) -- (7.5, -2.5);
        \draw[very thick, color=bluestroke, fill=bluefill] (6.3,-2) circle (0.2);
    
        \draw[very thick, redstroke] (0.9,-0.9) -- (4,-2.1);
        \draw[very thick] (4,-2.1) -- (7.5,-3.5);
            \draw[very thick, fill] (4,-2.1) circle (0.2);
            
        \draw[very thick, color=redstroke] (0,0) ellipse (2cm and 1.0cm);
    
        \draw[very thick, fill] (4,-1.25) circle (0.2);
        
    \end{scope}

    \begin{scope}[yshift=-8cm, xshift=-13cm]
    
        \draw[thick, dashed] (4,2.3) -- (4,-4.0);
            \node at (4,-4.3) {\small $\rho_1$};
        
        \draw[very thick, color=redstroke, fill=redfill] (0,0) ellipse (2cm and 1.0cm);
        \draw[very thick, color=redstroke, fill=white] (-0.8,-0.01) .. controls (-0.6,0.35) and (0.6,0.35) .. (0.8,-0.01); 
        \draw[very thick, color=redstroke, fill=white] (-0.8,0.01) .. controls (-0.6,-0.2) and (0.6,-0.2) .. (0.8,0.01);
        \draw[very thick, redstroke] (-1,0.15) .. controls (-0.7,-0.25) and (0.7,-0.25) .. (1,0.15); 
    
        \draw[very thick, redstroke] (1.7,0.5) -- (4,1);
        \draw[very thick, fill] (4,1) circle (0.2);
    
            \draw[very thick] (4, 1) -- (5.2, 1.75);
            \draw[very thick] (4, 1) -- (5.2, 0.25);
            \draw[very thick] (4, 1) -- (5.2, 1);
    
        \draw[very thick, fill] (4,-1.25) circle (0.2);
            \draw[very thick] (4,-1.25) -- (5.2, -0.75);
            \draw[very thick] (4,-1.25) -- (5.2, -1.75);
        
        \draw[very thick, redstroke] (1.7,-0.5) -- (4,-1.25);
    
        \draw[very thick, redstroke] (0.9,-0.9) -- (4,-2.1);
        \draw[very thick] (4,-2.1) -- (5.2, -2.6);
            \draw[very thick, fill] (4,-2.1) circle (0.2);
            
        \draw[very thick, color=redstroke] (0,0) ellipse (2cm and 1.0cm);
    
        \draw[very thick, fill] (4,-1.25) circle (0.2);
        
    \end{scope}

    \begin{scope}[yshift=-8cm, xshift=15cm]

        \draw[thick, dashed] (4,2.3) -- (4,-4.0);
            \node at (4,-4.3) {\small $\rho_2$};
        
        \draw[very thick] (0,0) ellipse (2cm and 1.0cm);
        \draw[very thick] (-0.8,-0.01) .. controls (-0.6,0.35) and (0.6,0.35) .. (0.8,-0.01); 
        \draw[very thick] (-0.8,0.01) .. controls (-0.6,-0.2) and (0.6,-0.2) .. (0.8,0.01);
        \draw[very thick] (-1,0.15) .. controls (-0.7,-0.25) and (0.7,-0.25) .. (1,0.15); 
    
        \draw[very thick] (1.7,0.5) -- (5.2,1.3);
            
            \draw[very thick] (2,0) -- (4,0.1);
            \draw[very thick, -to] (2,0) -- (3, 0.05);
            \draw[very thick, bluestroke] (4, 0.1) -- (5.2, 0.6);
            \draw[very thick, bluestroke] (4, 0.1) -- (5.2, -0.4);
        \draw[very thick, color=bluestroke, fill=bluefill] (4, 0.1) circle (0.2);
    
            \draw[very thick, bluestroke] (4,-1.25) -- (5.2, -0.75);
            \draw[very thick, bluestroke] (4,-1.25) -- (5.2, -1.75);
        
        \draw[very thick] (1.7,-0.5) -- (4,-1.25);
        \draw[very thick, -to] (1.7, -0.5) -- (5.7/2, -1.75/2);
    
        \draw[very thick] (0.9,-0.9) -- (4,-2.1);
        \draw[very thick] (4,-2.1) -- (5.2, -2.6);
            
        \draw[very thick] (0,0) ellipse (2cm and 1.0cm);
    
        \draw[very thick, color=bluestroke, fill=bluefill] (4,-1.25) circle (0.2);
        
    \end{scope}

    \begin{scope}[yshift=-16cm]

        \draw[very thick, -to] (-3.5,0) -- (-6,0);
        \draw[very thick, -to] (9,0) -- (11.5,0);
    
        \draw[thick, dashed] (4,2.3) -- (4,-4.0);
            \node at (4,-4.3) {\small $\rho_1$};
        \draw[thick, dashed] (6.3,2.3) -- (6.3,-4.0);
            \node at (6.3,-4.3) {\small $\rho_2$};
    
        \draw[very thick] (0,0) ellipse (2cm and 1.0cm);
        \draw[very thick] (-0.8,0) .. controls (-0.6,0.35) and (0.6,0.35) .. (0.8,0); 
        \draw[very thick] (-0.8,0) .. controls (-0.6,-0.2) and (0.6,-0.2) .. (0.8,0);
        \draw[very thick] (-1,0.15) .. controls (-0.7,-0.25) and (0.7,-0.25) .. (1,0.15); 
    
        \draw[very thick] (1.7,0.5) -- (4,1);
        \draw[very thick, fill] (4,1) circle (0.2);
    
        \draw[very thick] (4,1) -- (6.3, 1.5);
            \draw[very thick] (6.3, 1.5) -- (7.5, 1.8);
    
        \draw[very thick] (4, 1) -- (6.3, 0);
        \draw[very thick, -to] (4, 1) -- (10.3/2, 0.5);
        
            \draw[very thick, bluestroke] (6.3, 0) -- (7.5, 0.5);
            \draw[very thick, bluestroke] (6.3, 0) -- (7.5, -0.5);
        \draw[very thick, color=bluestroke, fill=bluefill] (6.3, 0) circle (0.2);
    
        \draw[very thick] (1.7,-0.5) -- (6.3, -2);
        \draw[very thick, -to] (1.7,-0.5) -- (16/3, -5/3);
            \draw[very thick, bluestroke] (6.3,-2) -- (7.5, -1.5);
            \draw[very thick, bluestroke] (6.3,-2) -- (7.5, -2.5);
        \draw[very thick, bluestroke, fill=bluefill] (6.3,-2) circle (0.2);
    
        \draw[very thick] (0.9,-0.9) -- (7.5, -3.5);

    \end{scope}

    \begin{scope}[yshift=-16cm, xshift=-13cm]

        \draw[thick, dashed] (4,2.3) -- (4,-4.0);
            \node at (4,-4.3) {\small $\rho_1$};
        
        \draw[very thick] (0,0) ellipse (2cm and 1.0cm);
        \draw[very thick] (-0.8,-0.01) .. controls (-0.6,0.35) and (0.6,0.35) .. (0.8,-0.01); 
        \draw[very thick] (-0.8,0.01) .. controls (-0.6,-0.2) and (0.6,-0.2) .. (0.8,0.01);
        \draw[very thick] (-1,0.15) .. controls (-0.7,-0.25) and (0.7,-0.25) .. (1,0.15); 
    
        \draw[very thick] (1.7,0.5) -- (4,1);
        \draw[very thick, fill] (4,1) circle (0.2);
    
            \draw[very thick] (4, 1) -- (5.2, 1.75);
            \draw[very thick] (4, 1) -- (5.2, 0.25);
            \draw[very thick] (4, 1) -- (5.2, 1);
    
        \draw[very thick] (1.7,-0.5) -- (4,-1.25);
        \draw[very thick, -to] (1.7,-0.5) -- (5.7/2, -1.75/2);
            \draw[very thick, bluestroke] (4,-1.25) -- (5.2, -0.75);
            \draw[very thick, bluestroke] (4,-1.25) -- (5.2, -1.75);
        \draw[very thick, color=bluestroke, fill=bluefill] (4,-1.25) circle (0.2);
    
        \draw[very thick] (0.9,-0.9) -- (4,-2.1);
        \draw[very thick] (4,-2.1) -- (5.2, -2.6);
            
        \draw[very thick] (0,0) ellipse (2cm and 1.0cm);

    \end{scope}

    \begin{scope}[yshift=-16cm, xshift=15cm]

        \draw[thick, dashed] (4,2.3) -- (4,-4.0);
            \node at (4,-4.3) {\small $\rho_2$};
        
        \draw[very thick] (0,0) ellipse (2cm and 1.0cm);
        \draw[very thick] (-0.8,-0.01) .. controls (-0.6,0.35) and (0.6,0.35) .. (0.8,-0.01); 
        \draw[very thick] (-0.8,0.01) .. controls (-0.6,-0.2) and (0.6,-0.2) .. (0.8,0.01);
        \draw[very thick] (-1,0.15) .. controls (-0.7,-0.25) and (0.7,-0.25) .. (1,0.15); 
    
        \draw[very thick] (1.7,0.5) -- (5.2,1.3);
            
            \draw[very thick] (2,0) -- (4,0.1);
            \draw[very thick, -to] (2,0) -- (3, 0.05);
            \draw[very thick, bluestroke] (4, 0.1) -- (5.2, 0.6);
            \draw[very thick, bluestroke] (4, 0.1) -- (5.2, -0.4);
        \draw[very thick, color=bluestroke, fill=bluefill] (4, 0.1) circle (0.2);
    
            \draw[very thick, bluestroke] (4,-1.25) -- (5.2, -0.75);
            \draw[very thick, bluestroke] (4,-1.25) -- (5.2, -1.75);
        
        \draw[very thick] (1.7,-0.5) -- (4,-1.25);
        \draw[very thick, -to] (1.7, -0.5) -- (5.7/2, -1.75/2);
    
        \draw[very thick] (0.9,-0.9) -- (4,-2.1);
        \draw[very thick] (4,-2.1) -- (5.2, -2.6);
            
        \draw[very thick] (0,0) ellipse (2cm and 1.0cm);
    
        \draw[very thick, color=bluestroke, fill=bluefill] (4,-1.25) circle (0.2);

    \end{scope}

    \end{tikzpicture}

}

\title[Compactifications with colliding markings]{On compactifications of $\moduli_{g,n}$ with colliding markings}
\author{Vance Blankers and Sebastian Bozlee}
\date{\today}

\begin{document}
\allowdisplaybreaks 

\maketitle

\begin{abstract}
  In this paper, we study all ways of constructing modular compactifications of the moduli space $\mathcal{M}_{g,n}$ of $n$-pointed smooth algebraic curves of genus $g$ by allowing markings to collide.
  We find that for any such compactification, collisions of markings are controlled by a simplicial complex which we call the collision complex. Conversely, we identify modular compactifications of $\mathcal{M}_{g,n}$ with essentially arbitrary collision complexes, including complexes not associated to any space of weighted pointed stable curves. These moduli spaces classify the modular compactifications of $\mathcal{M}_{g,n}$ by nodal curves with smooth markings as well as the modular compactifications of $\mathcal{M}_{1,n}$ with Gorenstein curves and smooth markings. These compactifications generalize previous constructions given by Hassett, Smyth, and Bozlee--Kuo--Neff.
\end{abstract}


\section{Introduction}\label{sec:intro}

A key insight of \cite{hassett_weighted_curves} is that alternatives to the Deligne-Mumford-Knudsen compactification of $\moduli_{g,n}$ can be constructed by allowing marked points to collide under prescribed conditions. In \cite{hassett_weighted_curves} weights $a_1,\ldots,a_n \in \Q$ are assigned to the markings $p_1, \ldots, p_n$. A subset $\{ p_i \}_{i \in I}$ of the markings is then allowed to coincide only if $\sum_{i \in I} a_i \leq 1$.

This paper is motivated by the observation that the weights are irrelevant to the definition of the moduli problem: all one really needs to decide are which subsets of markings are permitted to collide. These subsets form a simplicial complex. Our results make the case for the following thesis.

\bigskip

\noindent\textit{The ``correct" way to parametrize which smooth markings collide in any modular compactification of $\moduli_{g,n}$ is with a simplicial complex on $[n] = \{1, \ldots, n\}$.}

\bigskip

In particular, we show that for all modular compactifications $\moduli$ of $\mathcal{M}_{g,n}$, there is a simplicial complex $\calK(\moduli)$ on vertices $\{1, \ldots, n\}$ that describes when smooth markings may collide. Conversely, we prove that for all simplicial complexes $\calK$ on vertices $\{1,\ldots, n\}$ (barring some exceptional cases in genus 0), there is a modular compactification $\Moduli_{g,\calK}$ of $\moduli_{g,n}$ so that $\calK(\Moduli_{g,\calK}) = \calK$. These spaces generalize Hassett's spaces of weighted stable pointed curves and many are not associated to any choice of weights.
  
Combining these ideas with the recently constructed $Q$-stable spaces, we construct moduli spaces $\Moduli_{1,\calK}(Q)$ where a simplicial complex $\calK$ determines which markings may collide and a set of partitions $Q$ determines when elliptic $m$-fold singularities may appear. We prove the spaces $\Moduli_{1,\calK}(Q)$ classify the modular compactifications of $\moduli_{g,n}$ with Gorenstein singularities and smooth, possibly colliding markings in genus one.

\subsection{Relation to other work}

Our results most fundamentally deal with alternative compactifications of the moduli space of curves---alternative relative to the classical Deligne-Mumford-Knudsen compactification \cite{deligne_mumford, knudsen_mumford_projectivity_I, knudsen_projectivity_II, knudsen_projectivity_III}. For a review and discussion of many of the most important compactifications in the literature, the survey article \cite{fedorchuk_smyth} is an excellent and thorough resource.

The central touchstones underlying this work are \cite{hassett_weighted_curves}, \cite{smyth_zstable}, and \cite{bkn_qstable}. In the first, Hassett constructed a new collection of modular compactifications of $\moduli_{g,n}$ by allowing marks to collide as a step in the log minimal model program for $\moduli_{g,n}$. Along with illuminating some of the birational geometry of the moduli space of curves, his notion of weighted stable curves recovered constructions of Kapranov \cite{kap1,kap2}, Keel \cite{keel}, and Losev and Manin \cite{losev_manin} as special cases, shedding new light on existing theorems. The first half of this paper takes this idea of colliding points compactifications to its natural generality.

In \cite{smyth_zstable}, Smyth began the process of classifying all modular compactifications of $\moduli_{g,n}$, starting with ``stable" compactifications. These encompass the spaces developed by Hassett, as well as other ``popular" compactifications in the literature at the time, including compactification by pseudostable curves, first described in \cite{schubert_pseudostable}. One of Smyth's key tools is a type of combinatorial structure on the set of all stable graphs: extremal assignments. These structures---defined by a handful of simple properties---turn out to control all stable modular compactifications of $\moduli_{g,n}$. When the variation in the moduli problem consists solely of deciding whether or not a given collection of marked points are allowed to collide, extremal assignments become even more tractable: they are described entirely by simplicial complexes.

However, not all modular compactifications are stable. In \cite{smyth_mstable}, Smyth builds a sequence of ``semistable" modular compactifications in genus $1$ indexed by an integer $m$, where an $m$-stable curve is allowed to have elliptic Gorenstein singularities with $m$ or fewer branches. He first constructs compactifications by $m$-stable curves before showing that $m$-stability may be combined with Hassett stability in a compatible way. The second half of this paper generalizes this construction, combining the generalization of $m$-stability given in \cite{bkn_qstable} (called $Q$-stability) with a generalization of weighted stability given below (called simplicial stability).

Moving to the semistable setting, the authors of \cite{bkn_qstable} dig deeper into the classification of modular compactifications in genus $1$, themselves building on the work of \cite{smyth_mstable}. They find another combinatorial framework for classifying all modular compactifications of $\moduli_{1,n}$ by Gorenstein curves with distinct, smooth markings. Their techniques use contractions of universal curves from radially aligned curves, following the ideas of \cite{rsw} and \cite{keli_thesis}. A key technical component is a contraction construction compatible with base change, first worked out in \cite{bozlee_thesis} and \cite{bozlee_contractions}.

We are not the first to observe that simplicial complexes are a natural language in which to talk about colliding markings. To facilitate the discussion in \cite{alexeevguy}, Hassett stability is reframed in terms of a simplicial complex on the set of marked points. The authors note \cite[Remark 2.9]{alexeevguy} that asking for a complex $\calK$ to correspond to some Hassett stability condition is non-trivial and restrict themselves to this situation. The benefit of doing so is that they are able to lift some of the results from \cite{hassett_weighted_curves} more-or-less directly. In \cite{moon}, smooth compactifications in genus 0 are considered using an analysis of the structure of extremal assignments; these compactifications agree with our (genus 0) nodal compactifications. More recently, \cite{fry_thesis} uses a combinatorially dual notion to the simplicial complex framing in \cite{alexeevguy} to explore additional compactifications of $\moduli_{0,n}$ admitting collisions and their tropicalizations.

\subsection{Future directions}

In a subsequent paper we will compute the intersection theory of $\psi$-classes for the compactifications that appear in this paper as a natural extension of \cite{alexeevguy}, \cite{smyth_psi}, and \cite{bc_hassett}.

The classification results of this paper assume that the marked points are smooth. The situation becomes much more complicated even if one only allows markings to collide with nodes. Nevertheless this is a salient case: for example, the universal curves of Hassett's compactifications can be identified with moduli spaces of curves in which one point is allowed to collide with nodes (such marks are considered ``weight $0$" in \cite{hassett_weighted_curves}). A natural direction would therefore be to find a satisfying combinatorial description of the nodal modular compactifications of $\moduli_{g,n}$ where markings may collide both with nodes and each other. This would likely involve a careful study of the combinatorics of extremal assignments.

Our results on $(Q,\calK)$-stable spaces suggest the question of whether \emph{all} modular compactifications of $\moduli_{g,n}$ with distinct smooth markings admit simplicially stable variants. Battistella's spaces of Gorenstein modular compactifications in genus 2 \cite{battistella_mstable} are a reasonable next candidate to test this hypothesis. An important first step is to build a system for universal contractions like those in Section \ref{ssec:genus_one_univ_contractions} below, likely entailing a simultaneous generalization of the contraction techniques of \cite{battistella_carocci_stable_maps} and \cite{bozlee_contractions}.

Finally, it would be natural to consider simplicially stable variants of stable maps. The idea to use a simplicial complex to encode the collisions of marked points first arose in \cite{alexeevguy} in the context of constructing a variant of stable maps with weighted stable curves as sources, and a number of their results immediately generalize to the use of simplicially stable curves as sources. Extending even farther to $(Q,\calK)$-stable curves is significantly more involved.

\subsection*{Acknowledgments}
The authors would like to thank Luca Battistella, Francesca Carocci, Renzo Cavalieri, Andy Fry, Leo Herr, Navid Nabijou, Dhruv Ranganathan, and Jonathan Wise for helpful and interesting conversations, and David Smyth for suggesting a key step in Theorem \ref{thm:to_collide_or_not_to_collide}.

\section{Main definitions and theorem statements}

\begin{definition}
Let $n$ be a positive integer. A \emphbf{family of $n$-pointed curves} over a scheme $S$ consists of:
\begin{enumerate}
  \item A flat, proper, and finitely presented morphism $\pi : C \to S$ of schemes with connected, reduced, 1-dimensional geometric fibers;
  \item $n$ sections $p_1, \ldots, p_n : S \to C$ of $\pi$.
\end{enumerate}

An \emphbf{$n$-pointed curve} $(C; p_1,\ldots, p_n)$ is a family of $n$-pointed curves over the spectrum of an algebraically closed field. In this case, we identify the sections $p_1,\ldots, p_n$ with the respective closed points $p_1(S), \ldots, p_n(S)$.
\end{definition}

\begin{definition}[{{\cite[Definition 1.1]{smyth_zstable}}}]
  Let $\mathcal{U}_{g,n}$ be the algebraic stack over schemes whose $T$-points parametrize the families of $n$-pointed curves over $T$ of arithmetic genus $g$. Let $\mathcal{V}_{g,n}$ be the irreducible component of $\mathcal{U}_{g,n}$ containing $\moduli_{g,n}$.

  If $U$ is an open subscheme of $\Spec \Z$, a \emphbf{modular compactification over $U$} of $\mathcal{M}_{g,n}$ is an open substack $\mathcal{X} \subseteq \mathcal{V}_{g,n} \times_{\Spec \Z} U$ so that $\mathcal{X}$ is proper over $U$.
\end{definition}

\begin{definition}\label{def:modular_compactifications}
  A modular compactification $\moduli$ of $\moduli_{g,n}$ is said to
  \begin{enumerate}
    \item be \emphbf{nodal} if each geometric point of $\moduli$ is a pointed curve with at worst nodal singularities;
    \item be \emphbf{Gorenstein} if each geometric point of $\moduli$ is a pointed Gorenstein curve;
    \item have \emphbf{no collisions} if each geometric point of $\moduli$ is a curve with distinct, smooth marked points;
    \item admit \emphbf{collisions} if each geometric point of $\moduli$ is a curve with smooth, but not-necessarily distinct marked points;
    \item admit \emphbf{general collisions} if no hypotheses are imposed on markings.
  \end{enumerate}
\end{definition}

Unless qualified otherwise, nodal compactifications are taken to be modular compactifications over $\Spec \Z$, and not-necessarily nodal compactifications are taken to live over $\Spec \Z[1/6]$, in order to avoid problems with infinitesimal automorphisms of elliptic singularities in characteristic 2 and 3 (see \cite[\S 2.1]{smyth_mstable} for some analysis of this phenomenon).

\subsection{Simplicial complexes from moduli spaces}

We use the term ``simplicial complex" for what is sometimes called an ``abstract simplicial complex." That is:

\begin{definition}
  A \emphbf{simplicial complex on a set $S$} is a subset $\mathcal{K}$ of the powerset $2^S$ such that
  \begin{enumerate}
      \item $\calK$ is downward closed, i.e., if $I \subseteq J \subseteq S$ and $J \in \calK$, then $I \in \calK$; and
      \item for each $s \in S$, $\calK$ contains $\{ s \}$.
  \end{enumerate} 
\end{definition}

To obtain a simplicial complex from a modular compactification $\moduli$, we record the subsets of markings that collide somewhere in $\moduli$.

\begin{definition}
A \emphbf{subcurve} of a proper algebraic curve $C$ over an algebraically closed field is a reduced closed subscheme of $C$.

Let $(C; p_1,\ldots,p_n)$ be an $n$-pointed curve and let $x$ be a closed point of $C$. Write
\[
  \Marks(x) \coloneqq \{ i \in [n] \stc p_i = x \}
\]
for the set of marks at $x$. Similarly, if $Z$ is a subcurve of $C$, write
\[
  \Marks(Z) \coloneqq \{ i \in [n] \stc p_i \in Z \} = \bigcup _{x\in Z} \Marks(x)
\]
for the set of marks on $Z$.
\end{definition}

\begin{definition}
Let $U$ be an open subscheme of $\Spec \Z$ and $\moduli$ be a modular compactification of $\mathcal{M}_{g,n}$ over $U$. Say that a subset $I \subseteq [n]$ \emphbf{collides in $\moduli$} if there is a geometric point $(C; p_1,\ldots, p_n)$ of $\moduli$ and a smooth point $x \in C$ so that $I = \Marks(x)$. The \emphbf{collision complex} $\calK(\moduli)$ of $\moduli$ is the set of all subsets $I$ of $[n]$ so that $I$ collides in $\moduli$.
\end{definition}

One must check that $\calK(\moduli)$ is a simplicial complex, that is, that if $I \subseteq J \subseteq [n]$ and there exists some curve parametrized by a geometric point of $\moduli$ with a closed point marked by $J$, then there exists another curve and a point with set of marks indexed by $I$. This essentially follows from the deformation openness of $\moduli$: simply move away the points indexed by $J - I$ from the collision.

\begin{theorem} \label{thm:collision_complex_is_simplicial}
Let $U$ be an open subscheme of $\Spec \Z$ and $\moduli$ a modular compactification of $\moduli_{g,n}$ over $U$.
Then $\calK(\moduli)$ is a simplicial complex.
\end{theorem}

\begin{proof}
Suppose that $I \subsetneq J \subseteq [n]$ and $J \in \calK(\moduli)$.
Choose any pointed curve $(C; p_1, \ldots, p_n)$ parametrized by a geometric point in $\moduli$ so that $C$ has a smooth point $x$ with $J = \Marks(x)$. Let $C^{J-I} = \underbrace{C \times \cdots \times C}_{|J-I|\text{ copies}}$ and construct the constant family of curves
\begin{align*}
  \pi : C \times C^{J - I} &\to C^{J - I} \\
       (z, (y_j)_{j \in J - I}) &\mapsto (y_j)_{j \in J - I}.
\end{align*}
Construct $n$ sections $\sigma_1, \ldots, \sigma_n$ of $\pi$ by the rule
\[
  \sigma_i((y_j)_{j \in J - I}) = \begin{cases}
    (y_i, (y_j)_{j \in J - I}) \, \text{ if } i \in J - I \\
    (p_i, (y_j)_{j \in J - I}) \, \text{ if } i \not\in J - I.
  \end{cases}
\]
That is, the markings indexed by $J - I$ vary independently in the family, and all others remain constant.
Our original pointed curve $(C; p_1, \ldots, p_n)$ belongs to the family $\pi$. Since $\moduli$ is an open substack of $\mathcal{V}_{g,n}$, it follows $\moduli$ contains the generic member of $\pi$. In the generic member of $\pi$, $\Marks(x) = I$. Therefore $I \in \calK(\moduli)$, as desired.
\end{proof}

Before we narrow our focus to modular compactifications concerning any of the terms in Definition \ref{def:modular_compactifications}, we first establish that there is no finer invariant than collision complexes to describe when smooth marked points may collide in any modular compactification. In particular, the moduli spaces constructed below fully explore the possible ways of allowing smooth marks to collide. The data of a collision complex is still relevant for compactifications in which markings may collide at singularities---in this case, the complex merely describes the behavior of markings away from singularities. 

Semistable reduction implies that an alternative to permitting markings to collide at a smooth point is to sprout off a rational curve with the same markings.

\begin{definition}
If $(C; p_1, \ldots, p_n)$ is a pointed curve and $Z$ is a connected subcurve of $C$, we say \emphbf{$Z$ is a rational tail of $C$} if
\begin{enumerate}
  \item $Z$ has arithmetic genus zero;
  \item $Z \cap \ol{C - Z}$ is a single node of $C$.
\end{enumerate}

Let $U$ be an open subscheme of $\Spec \Z$ and $\moduli$ be a modular compactification of $\mathcal{M}_{g,n}$ over $U$. Say that a subset $I$ of $[n]$ \emphbf{does not collide in $\mathcal{M}_{g,n}$} if there is a pointed curve $(C; p_1,\ldots, p_n)$ of $\moduli$ and a rational tail $Z$ of $C$ so that $\Marks(Z) = I$.
Call such a subcurve $Z$ an \emphbf{$I$-marked rational tail of $C$}.
\end{definition}

\begin{theorem} \label{thm:to_collide_or_not_to_collide}
Let $\moduli$ be a modular compactification of $\moduli_{g,n}$ over an open subset $U$ of $\Spec \Z$. Let $I \subseteq [n]$ be a non-empty subset. Then either $I$ collides in $\moduli$ or $I$ does not collide in $\moduli$, but not both.
\end{theorem}

\begin{proof}
Suppose there exists an $I$ so that both $I$ collides in $\moduli$ and $I$ does not collide in $\moduli$. Then we can find a pointed curve $(C; p_1, \ldots, p_n)$
in $\moduli$ with a smooth point $x$ marked by $I$ and a pointed curve $(C'; p_1', \ldots, p_n')$ with an $I$-marked rational tail $T$. By definition of modular compactification, both curves are smoothable within $\moduli$. Replacing $C$ by a partial smoothing if necessary, we may assume that $C$ is smooth and the markings other than those indexed by $I$ are distinct. Write $q$ for the point in $C$ such that $\Marks(q) = I$. Similarly, smoothing $T \subset C'$ and $W := \ol{C' - T}$, we may assume that $C'$ consists of two smooth irreducible components $T$ and $W$ meeting at a node $q'$. Moreover, since $\moduli$ is open in $\mathcal{V}_{g,n}$ we may assume that the markings on $C'$ are smooth and distinct (see Figure \ref{fig:to_collide_or_not}).

Now both $(C; (p_j)_{j \in [n] - I}, q)$ and $(W; (p_j)_{j \in [n] - I}, q')$ are smooth curves with $n - |I| + 1$ distinct markings. Such curves are parametrized by $\moduli_{g,n - |I| + 1}$, which is irreducible. In particular, any open neighborhood of $C$ must intersect any open neighborhood of $W$ non-trivially. Then using openness of $\moduli$ in $\mathcal{V}_{g,n}$ to replace $W$ and $C$ if necessary, we may assume that $(C; (p_j)_{j \in [n] - I}, q)$ and $(W; (p_j)_{j \in [n] - I}, q')$ are isomorphic as $n - |I| + 1$ pointed curves.

Now choose a one parameter smoothing family $\pi' : \mathcal{C}' \to S$ over the spectrum of a discrete valuation ring with special fiber $(C'; p_1',\ldots, p_n')$. By \cite[Proposition 2.6]{smyth_zstable}, we may form a second flat family $\pi : \mathcal{C} \to S$ by contracting $T$ in the special fiber. Its limit is $(C; p_1,\ldots, p_n)$. Then both $(C; p_1,\ldots, p_n)$ and $(C'; p_1', \ldots, p_n')$ are limits in $\moduli$ of the same 1-parameter family of curves in $\moduli_{g,n}$. 
This contradicts the separatedness of $\moduli$, and the result follows.
\end{proof}

\begin{figure}
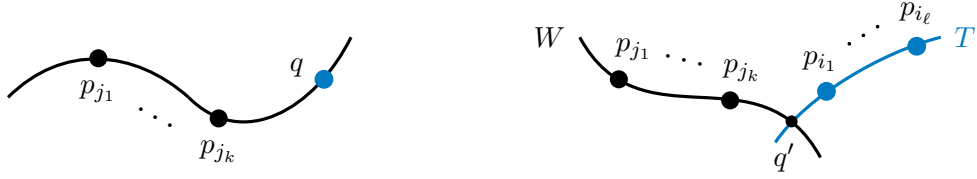

    \ToCollideOrNot
    \caption{The curves $C$ (left) and $C'$ (right).}
    \label{fig:to_collide_or_not}
\end{figure}

\begin{remark}
Let $\moduli, \moduli'$ be modular compactifications with the same collision complex. Suppose $C_{\moduli}$ and $C_{\moduli'}$ are the respective limits in $\moduli$ and $\moduli'$ of a family of curves in $\moduli_{1,n}$, and the points indexed by $I \subseteq [n]$ collide in $C_{\moduli}$. Theorem \ref{thm:to_collide_or_not_to_collide} implies that it is not possible for a rational tail of $C_{\moduli'}$ to contain only the marks indexed by $I$. However, it \emph{is} possible that the marks do not coincide in $C_{\moduli'}$: they may belong to a genus zero subcurve connected to the rest of $C_{\moduli'}$ by a singularity other than a node. See Example \ref{ex:different_collision_limits}.
\end{remark}

\subsection{Moduli spaces from simplicial complexes} \label{ssec:def_simplicial_spaces}

Does every simplicial complex $\calK$ on $[n]$ arise as $\calK(\moduli)$ for some modular compactification $\moduli_{g,n}$? Except in genus zero, the answer is yes.

\begin{definition}
Given a simplicial complex $\calK$ on $[n]$, say that a partition $\mathcal{P}$ of $[n]$ is a \emphbf{$\calK$-partition} if each part of $\mathcal{P}$ belongs to $\calK$.
A simplicial complex $\calK$ on $[n]$ is \emphbf{at least triparted} if each $\calK$-partition has at least three parts.

\end{definition}

\begin{definition}\label{def:simpstab2}
  Let $g \geq 0$ and $n \geq 1$ be integers, and let $\mathcal{K}$ be a simplicial complex on $[n]$.
  If $g = 0$ assume in addition that $\calK$ is at least triparted. We say an $n$-pointed curve $(C; p_1, \dots,p_n)$ is \emphbf{$\mathcal{K}$-stable} if
  \begin{enumerate}[label=\text{(K\arabic*)}]
    \item $C$ has at worst nodal singularities;
    \item $p_i \in C$ is smooth for every $i\in [n]$;
    \item for each smooth closed point $x \in C$, we have $\Marks(x) \in \mathcal{K}$;
    \item for each rational tail $Z$ of $C$, we have $\Marks(Z) \not\in \mathcal{K}$;
    \item (no infinitesimal automorphisms) the normalization of each rational irreducible component $Z$ of $C$ has at least three distinct special points.
  \end{enumerate}
  The substack of $\mathcal{V}_{g,n}$ parametrizing flat families of $\mathcal{K}$-stable $n$-pointed curves of genus $g$ is denoted $\Moduli_{g,\mathcal{K}}$.
\end{definition}

We remark that it is sufficient to check condition (K4) on the irreducible rational tails. We call the stacks $\Moduli_{g,\mathcal{K}}$ collectively \emphbf{moduli spaces of simplicially stable curves}, and we refer to them collectively as \emphbf{simplicially stable spaces}.

To explain the triparted hypothesis in genus zero, we give an alternative formulation of condition (K3) as follows.  Given a pointed curve $(C; p_1,\ldots, p_n)$, there is an associated partition of $[n]$ by the set of subsets $\Marks(x)$ as $x$ varies over the marked points of $C$. Then $(C; p_1,\ldots, p_n)$ satisfies (K3) if and only if the induced partition is a $\calK$-partition.

In light of this, a simplicial complex is at least triparted if and only if each $\calK$-stable curve has at least three \emph{distinct} marked points.  Intuitively, if $\mathcal{K}$ is not at least triparted, then $\Moduli_{0,\mathcal{K}}$ would ``want" to include a $\PP^1$ with only two distinct marked points. However, this curve has infinite automorphism group. Therefore it is necessary that $\calK$ is at least triparted in genus zero.

Our first main result is that the moduli stacks $\Moduli_{g,\calK}$ are in fact modular compactifications.

\begin{theorem_simplicial_spaces_are_modular_compactifications_intro}
  Let $g \geq 0, n \geq 1$ be integers and $\calK$ a simplicial complex on $[n]$. If $g = 0$ then assume in addition that $\calK$ is at least triparted. Then $\Moduli_{g,\calK}$ is a smooth Deligne-Mumford nodal modular compactification of $\moduli_{g,n}$ over $\Z$ admitting collisions such that $\calK(\Moduli_{g,\calK}) = \calK$.
\end{theorem_simplicial_spaces_are_modular_compactifications_intro}

Our next result is that simplicial spaces classify nodal compactifications with colliding points. They are therefore the fullest generalization of Hassett spaces from the perspective of altering compactifications of $\moduli_{g,n}$ by allowing marked points to selectively collide.

\begin{theorem_simplicial_spaces_classification_intro}
  If $\moduli$ is a nodal compactification of $\moduli_{g,n}$ over $\Z$ admitting collisions, then
  \[
    \moduli = \Moduli_{g,\calK(\moduli)}
  \]
  as open substacks of $\mathcal{V}_{g,n}$.
\end{theorem_simplicial_spaces_classification_intro}

Our approach to Theorem \ref{thm:simplicial_spaces_are_modular_compactifications_nonintro} and Theorem \ref{thm:simplicial_spaces_classification_nonintro} is to recognize that nodal compactifications with colliding points are \emphbf{stable} modular compactifications in the sense of \cite{smyth_zstable}. We can then use the combinatorics of \cite{smyth_zstable}'s extremal assignments. 

The local structure of the simplicially stable moduli spaces is not new: each is a regluing of Hassett spaces in the following sense.

\begin{theorem_simplicial_frankenstein_intro}
Let $g,n, \calK$ be as in Theorem \ref{thm:simplicial_spaces_are_modular_compactifications_nonintro}. Then there is a finite open cover $\{ U_i \}_{i \in I}$ of $\Moduli_{g,\calK}$ and weights $\{ \mathcal{A}_i \}_{i \in I}$ so that for each $i \in I$, $U_i$ is an open substack of the space of weighted stable curves $\Moduli_{g,\mathcal{A}_i}$.
\end{theorem_simplicial_frankenstein_intro}

It is from this fact and the known smoothness of weighted pointed stable spaces that we deduce the smoothness asserted in Theorem \ref{thm:simplicial_spaces_are_modular_compactifications_nonintro}.

It is often true in practice that when working with weighted pointed stable spaces, one first writes down the collisions one desires, then works backwards to find weights. If all one wants is a moduli space with the given collisions, then our work permits the latter step to be skipped. Moreover, it is not always possible to find the desired weights: not all simplicial complexes arise from a choice of weights \cite[Remark 2.9]{alexeevguy}. Table \ref{fig:simp_vs_hassett} illustrates the discrepancy by giving the number of distinct fine moduli spaces for low $n$ and $g \geq 1$.

\bigskip

\begin{table}[h] 
\begin{center}
\begin{tabular}{c|c|c}
     $n$ & \# of simplicially stable spaces & \# of weighted stable spaces \\ \hline
      2 & 2 & 2 \\
      3 & 9 & 9 \\
      4 & 114 & 96\\
      5 & 6,894 & 2,690 \\
      6 & 7,785,062 & 226,360\\
      7 & 2,414,627,396,434 & 64,646,855\\
      8 & 56,130,437,209,370,320,359,966 & 68,339,572,672 
\end{tabular}
\end{center}

\caption{The left column is Sequence A307249 of the OEIS \cite{oeis_simplicial_complexes}. The right column comes from \cite[Figure 1]{adgh-hassett}.}\label{fig:simp_vs_hassett}

\end{table}

We also construct reduction morphisms associated to each collision complex $\calK$, relating the Deligne-Mumford compactification $\Moduli_{g,n}$ to the simplicially stable compactifications.

\begin{theorem_simplicial_reduction_map_intro}
Let $g \geq 0, n \geq 1$ be integers and $\calK$ a simplicial complex on $[n]$. If $g = 0$, assume also that $\calK$ is at least triparted. Then there is a morphism of algebraic stacks
\[
  \rho_{\calK}: \Moduli_{g,n} \to \Moduli_{g,\mathcal{K}}.
\]
restricting to an isomorphism on the dense open substack $\moduli_{g,n}$.
\end{theorem_simplicial_reduction_map_intro}

\subsection{The genus one case}

The article \cite{bkn_qstable} classifies the Gorenstein compactifications of $\moduli_{1,n}$ with distinct markings. It is natural to ask if it is possible to extend this to a classification of Gorenstein compactifications in which markings may collide. On the other hand, it is also natural to wonder which collision complexes may be imposed on moduli spaces of curves admitting more exotic singularities than nodes.

Applying the philosophy of this paper to the definitions of \cite{bkn_qstable}, we are able to identify new moduli spaces $\Moduli_{1,\mathcal{K}}(Q)$ which exhaust the Gorenstein compactifications admitting collisions in genus one, and we extend the results of \cite{bkn_qstable} to these new spaces. Doing so takes up the second half of this paper.

\begin{definition}
Given a positive integer $n$, let $\PPart(n)$ be the set of partitions of $[n]$. Partially order $\PPart(n)$ by the rule that $\calP_1 \preceq \calP_2$ if the partition $\calP_2$ refines $\calP_1$.

Denote by $\mathfrak{Q}_n$ the collection of subsets $Q \subseteq \PPart(n)$ such that
\begin{enumerate}
    \item $Q$ is downward closed;
    \item $Q$ does not contain the discrete partition $\{ \{1\}, \ldots, \{n\} \}$. 
\end{enumerate}
\end{definition}

\begin{definition}
A closed point $p$ of an algebraic curve $C$ over an algebraically closed field is an \emphbf{elliptic Gorenstein singularity} if $\mathscr{O}_{C,p}$ is Gorenstein and $g(p) = 1.$
\end{definition}

We recall from \cite{smyth_mstable} that the elliptic Gorenstein singularities are classified by their number of branches $m$. If
$m = 1$, $p$ is a cusp; for $m = 2$, $p$ is a tacnode; for $m \geq 3$,
$p$ is the union of the coordinate axes of $\mathbb{A}^{m-1}$ and a line transverse to each of the coordinate hyperplanes of $\mathbb{A}^{m-1}$.

\begin{definition}[{{\cite[Definition 1.4]{bkn_qstable}}}]
Let $(C; p_1, \ldots, p_n)$ be an $n$-pointed curve of arithmetic genus one with smooth markings. Let $Z$ be a connected subcurve of $C$ of genus one and let $\Sigma$ be the divisor of markings. We define the \emphbf{level of $Z$}, $\lev(Z)$, to be the partition of $[n]$ induced by the connected components of $(C - Z) \cup \Sigma$.

If $q \in C$ is an elliptic Gorenstein singularity, we say the \emphbf{level of $q$}, $\lev(q)$ is the partition of $[n]$ induced by the connected components of the normalization of $C$ at $q$. (i.e. $i, j \in [n]$ belong to the same part of $\lev(q)$ if and only if the rational tails
containing $p_i, p_j$ are connected to the same branch of the singularity).
\end{definition}

\begin{definition} \label{def:do_not_overlap}
Given $Q \in \mathfrak{Q}_n$ and a simplicial complex $\calK$ on $[n]$, we say that $Q$ and $\calK$ \emphbf{do not overlap} if for each subset $I \in \calK$, $Q$ does not contain the partition in which $I$ is one part and the elements of $[n] - I$ belong to 1-element parts.
\end{definition}

\begin{definition} \label{def:qk_stable}
Let $Q \in \mathfrak{Q}_n$ and $\calK$ a simplicial complex on $[n]$ so that $Q$ and $\calK$ do not overlap.

An $n$-pointed Gorenstein curve $(C; p_1, \ldots, p_n)$ of genus one is \emphbf{$(Q, \calK)$-stable} if

\begin{enumerate}[label=\text{(Q\arabic*)}]
  \item $p_i \in C$ is smooth for every $i \in [n]$;
  \item for each elliptic Gorenstein singularity $p \in C$, we have $\lev(p) \in Q$;
  \item for each connected subcurve $E \subseteq C$ of genus one, $\lev(E) \not\in Q$;
  \item for each smooth point $x \in C$, we have $\Marks(x) \in \calK$;
  \item for each rational tail $Z$ of $C$, we have $\Marks(Z) \not\in \calK$;
  \item (no infinitesimal automorphisms) $H^0(C, \Omega_C^\vee(-\Sigma)) = 0$.
\end{enumerate}

A family of $n$-pointed curves of genus one is $(Q, \calK)$-stable if its geometric fibers are $(Q,\calK)$-stable. The \emphbf{moduli space of $(Q,\calK)$-stable curves} is the stack $\Moduli_{1,\calK}(Q)$ over $\Z[1/6]$ that assigns to a scheme $S$ the groupoid of $(Q,\calK)$-stable curves over $S$.
\end{definition}

The $(Q,\calK)$-stable spaces enjoy many of the same properties of the simplicially stable spaces, and classify compactifications of $\moduli_{1,n}$ in a parallel way.

\begin{theorem_qk_is_modular_compactification_intro}
Let $Q \in \mathfrak{Q}_n$ and let $\calK$ be a simplicial complex on $[n]$ so that $Q$ and $\calK$ do not overlap. Then $\Moduli_{1,\calK}(Q)$ is a Deligne-Mumford Gorenstein modular compactification of $\moduli_{1,n}$ over $\Z[1/6]$ admitting collisions so that $\calK(\Moduli_{1,\calK}(Q)) = \calK$.
\end{theorem_qk_is_modular_compactification_intro}

\begin{theorem_qk_classification_intro}
Let $\moduli$ be a Gorenstein modular compactification of $\moduli_{1,n}$ over $\Z[1/6]$ admitting collisions. Then there exists $Q \in \mathfrak{Q}_n$ and $\calK$ a simplicial complex on $[n]$ so that $Q$ and $\calK$ do not overlap and $\moduli = \Moduli_{1,\calK}(Q)$ as open substacks of $\mathcal{V}_{1,n} \times_{\Spec \Z} \Spec \Z[1/6]$.
\end{theorem_qk_classification_intro}

Finally, in Section \ref{ssec:genus_one_univ_contractions}, we find resolutions of the rational map between $\Moduli_{1,n}$ and $\Moduli_{1,\calK}(Q)$ for each $Q$ and $K$ by way of the moduli space of radially aligned curves $\Moduli_{1,n}^{rad}$ \cite[Definition 3.3.3]{rsw}, whose definition we recall in Section \ref{sec:tropical_log_background}.

\begin{theorem_contraction_to_qk_intro}
For each $Q \in \mathfrak{Q}_n$ and $\calK$ a simplicial complex on $[n]$ so that $Q$ and $\calK$ do not overlap, there is a diagram of birational morphisms of algebraic stacks
\[
  \begin{tikzcd}
   & \Moduli_{1,n}^{rad} \ar[dl] \ar[dr] & \\
   \Moduli_{1,n} & & \Moduli_{1,\mathcal{K}}(Q)
  \end{tikzcd}
\]
restricting to the identity on $\moduli_{1,n}$.
\end{theorem_contraction_to_qk_intro}

\section{Background on tropical and logarithmic curves} \label{sec:tropical_log_background}

We use the formalism of \cite{ccuw} for tropical curves and tropicalizations of log curves. The word ``log" here refers to the log structures of Kato and Illusie \cite{kato_log_structures}. The notion of a log curve is due to F. Kato in \cite{fkato_deformations}. We sketch a few key definitions and refer the reader to \cite{ccuw} and \cite{bkn_qstable} for further details. As is usual in logarithmic geometry, ``monoid" is taken to mean ``commutative monoid."

\begin{definition}
An ($n$-marked) \emphbf{tropical curve} $\Gamma$ of genus $g$ with edge lengths in a sharp fs monoid $P$ consists of the data of:
\begin{enumerate}
    \item a weighted ($n$-)marked graph $\ul{\Gamma}$ of genus $g$ with vertices $V(\Gamma)$, edges $E(\Gamma)$, legs $L(\Gamma)$, and a marking function $\{ 1, \ldots, n \} \overset{\sim}{\to} L(\Gamma)$;
    \item a length function $\delta : E(\Gamma) \to P$ so that $\delta(e) \neq 0$ for all $e \in E(G)$.
\end{enumerate}
\end{definition}

Morphisms of weighted $n$-marked graphs are taken to be compositions of weighted contractions of a sequence of edges and isomorphisms of weighted $n$-marked graphs.

\begin{definition}
If $\Gamma$ and $\Gamma'$ are tropical curves with edge lengths in $P$ and $P'$, respectively, then a \emphbf{morphism} of tropical curves $\pi : \Gamma' \to \Gamma$ consists of
\begin{enumerate}
  \item an ``underlying" morphism of weighted $n$-marked graphs $\ul{\pi}: \ul{\Gamma} \to \ul{\Gamma'}$ (note the reversal of arrow direction!) and
  \item a morphism of monoids $\pi^\sharp : P \to P'$
\end{enumerate}
such that for all $e \in E(\Gamma)$
\begin{enumerate}
  \item $\pi$ contracts $e$ if and only if $\pi^\sharp(e) = 0$ and
  \item if $\pi$ does not contract $e$, then $\pi^\sharp(\delta_{\Gamma}(e)) = \delta_{\Gamma'}(\pi(e))$.
\end{enumerate}
\end{definition}

We call morphisms of tropical curves and of their underlying graphs \emphbf{weighted edge contractions}. As in the definition, we use an underline to indicate a morphism of underlying graphs and the sharp symbol for the associated function on edge lengths.

\begin{definition}
A \emphbf{piecewise linear function} $f$ on a tropical curve $\Gamma$ with edge lengths in $P$ consists of
\begin{enumerate}
  \item a value $f(v) \in P$ for each vertex $v \in V(\Gamma)$;
  \item a slope $m(\ell) \in \mathbb{N}$ for each leg $\ell \in L(\Gamma)$
\end{enumerate}
such that if $e$ is an edge of $\Gamma$ with ends $v, w$, then $f(v) - f(w) \in P^{gp}$ is an integer multiple of the edge length $\delta(e)$.
We write $\PL(\Gamma)$ for the set of piecewise linear functions on $\Gamma$.
\end{definition}

A \emphbf{log curve} is a morphism $\pi : C \to S$ in the category of fs log schemes admitting the following \'etale local characterization, due originally to F. Kato but rephrased as in \cite[Theorem 2.3.1]{rsw}:

\begin{theorem} [{{\cite[Theorem 2.3.1]{rsw}}}]  \label{thm:log_curve_local_structure}
Let $\pi : C \to S$ be a family of log curves. If $x \in C$ is a geometric point with image $s \in S$, then there are \'etale neighborhoods $V$ of $x$ and $U$ of $s$ so that $V \to U$ has a strict morphism to an \'etale-local model
$V' \to U'$ where $V' \to U'$ is one of the following:
\begin{enumerate}
  \item (the smooth germ) $V' = \mathbb{A}^1_{U'} \to U'$ and the log structure on $V'$ is pulled back from the base;
  \item (the germ of a marked point) $V' = \mathbb{A}^1_{U'} \to U'$ with the log structure pulled back from the toric log structure on $\mathbb{A}^1$;
  \item (the node) $V' = \underline{\Spec} \mathscr{O}_{U'}[a,b] / (ab - t)$ for $t \in \mathscr{O}_{U'}$. The log structure on $V'$ is pulled back from the multiplication map $\mathbb{A}^2 \to \mathbb{A}^1$ of toric varieties along the morphism
  $U' \to \mathbb{A}^1$ of logarithmic schemes induced by $t$.
\end{enumerate}
\end{theorem}
If $x$ in $C|_s$ is a node, $x$ has a neighborhood of type (iii), and we say that the image of $t$ in $\ol{M}_{S,s}$ is the \emphbf{smoothing parameter} of $x$. We denote the smoothing parameter of $x$ by $\delta_x$.

The underlying morphism of schemes $\pi : C \to S$ is therefore a family of nodal curves. An \emphbf{$n$-marked log curve} is a proper log curve $\pi : C \to S$ together with $n$ disjoint sections $p_1,\ldots, p_n$ of $\pi$ through the ``marked" points of $C$, i.e., the points with neighborhood the germ of a marked point.

For each geometric point $s$ of $S$, we set the \emphbf{tropicalization} of $C|_s \to s$ to be the tropical curve whose underlying graph has
\begin{enumerate}
    \item a vertex of genus $g$ for each irreducible component of $C|_s$ with normalization of genus $g$;
    \item an edge of length $\delta_e$ for each node $e$ of $C|_s$;
    \item a leg for each marking of $C|_s$.
\end{enumerate}

Sections of the characteristic sheaf of $C$ may be interpreted as piecewise linear functions.

\begin{theorem} [{{\cite[Remark 7.3]{ccuw}}}] 
Let $\pi : C \to S$ be a log curve over the spectrum of an algebraically closed field. Then there is a bijection
\begin{align*}
  \PL : \Gamma(C, \ol{M}_C) &\overset{\sim}{\longrightarrow} \PL(\trop(C)) \\
     \sigma &\mapsto \PL(\sigma)
\end{align*}
where
\begin{enumerate}
  \item the value of $\PL(\sigma)$ at a vertex $v$ of $\Gamma(C)$ is the stalk of $\sigma$ at the generic point of the corresponding component of $C$;
  \item the slope of $\PL(\sigma)$ at a leg $l$ of $\Gamma(C)$ is the image of $\sigma$ in $(\ol{M}_C/\pi^{-1}\ol{M}_S)_{p} \cong \mathbb{N}$ where $p$ is the marked point corresponding to $l$.
\end{enumerate}
\end{theorem}

For a general log curve, this interpretation extends nicely over an \'etale neighborhood of each point. In particular, piecewise linear functions over the entire family turn out to be related by weighted edge contractions that we call face contractions.

\begin{definition}[{{\cite[Definition 3.5]{bkn_qstable}}}]
A weighted edge contraction $\pi : \Gamma' \to \Gamma$ where $\Gamma$ has edge lengths in $P$ is a \emph{face contraction} if there is a subset $S \subseteq P$ so that the map $\pi^\sharp$ is of the form
\[
  P \longrightarrow S^{-1}P \longrightarrow S^{-1}P / (S^{-1}P)^* \overset{\sim}{\longrightarrow} P'
\]
where the first arrow is localization, the second is the quotient by the submonoid of invertible elements, and the third is an isomorphism.
\end{definition}

In the case that $P$ is a finite free monoid $\N^r$, the face contractions are induced by the projections of $\N^r$ onto subsets of its coordinates.

\begin{theorem} [{{\cite[Theorem 3.10]{bkn_qstable}}}] \label{thm:uniform_charts}
Let $\pi : C \to S$ be an $n$-marked log curve and let $s$ be a geometric point of $S$. Then there is an \'etale neighborhood $U$ of $s$ in $S$ so that
\begin{enumerate}
  \item $\Gamma(U, \ol{M}_S) \to \ol{M}_{S,s}$ and $\Gamma(C|_U, \ol{M}_{C}) \to \Gamma(C|_s, \ol{M}_{C|_s})$  are isomorphisms;
  \item for each geometric point $t$ of $U$, there is a canonical face contraction
  \[
    \trop(C_t) \to \trop(C_s)
  \]
  induced by
  \[
    \ol{M}_{S,s} \overset{\sim}{\longleftarrow} \Gamma(U, \ol{M}_S) \longrightarrow \ol{M}_{S,t}.
  \]
  Moreover, this face contraction respects associated piecewise linear functions in the sense that
  \[
    \xymatrix{
      \Gamma(C|_s, \ol{M}_{C|_s}) \ar[d]_{\PL} & \Gamma(C|_U, \ol{M}_C) \ar[l]_{\sim} \ar[r] & \Gamma(C|_t, \ol{M}_{C|_t}) \ar[d]^{\PL} \\
      \PL(\trop(C|_s)) \ar[rr] & & \PL(\trop(C|_t))
    }
  \]
  commutes. 
\end{enumerate}
\end{theorem}

\begin{definition}
Say that $S$ is an \emphbf{atomic neighborhood} around a geometric point $s$ of $S$ for $\pi : C \to S$ if the conclusions of Theorem \ref{thm:uniform_charts} hold for $U = S$.
\end{definition}

A section of the characteristic sheaf of $C$ is therefore equivalent to a piecewise linear function on each geometric fiber of $C$ so that the resulting piecewise linear functions are compatible with generalization.

\smallskip

Sections of the characteristic sheaf of an fs log scheme induce line bundles with sections \cite{borne_vistoli}. In the case of curves we have the following characterization of these associated line bundles.

\begin{proposition} [{{\cite[Proposition 2.4.1]{rsw}}}] \label{prop:log_curve_line_bundles}
Let $\pi : C \to S$ be a log curve over $S$, where the underlying scheme of $S$ is the spectrum of an algebraically closed field. Let $\sigma$ be a global section of $\ol{M}_C$ with corresponding piecewise linear function $f$.
Let $v$ be a component of $C$. Then
\[
  \mathscr{O}_C(\sigma)|_{v} = \mathscr{O}_{v}\left( \sum_p \mu_p p \right) \otimes \pi^*\mathscr{O}_S(f(v))
\]
where the sum is over the edges and half-edges $p$ incident to $v$, and $\mu_p$ is the ``outgoing slope" of the piecewise linear function at the point $p$: when $p$ is a marked point, this is the integer $n_p$; when $p$ is a node joining $v$ to $w$,
this is $(f_w - f_v)/\delta_p$.
\end{proposition}

Given a tropical curve $\Gamma$, the sub-tropical curves of interest are given by vertex-induced subgraphs: graphs with edges and legs determined solely by a subset $V' \subseteq V(\Gamma)$. More precisely, in the notation of \cite[Section 3.1]{ccuw}:

\begin{definition}
Let $\Gamma$ be a tropical curve with edge lengths in $P$ and let $V' \subseteq V(\Gamma)$. Write $G$ for the underlying weighted marked graph of $\Gamma$. The \emphbf{tropical subcurve $\Gamma'$ determined by $V'$} is the weighted (but not marked) graph with
\begin{enumerate}
  \item vertices $V(\Gamma') = V'$;
  \item flags
  \[
    F(\Gamma') = \{ f \in F(G) \mid r_G(f) \in V' \text{ and } r_G(\iota_G(f)) \in V' \},
  \]
  i.e. $\Gamma'$ has edges the edges of $G$ both of whose ends lie in $V'$ and legs the legs of $G$ rooted in $V'$;
  \item $\iota_{\Gamma'} = \iota_{G}|_{F(G')}$;
  \item genus function $g_{\Gamma'} = g|_{V'}$.
\end{enumerate}
\end{definition}

\begin{definition}
Let $\Gamma$ be a tropical curve with edge lengths in $P$ and let $\sigma \in \PL(\Gamma)$. The \emphbf{support}
$|\sigma|$ of $\sigma$ is the tropical subcurve of $\Gamma$ whose vertices $v$ are those with $\sigma(v) > 0$.
\end{definition}

\begin{remark}
If $\sigma(v) > 0$, $\sigma(w) = 0$, and $e$ is an edge with ends in $v$ and $w$, our definition does not include $e$ in the support of $\sigma$. While this is slightly unnatural from the point of view of tropical curves with real edge lengths, our definition allows us to avoid introducing language for legs of finite length and is sufficient for our purposes.
\end{remark}


\begin{definition}
%
The \emphbf{core} of $C$ is the minimal connected subcurve of $C$ with the same arithmetic genus as $C$. Analogously the core of a tropical curve $\Gamma$ is the minimal connected vertex-induced subgraph of the same genus as $\Gamma$.
\end{definition}

\begin{definition}[{{\cite[Section 3.3]{rsw}}}]
We order the elements of a sharp fs monoid $P$ by the rule $p \leq q$ if and only if there is an element $p' \in P$ so that $p + p' = q$.

Given a tropical curve $\Gamma$ of genus one, we define a piecewise linear function $\lambda$ on $\Gamma$ measuring ``distance from the core'' as follows. If $v$ is a vertex in the core of $\Gamma$, we set
\[
  \lambda(v) = 0.
\]
If $v$ is a vertex outside of the core of $\Gamma$, we let $W = v_0e_1v_1e_2 \cdots e_kv_k$ be the unique path from the core of $\Gamma$ to $v$ and set
\[
  \lambda(v) = \sum_{i = 1}^k \delta(e_i).
\]
Finally, we set the slope of $\lambda$ to be 1 at all marked points.

As this is compatible with generalization, for any stable log curve $(\pi : C \to S; \sigma_1, \ldots, \sigma_n)$ of genus one, we let $\lambda \in \Gamma(S, \ol{M}_S)$ be the unique section of the characteristic bundle whose
restriction to geometric fibers has corresponding piecewise linear function as in the previous paragraph.
\end{definition}


\begin{definition}[{{\cite[Definition 3.3.3]{rsw}}}]\label{def:basic_radially_aligned}
A stable $n$-marked tropical curve of genus one with edge lengths in $P$ is \emphbf{radially aligned} if, for each pair of vertices $v, w$ of $\Gamma$, $\lambda(v)$ is comparable to $\lambda(w)$ in $P$.

Given such a radially aligned curve, let
\[
  0 < \rho_1 < \cdots < \rho_k
\]
be the distinct values of $\lambda(v)$ as $v$ varies over the components of $C$, and let $\delta_1, \ldots, \delta_l$ be the lengths of the edges of $\trop(C)$ internal to the core of $\Gamma$. Let $e_1 = \rho_1, e_2 = \rho_2 - \rho_1, \ldots, e_k = \rho_k - \rho_{k - 1}$. If $P$ is freely generated by
\[
  \{ e_1, \ldots, e_k \} \cup \{ \delta_1, \ldots, \delta_l \},
\]
then we say that $\Gamma$ is a \emphbf{basic radially aligned tropical curve}. An element of $P$ is said to \emphbf{have no contribution from the core} if it lies in the submonoid generated by $e_1, \ldots, e_k$.

A stable log curve $(\pi : C \to S; p_1, \ldots, p_n)$ of genus one with $n$ markings is \emphbf{radially aligned} or \emphbf{has a basic radially aligned log structure} if the tropicalizations of its geometric fibers
with their pulled back log structure are respectively radially aligned or basic radially aligned. An element $\rho \in \Gamma(S, \ol{M}_{S})$ \emphbf{has no contribution from the core} if the same holds of its
stalks at the geometric points of $S$.
\end{definition}

There is a moduli stack with log structure parametrizing radially aligned log curves.

\begin{theorem} [{{\cite[Proposition 3.3.4]{rsw}}}] \label{thm:M_1n_rad_exists} \hfill
\begin{enumerate}
\item There is a Deligne-Mumford stack with locally free log structure $\ol{\mathcal{M}}_{1,n}^{rad}$ whose $S$-points for $S$ an fs log scheme are the $n$-marked radially aligned curves $\pi : C \to S$ over $S$.  We say its log structure is \emph{the} basic radially aligned log structure.

\item There is a natural map $\ol{\mathcal{M}}_{1,n}^{rad} \to \ol{\mathcal{M}}_{1,n}$ induced by a logarithmic blowup and it restricts to an isomorphism on $\mathcal{M}_{1,n}$.
\end{enumerate}
\end{theorem}

A stable log curve $(\pi : C \to S; \sigma_1, \ldots, \sigma_n)$ has a basic radially aligned
log structure precisely when the log structure on $\pi : C \to S$ is that pulled back from the
the universal curve $\mathcal{C}_{1,n}^{rad} \to \ol{\mathcal{M}}_{1,n}^{rad}$ along the map
$S \to \ol{\mathcal{M}}_{1,n}^{rad}$.

\section{Simplicially stable spaces}\label{sec:collide}

The main goals of this section are to prove that simplicially stable spaces are nodal modular compactifications with collisions and that they are the \emph{only} nodal modular compactifications with collisions. We begin with a review of Smyth's notion of $\calZ$-stability.

\subsection{$\calZ$-stability}

We show that nodal compactifications admitting collisions are stable compactifications in the sense of \cite{smyth_zstable}, using the results therein as a key tool. We begin by recalling the pertinent definitions and results.

\begin{definition}
A pointed curve $(C; p_1,\ldots, p_n)$ is \emphbf{prestable} if each rational component of its normalization has at least three distinguished points and each elliptic component of its normalization has at least one distinguished point.

A modular compactification $\moduli$ of $\moduli_{g,n}$ over $\Z$ is \emphbf{stable} if each geometric point $(C; p_1,\ldots, p_n)$ of $\moduli$ is prestable.
\end{definition}

The main result of \cite{smyth_zstable} (stated in full as Theorem \ref{thm:smythclassification} below) says that stable modular compactifications are uniquely determined by a collection of combinatorial data called an \emphbf{extremal assignment}. We restate the definition of extremal assignment here in the language of edge contractions (the ``morphisms of graphs" in \cite[Section 3.1]{ccuw}).

\begin{definition}[{\cite[Definition 1.5]{smyth_zstable}}]\label{def:extremalassignment}
Let $\mathbb{G}_{g,n}$ be the set of isomorphism classes of the dual graphs of $n$-pointed stable curves of genus $g$ up to isomorphism.
For each $G \in \mathbb{G}_{g,n}$ assign a vertex-induced subgraph $\calZ(G)$ of $G$. We say that $\calZ$ is an \emphbf{extremal assignment (on $\mathbb{G}_{g,n}$)} if it satisfies the following two axioms:
\begin{enumerate}[label=\text{(Z\arabic*)}]
    \item for each $G \in \mathbb{G}_{g,n}$, $\calZ(G) \neq G$;
    \item for each weighted edge contraction $\pi: G' \to G$ between graphs in $\mathbb{G}_{g,n}$ and each vertex $v \in V(G)$, we have
    \[
        v \subseteq \calZ(G) \iff \pi^{-1}(v) \subseteq \calZ(G').
    \]
\end{enumerate}
\end{definition}

Note that if axiom (Z2) holds, it suffices to check axiom (Z1) for $G$ equal to the graph with a single genus $g$ vertex.

\smallskip

An extremal assignment $\calZ$ determines a subcurve $\calZ(C) \subset C$ for each stable curve $(C;p_1,\dots,p_n)$ by taking $\calZ(C)$ to be the union of all irreducible components of $C$ whose corresponding vertices in the dual graph $G_C$ of $C$ belong to $\calZ(G_C)$ \cite[Remark 1.6]{smyth_zstable}. In the following definition, if $C$ is some curve, $g_a(Z)$ denotes the arithmetic genus of a connected component $Z\subseteq C$, $g_s(p)$ denotes the genus of a point $p\in C$, and $m(p)$ denotes the number of branches in the normalization of $\tilde{C} \to C$ above $p$ (see \cite[Definition 1.7]{smyth_zstable}).

\begin{definition}
Given a vertex-induced subgraph of $Z$ of a weighted $n$-marked graph $G$, the \emph{edge valence} of $Z$ is the number of edges connecting $Z$ to the complement of $Z$ in $G$.
\end{definition}

\begin{definition}[{\cite[Definition 1.8]{smyth_zstable}}]\label{def:zstability}
Given an extremal assignment $\calZ$ on $\mathbb{G}_{g,n}$, a smoothable $n$-pointed curve $(C;p_1,\dots,p_n)$ is \emphbf{$\calZ$-stable} if there exists a stable curve $(C^s,p_1^s,\dots,p_n^s)$ and a morphism $\phi:(C^s,p_1^s,\dots,p_n^s) \to (C;p_1,\dots,p_n)$ satisfying the following axioms:
\begin{enumerate}
\item $\phi$ is surjective with connected fibers;
\item $\phi$ maps $C^s - \calZ(C^s)$ isomorphically onto its image;
\item if $Z_1,\dots,Z_k$ are the connected components of $\calZ(C^s)$, then $p_i := \phi(Z_i)\in C$ satisfies $g_s(p_i) = g_a(Z_i)$ and $m(p_i)$ is the edge valence of $Z_i$.
\end{enumerate}
For any extremal assignment $\calZ$, define $\Moduli_{g,n}(\calZ)$ to be the substack of $\mathcal{V}_{g,n}$ parametrizing flat families of $\calZ$-stable curves.
\end{definition}

The power of these definitions comes from the following result.

\begin{theorem}[{\cite[Theorem 1.9]{smyth_zstable}}]\label{thm:smythclassification} {}\hfill
\begin{enumerate}
  \item If $\calZ$ is an extremal assignment, then $\Moduli_{g,n}(\calZ)$ is a stable modular compactification of $\moduli_{g,n}$.\
    In particular, it is an irreducible stack, proper over $\Spec \mathbb{Z}$, and has quasi-finite diagonal.
  \item If $\moduli$ is a stable modular compactification, then $\moduli = \Moduli_{g,n}(\calZ)$ for some extremal assignment $\calZ$.
\end{enumerate}
\end{theorem}

\subsection{Construction and classification of simplicial stable spaces} \label{sec:simplicial_spaces}
Given a graph $G\in\mathbb{G}_{g,n}$, a connected subgraph $T \subset G$ is a \emphbf{tail} if $T$ has edge-valence $1$ and a \emphbf{rational tail} if in addition its genus is $0$. The unique edge $e_T$ connecting $T$ to the rest of $G$ is called the \emphbf{leading edge of $T$}.

Extremal assignments behave well with respect to tails in general.

\begin{lemma}\label{lem:tailspersist1}
  Let $G \in \mathbb{G}_{g,n}$, and let $T$ be a tail of $G$. Suppose $\calZ$ is an extremal assignment, with $T$ a subgraph of $\calZ(G)$. Suppose $G' \in \mathbb{G}_{g,n}$ has a tail $T'$ with $\Marks(T) = \Marks(T')$. Then $T'$ is a subgraph of $\calZ(G')$.
\end{lemma}
\begin{proof}
  Let $\hat{G}$ be the stable graph obtained by contracting all edges of $G$ except $e_T$. Note that this is isomorphic to the graph obtained by contracting all edges of $G'$ except $e_{T'}$. Write $\pi: G \to \hat{G}$ and $\pi' : G' \to \hat{G}$ for the associated edge contractions. By axiom (Z2) applied to $\pi$, $\pi(T) \subseteq \calZ(\hat{G})$. Then, since $\pi'(T') = \pi(T) \subseteq \calZ(\hat{G})$, axiom (Z2) implies $T' \subseteq (\pi')^{-1}(\pi(T')) \subseteq \calZ(\hat{G})$.
\end{proof}

\begin{remark}
This is essentially a special case of Theorem \ref{thm:to_collide_or_not_to_collide} for stable compactifications.
\end{remark}

\begin{definition}
An extremal assignment $\calZ$ is \emphbf{supported on rational tails} if for each $G \in \mathbb{G}_{g,n}$, $\calZ(G)$ is a disjoint union of rational tails.
\end{definition}

\begin{lemma} \label{lem:nodal_implies_stable}
Let $g, n$ be integers with $g \geq 0$, $n \geq 1$ and $n \geq 3$ if $g = 0$.
If $\moduli$ is a nodal modular compactification of $\moduli_{g,n}$ admitting collisions over $\Spec \Z$, then $\moduli \cong \Moduli_{g,n}(\calZ)$ for some extremal assignment $\calZ$ supported on rational tails.
\end{lemma}

\begin{proof}
Let $(C; p_1, \ldots, p_n)$ be a geometric point of $\moduli$, so that $C$ is at-worst nodal with smooth markings. As in \cite[Remark 1.3]{smyth_zstable}, since $\moduli$ is proper, $C$ cannot contain any rational curve with fewer than three special points. As $C$ is connected and $n \geq 1$, $C$ cannot be an unmarked curve of genus one. Therefore $C$ is prestable, and Theorem \ref{thm:smythclassification} implies there is an extremal assignment $\mathcal{Z}$ so that $\moduli = \Moduli_{g,n}(\mathcal{Z})$.

It remains to show that for all graphs $G \in \mathbb{G}_{g,n}$, $\mathcal{Z}(G)$ is a disjoint union of rational tails. If for some $G$, $\mathcal{Z}(G)$ has a connected component of positive genus, then $\Moduli_{g,n}(\calZ)$ contains curves with singularities of positive genus. But this does not happen in the nodal compactification $\moduli = \Moduli_{g,n}(\calZ)$, so connected components of $\mathcal{Z}(G)$ must have genus zero. Similarly, if for some $G$, $\calZ(G)$ has a connected component with edge valence greater than two, then $\Moduli_{g,n}(\mathcal{Z})$ contains a curve with a singularity with more than two branches. Again, no such curves appear in $\moduli$, so $\calZ$ must have connected components with edge valence no larger than two. If for some $G$, $\calZ(G)$ possesses a connected component $Z$ of genus zero with edge valence exactly two, then stability requires that $Z$ possesses at least one marking. This implies that $\Moduli_{g,n}(\calZ)$ contains a curve with a non-smooth marking, a contradiction. We conclude $\calZ$ is supported on rational tails. 
\end{proof}

Note that there are no graphs $G \in \mathbb{G}_{0,n}$ containing a tail $T$ with $|\Marks(T)| \in \{n - 1, n\}$. With this and Lemma \ref{lem:tailspersist1} in mind, an extremal assignment supported on rational tails must take the following form.

\begin{definition}\label{def:simpstab}
Suppose $\calK$ is a subset of the powerset $2^{[n]}$ containing $\varnothing,\{1\},\dots,\{n\}$. If $g = 0$, assume in addition that $\calK$ does not contain any subsets of size $n - 1$ or $n$. Denote by $\calZ_\calK$ the assignment that sends $G \in \mathbb{G}_{g,n}$ to the smallest vertex-induced subgraph of $G$ containing all rational tails $T$ with $\Marks(T) \in \calK$.

Conversely, if $\calZ$ an extremal assignment on $\mathbb{G}_{g,n}$ supported on rational tails, denote by $\calK_\calZ$ the subset of $2^{[n]}$ that contains $\varnothing,\{1\},\dots,\{n\}$ and all $I \subseteq [n]$ such that there exists a $G \in \mathbb{G}_{g,n}$ containing an $I$-marked rational tail.
\end{definition}

\begin{lemma}\label{lem:extremal_K_is_simplicial}
With notation as above, if $\calZ_\calK$ is an extremal assignment, then $\calK$ is a simplicial complex.
\end{lemma}

\begin{proof}
Suppose $I \subseteq J \subseteq [n]$, $J \in \calK$, and $|I| \geq 2$. Let $G_I$ be the graph with a genus $g$ vertex marked by $[n] - I$ attached by a single edge to a genus 0 vertex marked by $I$.
Similarly, let $G_{J}$ be the graph with a genus $g$ vertex marked by $[n] - J$ attached by a single edge to a genus 0 vertex marked by $J$. Finally, let $G_{I \subseteq J}$ be the graph with a genus $g$ vertex marked by $[n] - J$ attached to a genus 0 vertex marked by $J - I$ attached in turn to a genus 0 vertex marked by $I$. Let $\pi_I : G_{I \subseteq J} \to G_I$ and $\pi_J : G_{I \subseteq J} \to G_J$ be the obvious weighted contractions (see Figure \ref{fig:ExtremalIsSimplicial}).
\begin{figure}[h]
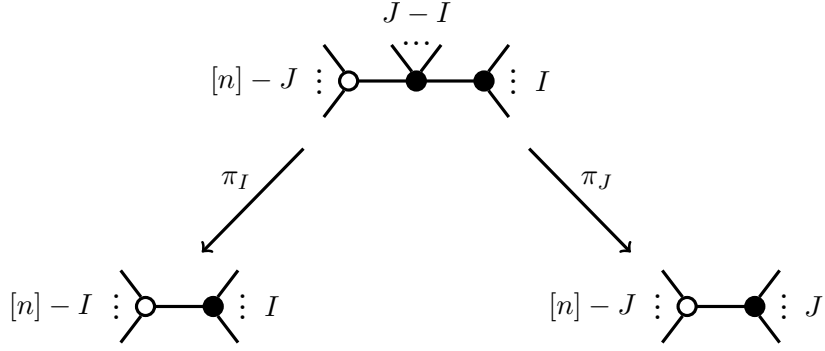
 
    \ExtremalIsSimplicial
    \caption{The graph $G_{I \subseteq J}$ (top) with weighted contractions $\pi_I$ to $G_I$ (bottom left) and $\pi_J$ to $G_J$ (bottom right). White vertices have genus $g$, and black vertices have genus $0$.}\label{fig:ExtremalIsSimplicial}
\end{figure}

Then axiom (Z2) applied to $\pi_J$ implies that $\calZ(G_{I \subseteq J})$ consists of both genus 0 vertices. By axiom (Z2) again applied to $\pi_I$, $\calZ(G_I)$ consists of the vertex marked by $I$. Then $I$ must belong to $\calK$.
\end{proof}

\begin{remark} \label{rem:extremal_K_is_simplicial}
This is essentially Theorem \ref{thm:collision_complex_is_simplicial} applied to stable compactifications after translation into the language of extremal assignments.
\end{remark}

\begin{lemma}\label{lem:intersecting_tails}
    Let $G \in \mathbb{G}_{g,n}$, and let $T_1,T_2$ be tails of $G$. If $T_1 \cap T_2 \neq \varnothing$, then either one of the tails contains the other tail, or $T_1 \cup T_2 = G$.
\end{lemma}
\begin{proof}
    Suppose neither tail contains the other. Let $e_i$ be the leading edge of $T_i$. Since $T_2 \not\subseteq T_1$ and $T_2$ is connected, there exists some $e \in E(T_2)$ one of whose ends lies in $T_1$ and whose other end lies in $T_2$ but not $T_1$. As $T_1$ has edge valence 1, this edge $e$ must be $e_1$. By symmetry, $e_1, e_2 \in T_1 \cup T_2$.
    
    Let $G_i,G_i' \subset G$ be the subgraphs induced by deleting $e_i$, with the convention that $T_i \subseteq G_i$. Note that $V(G_i) \cup V(G_i') = V(G)$. Since $T_i$ has valence $1$ and $G_i$ is connected, $T_i = G_i$. On the other hand, $T_j$ intersects $G_i'$ non-trivially, $e_j$ belongs to $G_i$, and $T_j$ has valence 1, we must have $G_i' \subseteq T_j$. Therefore $V(G) = V(G_i) \cup V(G_i') = V(T_1) \cup V(T_2)$, and as $e_1,e_2 \in T_1\cup T_2$, we have $T_1 \cup T_2 = G$.
    %
\end{proof}

\begin{lemma} \label{lem:subset_of_tails}
Let $G \in \mathbb{G}_{g,n}$. Let $T_1, \ldots, T_k$ be arbitrary rational tails of $G$. If $H$ is a connected vertex-induced subgraph of $G$ so that $V(H) \subseteq \bigcup_{i = 1}^k V(T_i)$, then either $H \subseteq T_i$ for some $i$ or $V(G) = V(T_\ell) \cup V(T_r)$ for some $\ell, r$.
\end{lemma}
\begin{proof}
Suppose $V(H) \subseteq \bigcup_{i=1}^k V(T_i)$ but $H \not\subseteq T_i$ for any $i$. 

If $E(H) \subseteq \bigcup_{i=1}^k E(T_i)$, then since $H \not\subset T_i$ for any $i$, there exists an edge $e \in E(T_i) \cap E(T_j)$ for some $T_i$ and $T_j$ neither of which contains the other. So $T_i \cap T_j \neq \emptyset$, and by Lemma \ref{lem:intersecting_tails}, $T_i \cup T_j = G$.

On the other hand, if $E(H) \not\subseteq \bigcup_{i=1}^k E(T_i)$, then there exists some $e \in E(H)$ such that $e \not\in \bigcup_{i=1}^k E(T_i)$. Without loss of generality, $H$ intersects both $T_1$ and $T_2$ non-trivially, and neither tail contains the other. Deleting $e$ from $G$ defines two subgraphs $G_1 \supseteq T_1$ and $G_2 \supseteq T_2$ such that $V(G_1) \cup V(G_2) = V(G)$. But $T_1 = G_1$ and $T_2 = G_2$, since $T_1$ and $T_2$ are both $1$-valent (given by $e$). So $V(T_1) \cup V(T_2) = G$, as desired.
\end{proof}

\begin{lemma} \label{lem:extremal_assignments_biject_simplicial_complexes}
If $g \geq 1$, $n \geq 1$, there is a bijection between the extremal assignments supported on rational tails and  simplicial complexes via $\calK \mapsto \calZ_\calK$ (Definition \ref{def:simpstab}).

If $g = 0$, $n \geq 3$, there is a bijection between the extremal assignments supported on rational tails and at least triparted simplicial complexes via $\calK \mapsto \calZ_\calK$.
\end{lemma}

\begin{proof}
First, we show that $\calK \to \calZ_\calK$ is well-defined, i.e., that $\calZ_\calK$ is an extremal assignment. Let $\calK$ be a simplicial complex on $[n]$, and assume $\calK$ is at least triparted if $g = 0$. We begin by checking Axiom (Z2).

Let $\pi : G' \to G$ be an edge contraction. Suppose $v \in V(G)$. If $v \in V(\calZ_\calK(G))$ then $v$ belongs to a rational tail $T$ of $G$ with $I = \Marks(T) \in \calK$. The preimage $\pi^{-1}(T)$ is also an $I$-marked rational tail, so $\pi^{-1}(v) \subseteq \calZ_\calK(G')$.

Conversely, if $v \in V(G)$ is such that $\pi^{-1}(v) \subseteq \calZ_\calK(G')$, then the vertices of $\pi^{-1}(v)$ all belong to rational tails with marks in $\calK$. By Lemma \ref{lem:subset_of_tails}, either $\pi^{-1}(v)$ factors through a tail $T'$ of $G'$ so that $I = \Marks(T') \in \calK$ or every vertex of $G'$ belongs to a rational tail with marks in $\calK$. The latter is impossible if $g > 0$ or $\calK$ is at least triparted. Then $\pi(T')$ is an $I$-marked rational tail of $G$ so $v \in \calZ_\calK(G)$, as required. We conclude that Axiom (Z2) holds.

In either case it is clear that $\calZ_\calK(G)$ is empty for the one-vertex no-edge graph. Axiom (Z1) follows, and we conclude that $\calZ_\calK$ is an extremal assignment.

Next, we claim that the map $\calK \mapsto \calZ_\calK$ is injective on the relevant complexes $\calK$. If $g > 0$, and $\calK$ and $\calK'$ are distinct simplicial complexes, then without loss of generality there is a subset $S \subseteq [n]$ with $2 \leq S \leq n$ so that $S \in \calK$ and $S \not\in \calK'$. Since $g > 0$ there is a graph $G_S \in \mathbb{G}_{g,n}$ with an $S$-marked rational tail. Then $\calZ_{\calK}$ and $\calZ_{\calK'}$ differ on $G_S$, so we have injectivity in the $g > 0$ case.

Let $g=0$ and note that if $\calK$ is an at least triparted simplicial complex, then $\calK$ does not contain any subset $S \subseteq [n]$ with $|S| = n - 1$ or $|S| = n$. (Otherwise, either $S \cup ([n] - S)$ or $S$ would be a $\calK$-partition of $[n]$.) Now if $\calK$ and $\calK'$ are distinct at least triparted simplicial complexes, then without loss of generality, there is a subset $S \subseteq [n]$ with $2 \leq |S| \leq n - 2$ so that $S \in \calK$ and $S \not\in \calK'$. Since $2 \leq |S| \leq n - 2$, there is a graph $G_S \in \mathbb{G}_{0,n}$ with an $S$-marked rational tail, and we conclude that we have injectivity as in the previous paragraph.

It is straightforward to check that $\calZ_{\calK_\calZ} = \calZ$ as extremal assignments. It remains only to show that $\calK_\calZ$ is at least triparted in genus zero. Let $\calZ$ be an extremal assignment on $\mathbb{G}_{0,n}$ supported on rational tails, and suppose that $\calK_\calZ$ is not at least triparted. Then there exists some partition $\calP$ of $[n]$ such that $\calP = \{P_1\}$ (Case I) or $\calP = \{P_1,P_2\}$ (Case II), and $P_i \in \calK_\calZ$. In Case I, this implies that $\calZ(G) = G$ for $G$ the single-vertex graph in $\mathbb{G}_{0,n}$, since the simplices in $\calK_\calZ$ correspond to rational tails determined by $\calZ$.

In Case II, consider a graph $G_1 \in \mathbb{G}_{0,n}$ so that a connected component of $\calZ(G_1)$ is a $P_1$-marked rational tail, and define $G_2 \in \mathbb{G}_{0,n}$ analogously. Let $G' \in \mathbb{G}_{0,n}$ be the graph defined by two vertices $v_1$ and $v_2$ sharing an edge, with $v_i$ marked by $P_i$. Since $\calZ$ is an extremal assignment, the edge contractions $G_1 \to G'$ and $G_2 \to G'$ imply that $\calZ(G') = G'$, contradicting the fact that $\calZ$ is an extremal assignment. We conclude that $\calK_\calZ$ must be at least triparted.
\end{proof}

We have now described all the possible extremal assignments of nodal compactifications admitting collisions. Our next step is to identify the moduli problem handed to us by $\calZ_\calK$-stability with the more typical moduli problem we described in Section \ref{ssec:def_simplicial_spaces}.

\begin{proposition} \label{prop:simplicial_stable_is_ZK_stable}
Let $g, n$ be integers with $g \geq 0$ and $n \geq 1$ ($n \geq 3$ if $g = 0$). For each simplicial complex $\calK$ on $[n]$ (at least triparted if $g = 0$)
we have
\[
  \Moduli_{g,n}(\calZ_\calK) = \Moduli_{g,\calK}
\]
as open substacks of $\mathcal{V}_{g,n}$. In particular, $\Moduli_{g, \calK}$ is a modular compactification of $\moduli_{g,n}$ over $\Spec \Z$.
\end{proposition}

\begin{proof}
We need to show that a curve is $\calK$-stable if and only if it is $\calZ_\calK$-stable.

Suppose that $(C; p_1,\ldots, p_n)$ is a $\calK$-stable curve. Let $x_1, \ldots, x_\ell$ be the closed points of $C$ marked by at least two markings. For each $i = 1,\ldots, \ell$, let $Z_i$ be a smooth rational curve with $|\Marks(x_i)|$ distinct markings indexed by $\Marks(x_i)$ and an additional disjoint marking $q_i$. Let $C^s$ be the $n$-pointed nodal curve obtained by gluing the point $x_i$ of $C$ to the point $q_i$ of $Z_i$ for each $i$. Now, $C^s$ is nodal with distinct markings by construction. Moreover, each rational component of $C^s$ coming from $C$ has at least 3 special points by definition of $\calK$-stable, and the remaining components $Z_1, \ldots, Z_\ell$ have at least 3 special points by construction, so $C^s$ is stable. There is a map $\phi : C^s \to C$ given by the identity on $C$ and by $Z_i \mapsto q_i$ on the remaining components. The exceptional locus of this map is $Z_1 \cup \cdots \cup Z_\ell$. Chasing definitions, these are precisely the tails of $C^s$ with markings in $\calZ_\calK$. Clearly, $\phi$ is surjective with connected fibers, an isomorphism on the complement of $Z_1, \ldots, Z_\ell$, and the image of each $Z_i$ is a point with $g_s(x_i) = 0$, so $(C; p_1,\ldots, p_n)$ is $\calZ_\calK$-stable.

Conversely, suppose that $(C; p_1, \ldots, p_n)$ is a $\calZ_\calK$-stable curve. Then there is a contraction $\phi : C^s \to C$ as in Definition \ref{def:zstability} where $C^s$ is a stable curve with exceptional locus $\calZ_\calK(C^s)$. By definition of $\calZ_\calK$, this exceptional locus consists of the tails with markings in $\calK$. Since these are rational tails, their images are smooth points marked by collections of points indexed by subsets belonging to $\calK$. It follows that $C$ is a nodal curve with smooth markings, so axioms (K1) and (K2) are satisfied. The markings not in $\calZ_\calK(C^s)$ remain distinct, so axiom (K3) is satisfied. Now, if $Z$ is a rational tail of $C$, its pre-image $\phi^{-1}(Z)$ is also a rational tail, and since $\phi^{-1}(Z)$ was not entirely contracted by $\phi$, $\phi^{-1}(Z) \not\subseteq \calZ_\calK(C^s)$. It follows $\Marks(Z) = \Marks(\phi^{-1}(Z)) \not\in \calK$, so we satisfy axiom (K4). Finally, if $K$ is a rational irreducible component of $C$, then there is a rational irreducible component $K^s$ of $C^s$ mapping birationally onto $K$. Since $C^s$ is stable, the normalization of $K^s$ has at least three special points. Following the contraction, the normalization of $K$ also has at least three special points. Then $(C; p_1,\ldots,p_n)$ satisfies the final axiom (K5). We conclude that $(C; p_1,\ldots, p_n)$ is $\calK$-stable, as desired.
\end{proof}

Now that we have identified the moduli spaces $\Moduli_{g,\calK}$ with the moduli spaces $\Moduli_{g,n}(\calZ_\calK)$, we learn that $\Moduli_{g,\calK}$ has all of the nice properties of stable modular compactifications. In particular, the spaces $\Moduli_{g,\calK}$ are open inside of $\mathcal{V}_{g,n}$. A consequence of this is is that each $\Moduli_{g,\calK}$ is a regluing of spaces of spaces of weighted pointed stable curves. We use this to conclude that the moduli spaces $\Moduli_{g,\calK}$ are smooth in Theorem \ref{thm:simplicial_spaces_are_modular_compactifications_nonintro}.

\begin{theorem} [{{Frankenstein property}}] \label{thm:simplicial_frankenstein}
Let $g,n, \calK$ be as in Theorem \ref{thm:simplicial_spaces_are_modular_compactifications_nonintro}. Then there is a finite open cover $\{ U_i \}_{i \in I}$ of $\Moduli_{g,\calK}$ and weights $\{ \mathcal{A}_i \}_{i \in I}$ so that for each $i \in I$, $U_i$ is an open substack of the space of weighted stable curves $\Moduli_{g,\mathcal{A}_i}$.
\end{theorem}

\begin{proof}
Proposition \ref{prop:simplicial_stable_is_ZK_stable} implies that $\Moduli_{g,\calK}$ is an open substack of $\mathcal{V}_{g,n}$. Similarly, the stack $\Moduli_{g,\calA}$ for each weight vector $\mathcal{A}$ is an open substack of $\mathcal{V}_{g,n}$.
For each collection of weights $\mathcal{A}$, let $U_\mathcal{A}$ be the open substack $\Moduli_{g,\mathcal{A}} \cap \Moduli_{g,\calK}$ of $\Moduli_{g,\calK}$. To see that we have a cover, we show that every geometric point of $\Moduli_{g,\calK}$ belongs to some $U_\mathcal{A}$.

Suppose $(C;p_1,\dots,p_n)$ is a $\calK$-stable curve. For each $i \in [n]$, define $a_i := \frac{1}{N_i}$, where $N_i = |\{p_j \stc p_j = p_i, j\in [n]\}|$, and let $\calA := (a_1,\dots,a_n)$. It is easily verified that $(C;p_1,\dots,p_n)$ is stable with respect to the weights $\calA$.

Finally, since there are finitely many distinct collision complexes on $[n]$, there are only finitely many distinct modular compactifications $\Moduli_{g,\calA}$, so we may pass to a finite cover.
\end{proof}

\begin{theorem} \label{thm:simplicial_spaces_are_modular_compactifications_nonintro}
  Let $g \geq 0, n \geq 1$ be integers and $\calK$ a simplicial complex on $[n]$. If $g = 0$ then assume in addition that $\calK$ is at least triparted. Then $\Moduli_{g,\calK}$ is a smooth Deligne-Mumford nodal modular compactification of $\moduli_{g,n}$ over $\Z$ admitting collisions such that $\calK(\Moduli_{g,\calK}) = \calK$.
\end{theorem}
\begin{proof}
We know $\Moduli_{g,\calK}$ is a modular compactification of $\moduli_{g,n}$ over $\Z$ by Proposition \ref{prop:simplicial_stable_is_ZK_stable}. It is nodal and has only smooth markings by definition. It is also clear that $\calK(\Moduli_{g,\calK}) = \calK$. Finally, $\Moduli_{g,\calK}$ is smooth and Deligne-Mumford since, by Theorem \ref{thm:simplicial_frankenstein}, it admits an open cover by smooth Deligne-Mumford stacks.
\end{proof}

\begin{theorem} \label{thm:simplicial_spaces_classification_nonintro}
  If $\moduli$ is a nodal compactification of $\moduli_{g,n}$ over $\Z$ admitting collisions, then, as open substacks of $\mathcal{V}_{G,n}$,
  \[
    \moduli = \Moduli_{g,\calK(\moduli)}.
  \]
\end{theorem}
\begin{proof}
By Lemma \ref{lem:nodal_implies_stable} and Lemma \ref{lem:extremal_assignments_biject_simplicial_complexes}, there is a simplicial complex $\calK$ so that $\moduli = \Moduli_{g,n}(\calZ_\calK)$. By Proposition \ref{prop:simplicial_stable_is_ZK_stable}, $\Moduli_{g,n}(\calZ_\calK) = \Moduli_{g,\calK}$. Finally, by Theorem \ref{thm:simplicial_spaces_are_modular_compactifications_nonintro}, $\calK(\moduli) = \calK(\Moduli_{g,\calK}) = \calK$.
\end{proof}

\subsection{Connections to Hassett compactifications}

As suggested by the Frankenstein property (Theorem \ref{thm:simplicial_frankenstein}), the notion of simplicial stability recovers that of Hassett stability but is strictly more general. One easily verifies, for example, that there does not exist weight data $\calA = (a_1,\dots,a_5)$ so that $\Moduli_{g,\calA} = \Moduli_{g,\calK}$ for $\calK$ the union of two disjoint one-simplices and an isolated vertex.

\begin{proposition}[cf. {\cite[Definition 2.8]{alexeevguy}}]\label{prop:hassettsupsimp}
Let $\mathcal{A} = (a_1,\dots,a_n)$ be weight data. Then there exists a simplicial complex $\mathcal{K}_{\mathcal{A}}$ such that $\Moduli_{g,\mathcal{A}} \cong \Moduli_{g,\mathcal{K}_{\mathcal{A}}}$.
\end{proposition}
\begin{proof}
One computes the collision complex $\mathcal{K}_{\mathcal{A}}$ of $\Moduli_{g,\calA}$ as follows: let the 0-skeleton be given by $\set{\set{1},\dots,\set{n}}$, and add a $d$-simplex on the vertices $i_1,\dots,i_{d+1}$ whenever $a_{i_1} + \cdots + a_{i_{d+1}} \leq 1$. Then the desired equality follows from Theorem \ref{thm:simplicial_spaces_classification_nonintro}.
\end{proof}

Under this identification, the weight data $(1,1,\dots,1)$ corresponds to the trivial simplicial complex, which gives the following immediate corollary.

\begin{corollary}
If $\mathcal{T}$ is the trivial simplicial complex on $[n]$, then $\Moduli_{g,\mathcal{T}} \cong \Moduli_{g,n}$. \qed 
\end{corollary}

The description given in Proposition \ref{prop:hassettsupsimp} characterizes Hassett spaces within the larger notion of simplicially stable spaces in the following way. Recall that a \emphbf{threshold complex} is a complex $\mathcal{K}$ such that there exists $M\in\mathbb{Q}$ and a weight function $w:K_0 \to \mathbb{Q}_{\geq 0}$ so that for any subset $V \subseteq K_0$ we have $V \in \mathcal{K}$ if and only if $\sum_{v\in V}w(v) \leq M$. Given a threshold complex, we may scale $M$ and $w$ by $M$ in order to recover valid weight data for a Hassett space. In this way we arrive at Corollary \ref{cor:hassettthreshold}.

\begin{corollary}\label{cor:hassettthreshold}
A simplicial complex $\mathcal{K}$ is a threshold complex if and only if there exists weight data $\mathcal{A}$ such that $\Moduli_{g,\mathcal{K}} \cong \Moduli_{g,\mathcal{A}}$. \qed
\end{corollary}

This formulation transforms the question concerning the number of isomorphism classes of $\Moduli_{g,\mathcal{A}}$ in \cite{hassett_weighted_curves} into a purely combinatorial problem.

\begin{corollary}[cf. {\cite[Corollary 3.25]{adgh-hassett}}]\label{cor:chambercount}
For fixed $g\geq1$ and $n$, the number of Hassett spaces (as open substacks of $\mathcal{V}_{g,n}$) is the same as the number of threshold complexes on $[n]$ up to isomorphism (of complexes). \qed
\end{corollary}

\section{Contractions via tropical geometry}

In this section, we develop a system for contracting genus 0 tails of families of log curves then show how to combine this with the formalism of \cite{rsw} and \cite{bozlee_thesis} for contracting elliptic subcurves to Gorenstein singularities. We start by defining ``tail functions," which encode which genus 0 tails to contract.

\subsection{Contractions of tails}

We begin by translating the notion of rational tail to the tropical setting. We only consider contracting stable rational tails.

\begin{definition} \label{def:tail}
Let $\Gamma$ be a tropical curve with edge lengths in a sharp fs monoid $P$.
Say that a sub-tropical curve $T$ of $\Gamma$ is a \emphbf{rational tail} if its underlying graph is a rational tail.
We say that $T$ is a \emphbf{stable rational tail} if in addition we have that the valence in $\Gamma$ of each vertex of $T$ is at least 3.

Observe that since $g(T)$ is zero, each edge of $T$ admits a well-defined orientation away from its leading edge $e_T$. Note also that $e_T$ does not belong to $T$, since $T$ is a vertex-induced subgraph of $\Gamma$.
\end{definition}

In \cite{hassett_weighted_curves}, contractions of rational tails were performed using a sequence of appropriate twists of the relative log dualizing sheaf by the markings. We also ultimately use a twist of the log dualizing sheaf. However, our twists are given not by a single linear combination of the markings, but instead by a piecewise linear function. Non-tropically, this translates to a twist by a vertical divisor rather than by sections. The kind of piecewise linear function we need is what we call a tail function below.

\begin{definition} \label{def:tail_function}
Suppose given a semistable tropical curve $\Gamma$, an integer $k \geq 0$, and $k$ disjoint stable rational tails $T_1, \ldots, T_k$ such that $V(T_1) \cup \cdots \cup V(T_k) \neq V(\Gamma)$.

If $e$ is an edge of $\Gamma$ incident to $T_i$ for some $i\in \{1,\dots,k\}$, let $T_e$ be the unique stable rational tail of $\Gamma$ contained in $T_i$ with leading edge $e$, and write $n_e$ be the number of markings incident to $T_e$. Now, let $\mu = \mu_{T_1,\ldots,T_k}$ be the unique piecewise linear function satisfying:
\begin{enumerate}
  \item $\mu$ has slope 0 on legs
  \item $\mu(v) = 0$ for each vertex $v \in V(\Gamma) - (V(T_1) \cup \cdots \cup V(T_k))$
  \item for each edge $e$ with at least one endpoint in one of the tails $T_1,\ldots,T_k$, the slope of $\mu$ on $e$ oriented in the direction of $T_e$ is $n_e - 1$.
\end{enumerate}
If $k = 0$, observe that $\mu$ is identically 0.

We say that a piecewise linear function $\mu$ on $\Gamma$ is a \emphbf{tail function} if there are stable rational tails $T_1,\ldots, T_k$ so that $\mu = \mu_{T_1, \ldots, T_k}$.

If $\pi : C \to S$ is a family of log curves, we say that $\mu \in \Gamma(C, \ol{M}_C)$ is a \emphbf{tail function} if for each geometric point $s \to S$, $\PL(\mu|_{C_s})$ is a tail function of $\trop(C_s)$.
\end{definition}

\begin{lemma} \label{lem:tail_function_determined_by_support}
If $\mu = \mu_{T_1,\ldots,T_k}$, then $|\mu| = T_1 \cup \cdots \cup T_k$. 
\end{lemma}
\begin{proof}
By definition, if $v$ is a vertex of $V(\Gamma) - (V(T_1) \cup \cdots \cup V(T_k))$, then $\mu(v) = 0$, so $v \not\in |\mu|$.

Conversely, let $v$ be a vertex of $T_i$ for some $i$. Let $w$ be the vertex of $T_i$ adjacent to its leading edge $e_{T_i}$. As $T_i$ is a tree, there is a unique (possibly trivial) path $e_1,\dots,e_\ell$ from $v$ to $w$. Since $T_{e_j}$ is itself a stable rational tail for each $j\in\{1,\dots,\ell\}$, we have $n_{e_j} \geq 2$ (similarly,  $n_e \geq 2$). Since $(n_e-1)\delta(e) > 0$, altogether we have
\[
  \mu(v) = (n_e - 1)\delta(e) + \sum_{j=1}^\ell(n_{e_j} - 1)\delta(e_j) > 0,
\]
as desired.
\end{proof}

\begin{remark}
The conclusion of the preceding lemma fails if the tails $T_i$ are not assumed to be stable.
\end{remark}

\begin{figure}[h]
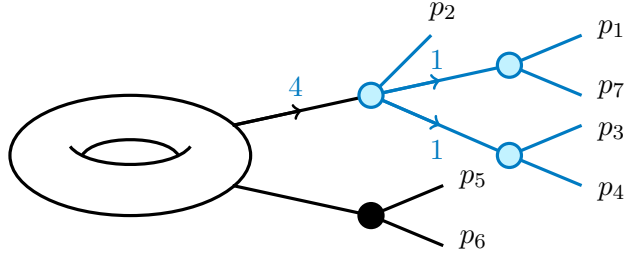

    \TailFunction
    \caption{A tail function on a tropical curve $\Gamma$. The associated tail and non-zero slopes are colored {\color{bluestroke}blue}.} \label{fig:tailfunction}
\end{figure}

In order to construct tail functions in families, one needs to assign tail functions in a manner compatible with generalization, or tropically, compatibly with weighted edge contractions. The following lemma explains how tail functions vary along these weighted edge contractions.

\begin{lemma} \label{lem:tail_function_weighted_edge_contraction}
Let $\pi : \Gamma \to \Gamma'$ be a weighted edge contraction and let $\mu_{\Gamma'}$ be a tail function on $\Gamma'$. Then $\mu_\Gamma = \pi^*(\mu_{\Gamma'})$ is the tail function on $\Gamma$ supported on the vertices $v \in V(\Gamma)$ such that $\ul{\pi}^{-1}(v) \subseteq |\mu_{\Gamma'}|$.
\end{lemma}
\begin{proof}
Let $e\in \Gamma$ be an edge and $e' = \ul{\pi}^{-1}(e) \in \Gamma'$ its unique pre-image. Observe that $T$ is an $I$-marked rational tail of $\Gamma$ with leading edge $e$ if and only if $T' = \ul{\pi}^{-1}(T)$ is an $I$-marked rational tail of $\Gamma'$ with leading edge $e'$. Moreover, if $T'$ is contained in a \emph{stable} rational tail of $\Gamma'$, then $T$ is also stable.

Let $T_1, \ldots, T_k$ be the maximal connected subgraphs of $\Gamma$ such that $\ul{\pi}^{-1}(T_i) \subseteq |\mu_{\Gamma'}|$. Note that a vertex $v \in V(\Gamma)$ belongs to one of the $T_i$ if and only if $\ul{\pi}^{-1}(v) \subseteq |\mu_{\Gamma'}|$. By the preceding paragraph, taking into account that the support of $\mu_{\Gamma'}$ is a union of disjoint stable rational tails, $T_1, \ldots, T_k$ is also a family of disjoint stable rational tails. To complete the proof we now check that $\mu_{\Gamma} = \mu_{T_1,\ldots,T_k}$.

Notice that $e$ is incident to one of the tails $T_1, \ldots, T_k$ if and only if $e'$ is incident to the support of $\mu_{\Gamma'}$. If $e'$ is not adjacent to a tail of $\Gamma'$, then $\mu_\Gamma(v) = \mu_\Gamma(w) = 0$ for $v,w$ adjacent to $e$. This implies that $\mu(v) = 0$ for all $v \not\in V(T_1) \cup \cdots \cup V(T_k)$.
On the other hand, if $e'$ is adjacent to a tail of $\Gamma'$, then, since number of markings (and genus) is preserved under $\pi$, we have $n_{e} = n_{e'}$. Therefore, if $e$ is an edge incident to one of $T_1, \ldots, T_k$, $\mu_\Gamma$ has the desired slope on $e$.
\end{proof}

The following lemma describes the twisted line bundle used to perform the contraction in Proposition \ref{prop:tail_contraction}.

\begin{lemma} \label{lem:tail_contraction_sheaf}
Let $\pi : C \to S$ be a family of semistable log curves of positive genus over a geometric point and let
$\mu \in \Gamma(C, \ol{M}_C)$ be the tail function associated to the disjoint rational tails
$T_1, \ldots, T_k$. Let $e_1, \ldots, e_k$ be their leading edges. Let $T = T_1 \cup \cdots \cup T_k$ be the support of $\mu$, and let $R = \overline{C - T}$. Let $\mathcal{L} = \omega_{C/S}(\Sigma) \otimes \mathscr{O}_C(\mu)$. Then
\begin{enumerate}
  \item \label{lem:tail_contraction_sheaf_1} $\mathcal{L} \otimes \mathscr{O}_T \cong \mathscr{O}_T$
  \item \label{lem:tail_contraction_sheaf_2} $\mathcal{L} \otimes \mathscr{O}_R = \omega_{C/S}(\Sigma)|_R \otimes \mathscr{O}_R(\sum_{i = 1}^k (n_{e_i} - 1) e_i).$
  \item \label{lem:tail_contraction_sheaf_3} $H^1(C,\mathscr{L}) = 0$
  \item \label{lem:tail_contraction_sheaf_4} if $C$ is stable, then $\mathscr{L}$ is ample on $R$
\end{enumerate}
\end{lemma}
\begin{proof} \hfill
\begin{enumerate}
  \item Since $T_1, \ldots, T_k$ have arithmetic genus 0, it suffices to show that the isomorphism holds on each irreducible component $Z$ of $T$. Let $q$ be the node of $Z$ heading towards the core and $q_1, \ldots, q_\ell$ be the remaining nodes contained in $Z$. Identifying nodes with the corresponding edges of $\trop(C)$, observe that
  $n_q = \sum_{i = 1}^\ell n_{q_i} + \deg(\Sigma|_Z)$.
  
  Then
  \begin{align*}
    \omega_{C/S}(\Sigma) \otimes \mathscr{O}_Z &\cong \mathscr{O}_{Z}(-2 + ( \ell + 1) + \Sigma|_Z) \\
      &\cong \mathscr{O}_Z(\Sigma|_Z + \ell - 1)
  \end{align*}
  
  On the other hand, by Proposition \ref{prop:log_curve_line_bundles},
  \begin{align*}
    \mathscr{O}_C(\mu) \otimes \mathscr{O}_Z &\cong \mathscr{O}_Z\left( -(n_q - 1) e + \sum_{i = 1}^\ell (n_{q_i} - 1) q_i \right) \\
      &\cong \mathscr{O}_Z\left( \left(\sum_{i = 1}^\ell n_{q_i} - n_q\right)  - (\ell - 1) \right) \\
      &\cong \mathscr{O}_Z\left( -\Sigma|_Z - \ell + 1 \right).
  \end{align*}
  Tensoring, $\mathcal{L}|_Z \cong \mathscr{O}_Z$ as desired.

\item This follows directly from Proposition \ref{prop:log_curve_line_bundles} and the definition of tail function.

\item 
By (semi-)stability, $n_{e_i} \geq 2$ ($\geq 1)$ for all $i$. Then by (ii), $\mathcal{L} \otimes \mathscr{O}_R$ is at least as positive as $\omega_{C/S}(\Sigma) \otimes \mathscr{O}_R$, so $H^1(R, \mathcal{L} \otimes \mathscr{O}_R) = 0$ too. On the other hand, $H^1(T, \mathcal{L} \otimes \mathscr{O}_T) = H^1(T, \mathscr{O}_T) = 0$ since $T_1, \ldots, T_k$ are genus 0. Let $\nu : R \sqcup T \to C$ be the natural map. Then the result follows by analyzing the long exact sequence associated to $0 \to \mathcal{L} \to \nu_*\nu^*\mathcal{L} \to Q \to 0$.
\item This follows immediately from \eqref{lem:tail_contraction_sheaf_2} 
\end{enumerate}
\end{proof}

Everything is now in place to construct contractions.

\begin{proposition} \label{prop:tail_contraction}
Let $\pi : C \to S$ be a family of stable log curves of genus $g > 0$ and let $\mu \in \Gamma(C, \ol{M}_C)$ be a tail function. Let $\mathcal{L} = \omega_{C/S}(\Sigma) \otimes \mathscr{O}_C(\mu)$. Then there is a contraction
\[
  \begin{tikzcd}
    C \ar[rr, "\tau"] \ar[rd, "\pi"] & & \ol{C} = \PProj \bigoplus_{i = 0}^\infty \pi_*\mathcal{L}^{\otimes i} \ar[ld, "\ol{\pi}"] \\
     & S
  \end{tikzcd}
\]
so that
\begin{enumerate}
  \item $\ol{\pi}$ is a flat family of pointed nodal curves;
  \item the contraction commutes with base change in the sense that $\ol{C} \times_S T \cong \ol{C \times_S T}$;
  \item The pull back of $\tau$ to each geometric point $s \to S$
  \begin{enumerate}
    \item has connected fibers;
    \item contracts each connected component of the support of $\mu|_s$ to a smooth point; and
    \item restricts to an isomorphism on the complement of the support of $\mu$.
  \end{enumerate}
\end{enumerate}
\end{proposition}

\begin{proof}
Note that $\pi_*$ commutes with base change since $\mathcal{L}^{\otimes i}$ has no cohomology on geometric fibers (Lemma \ref{lem:tail_contraction_sheaf}\eqref{lem:tail_contraction_sheaf_3}). This implies (ii).

Then (iii) can be checked on geometric fibers, where it holds by Lemma \ref{lem:tail_contraction_sheaf} part \eqref{lem:tail_contraction_sheaf_1} and \eqref{lem:tail_contraction_sheaf_4}.

In order to show that $\ol{\pi}$ is flat, it is sufficient to show that $\pi_*\mathcal{L}^{\otimes i}$ is locally free for $i \geq 1$. We note that the degree of $\mathcal{L}^{\otimes i}$ is constant on fibers and $\mathcal{L}^{\otimes i}$ has no cohomology on fibers. It follows by Riemann-Roch that the rank of $\pi_*\mathcal{L}^{\otimes i}$ is constant over $S$. Grauert's Theorem then implies that $\pi_*\mathcal{L}^{\otimes i}$ is locally free.
\end{proof}

\subsection{Contractions of tails and elliptic subcurves}

Our goal now is to construct contractions of both elliptic subcurves and tails of radially aligned log (resp. tropical) curves. To do so, we package the necessary data into a \emphbf{contraction datum}.

\begin{definition}
If $\Gamma$ is a radially aligned tropical curve of genus one with edge lengths in a sharp fs monoid $P_\Gamma$, an element $\rho \in P_\Gamma$ is said to be a \emphbf{radius} if $\rho = \lambda(v)$ for some vertex $v$ of $\Gamma$.

Given a radially aligned tropical curve $\Gamma$ and radius $\rho$, write $\tilde{\Gamma}$ for the subdivision of $\Gamma$ where $\lambda = \rho$. Write $\Delta$ for the tropical subcurve of $\tilde{\Gamma}$ where $\lambda < \rho$ and $\ol{\Delta}$ for the tropical subcurve where $\lambda \leq \rho$.

A \emphbf{contraction datum} $D = (\rho, \mu)$ on $\Gamma$ consists of
\begin{enumerate}
  \item a radius $\rho \in P_\Gamma$;
  \item a tail function $\mu \in \PL(\tilde{\Gamma})$ with support $T$
\end{enumerate}
such that $\ol{\Delta} \cap |\mu| = \emptyset$.
\end{definition}

\begin{definition}
If $\pi : C \to S$ is a family of radially aligned log curves, we say that a section $\rho \in \Gamma(S, \ol{M}_S)$ is a \emphbf{radius} if the restriction of $\rho$ at each geometric point $s \to S$ is a radius of $\trop(C \times_S s)$.

Given a radially aligned log curve $C$ and radius $\rho$, write $\tilde{C}$ for the subdivision of $C$ where $\lambda = \rho$. Write $E$ for the subcurve of $\tilde{C}$ where $\lambda < \rho$.

A \emphbf{contraction datum} $D = (\rho, \mu)$ on a family of radially aligned log curves $\pi : C \to S$ of genus one consists of
\begin{enumerate}
  \item a radius $\rho \in \Gamma(S, \ol{M}_S)$;
  \item a tail function $\mu \in \Gamma(\tilde{C}, \ol{M}_{\tilde{C}})$ with support $T$
\end{enumerate}
such that $E \cap T = \emptyset$.
\end{definition}

Ultimately the connected components of $E$ and $T$ are contracted, with $E$ collapsing to an elliptic singularity, and the connected components of $T$ collapsing to smooth points where markings collide. The condition $E \cap T = \emptyset$ implies that $E \cup T$ does not contain the components of $\tilde{C}$ where $\lambda = \rho$. This condition (and its tropical analogue) are imposed to avoid markings on the elliptic singularity, or worse, contracting the entire curve to a point.

\begin{proposition} \label{prop:contract_everything}
Let $(\rho, \mu)$ be a contraction datum on an $n$-pointed stable radially aligned curve $(\pi : C \to S; p_1, \ldots, p_n)$ of genus one, with $T$ the support of $\mu$ and $E$ the radially aligned log subcurve of $C$ where $\lambda < \rho$. Then there is a diagram
\[
  \begin{tikzcd}
   & \tilde{C} \ar[dl, "\phi"'] \ar[dr, "\tau"] & \\
    C \ar[dr, "\pi"'] & & \ol{C} \ar[ld, "\ol{\pi}"] \\
     & S & 
  \end{tikzcd}
\]
so that
\begin{enumerate}
  \item $\phi$ is a log blowup inducing the subdivision at the locus where $\lambda = \rho$ on tropicalizations;
  \item $\ol{\pi}$ is a flat family of Gorenstein curves;
  \item the diagram commutes with base change in $S$;
  \item $\tau$ is a contraction with exceptional locus $E \cup T$
\end{enumerate}
and for each geometric point $s$ of $S$
\begin{enumerate}
  \item[(a)] When $E|_s$ is nonempty, the image of $E|_s$ under $\tau$ is an elliptic Gorenstein singularity $x$ with $\lev(x) = \lev(E|_s)$;
  \item[(b)] $\tau(\tilde{p}_i(s)) = \tau(\tilde{p}_j(s))$, where $\tilde{p}_i$ denotes the proper transform of $p_i$ in $\tilde{C}$, precisely when the images of $\tilde{p}_i(s)$ and $\tilde{p}_j(s)$ belong to a common connected component of $T|_s$,
\end{enumerate}

\end{proposition}

\begin{proof}
We first contract the elliptic subcurve by applying \cite[Theorem 4.1]{bkn_qstable} to form a diagram
\begin{equation} \label{eq:diagram_contracting_E}
 \begin{tikzcd}
   & \tilde{C} \ar[dl, "\phi"'] \ar[dr, "\tau'"] & \\
    C \ar[dr, "\pi"'] & & C' \ar[ld, "\pi'"] \\
     & S & 
  \end{tikzcd}
\end{equation}
where
\begin{enumerate}
  \item $\phi$ is a log blowup inducing the subdivision at the locus where $\lambda = \rho$ on tropicalizations;
  \item $\pi'$ is a flat and proper family of Gorenstein curves of genus one;
  \item $\tau'$ is a surjective map whose restriction to fibers contracts $E|_s$ to a singularity of level $\lev(E|_{s})$;
  \item $\tau'$ restricts to an isomorphism $\tilde{C} - E \to C' - \tau'(E)$.
\end{enumerate}
Moreover, this construction commutes with base change. Let $\Sigma'$ be the divisor of markings on $C'$ induced by $\tau' \circ \phi^{-1}$.

Our next step is to contract the rational tails. We begin by constructing a sheaf $\mathscr{J}$ on $C'$ so that $(\tau')^*\mathscr{J} = \mathscr{O}_{\tilde{C}}(\mu)$.

Let $T' = \tau'(T)$ and let
\begin{align*}
    \tilde{U} = \tilde{C} - T \quad &\text{ and } \quad U = C' - T'; \\
    \tilde{V} = \tilde{C} - E \quad &\text{ and } \quad V = C' - \tau'(E).
\end{align*}
Note that $U \cup V$ is an open cover of $C'$ and that $\tau'(\tilde{U}) = U$, $\tau'(\tilde{V}) = V$, $\tilde{U} = (\tau')^{-1}(U)$, and $\tilde{V} = (\tau')^{-1}(V)$. By (iv) above, $\tau'|_{\tilde{V}} : \tilde{V} \to V$ is an isomorphism. Let $\mathscr{J}$ be the invertible sheaf on $C'$ formed by gluing $\mathscr{O}_U$ on $U$ with $(\tau')_*(\mathscr{O}_{\tilde{V}}(\mu|_{\tilde{V}}))$ on $V$ via the map
\[
(\tau')_*\big(\,\mathscr{O}_{\tilde{U}}|_{\tilde{U} \cap \tilde{V}} \overset{\mu}{\to} \mathscr{O}_{\tilde{U} \cap \tilde{V}}(\mu)\,\big)
\]
given by the log structure. This is an isomorphism since $\mu$ restricts to 0 on $\tilde{U} \cap \tilde{V}$. Observe that $(\tau')^*\mathscr{J} = \mathscr{O}_{\tilde{C}}(\mu)$, as desired.

Now let $\mathscr{L} = \omega_{C'/S}(\Sigma') \otimes \mathscr{J}$. Since $C'$ is Gorenstein, $\mathscr{L}$ is an invertible sheaf.

We next claim that $R^i(\pi')_{*}(\mathscr{L}) = 0$ for all $i > 0$. If $S$ has finite dimension $n$, then $C'$ has dimension $n + 1$, so $R^{i}(\pi')_{*}(\mathscr{L}) = 0$ for $i > n + 1$. By standard arguments $\pi : C \to S$, its contraction datum, and hence all of diagram \eqref{eq:diagram_contracting_E}, are pulled back from a finite dimensional base. Then we may use cohomology and base change to conclude that $R^i(\pi')_*(\mathscr{L}) = 0$ for $i \gg 0.$ Then since cohomology vanishes on the fibers of $\pi'$ for $i > 1$, we use descending induction to conclude $R^i(\pi')_*(\mathscr{L}) = 0$ for $i > 1$.

To conclude that $R^1(\pi')_{*}(\mathscr{L}) = 0$, it suffices to show that $H^1(C'|_s, \mathscr{L}|_s) = 0$ for geometric points $s$ of $S$. 
Let $Z = \ol{C'|_s - T'|_s}$ and let $\nu : Z \sqcup T'|_{s} \to C'|_{s}$ be the natural map. Just as in Lemma \ref{lem:tail_contraction_sheaf} part \eqref{lem:tail_contraction_sheaf_2}, we may observe that $\mathscr{L} \otimes \mathscr{O}_Z \cong \omega_{C'/S}(\Sigma')|_Z \otimes \mathscr{O}_Z(D)$ where $D$ is a strictly effective divisor supported on $Z \cap T'|_s$. Since $T'$ is contained in $V$, we may use Lemma \ref{lem:tail_contraction_sheaf} part \eqref{lem:tail_contraction_sheaf_1} to see that $\mathscr{L} \otimes \mathscr{O}_{T'|_{s}} \cong \mathscr{O}_{T'|_{s}}$. 
Then we conclude $H^1(C'|_s, \mathscr{L}|_s) = 0$ as in Lemma \ref{lem:tail_contraction_sheaf} part \eqref{lem:tail_contraction_sheaf_3}.

Then by cohomology and base change, we have commutativity of $\pi_*'$ and base change for $\mathscr{L}$. Now, fiber by fiber, $\mathscr{L}|_{Z}$ is ample, since $\omega_{C'/S}(\Sigma')|_Z$ is already ample.

We now set $\ol{C} = \PProj_{S} \bigoplus_{i = 0}^\infty (\pi')_*(\mathscr{L}^{i})$. The rest follows.

In order to show that $\ol{\pi}$ is flat, it is sufficient to show that $\pi_*\mathcal{L}^{\otimes i}$ is locally free for $i \geq 1$. We note that the degree of $\mathcal{L}^{\otimes i}$ is constant on fibers and $\mathcal{L}^{\otimes i}$ has no cohomology on fibers. It follows by Riemann-Roch that the rank of $\pi_*\mathcal{L}^{\otimes i}$ is constant over $S$. Grauert's Theorem then implies that $\pi_*\mathcal{L}^{\otimes i}$ is locally free.
\end{proof}

\section{Maps induced by universal contraction data}

\subsection{Simplicially stable spaces}

We now use the technology of the previous section to construct reduction maps
\begin{align*}
  \rho_{\calK}: \Moduli_{g,n} \to \Moduli_{g,\mathcal{K}}
\end{align*}
analogous to the reduction maps of \cite[Theorem 4.1]{hassett_weighted_curves}.

Our strategy is to find all tail functions on the universal curve of $\Moduli_{g,n}$. Each of these induces a contraction by Proposition \ref{prop:tail_contraction}, and the resulting family of curves induces a map to $\Moduli_{g,\mathcal{K}}$ for some $\calK$. It is interesting to note that such universal tail functions are in bijection with collision complexes $\calK$ and with the extremal assignments $\calZ$ supported on rational tails: in principle one could discover the moduli spaces $\Moduli_{g,\calK}$ by starting with contractions of the universal curve.

\begin{lemma} \label{lem:univ_tail_function_equals_extremal_assignment}
Let $g, n$ be integers so that $g \geq 0$ and $n \geq 1$, and $n \geq 3$ if $g = 0$.
Let $\pi_{g,n} : C_{g,n} \to \Moduli_{g,n}$ be the universal curve with the basic log structure. Then tail functions on $\pi_{g,n}$ are in bijection with extremal assignments on $\Moduli_{g,n}$ supported on rational tails.
\end{lemma}
\begin{proof}
Regard the graphs $\Gamma \in \mathbb{G}_{g,n}$ as stable tropical curves by giving each the basic log structure.
Arguing analogously to \cite[Lemma 4.5]{bkn_qstable}, a tail function $\mu$ on $\pi_{g,n}$ is equivalent to a choice of tail function $\mu_{\Gamma}$ for each $\Gamma \in \mathbb{G}_{g,n}$ so that for each face contraction $f : \Gamma \to \Gamma'$, we have $f^*\mu_{\Gamma'} = \mu_{\Gamma}$.

Now, a choice of tail function $\mu_{\Gamma}$ on $\Gamma$ is the same as a choice of disjoint rational tails $T_1,\ldots, T_k$ of $\Gamma$ so that $\bigcup_i T_i \neq \Gamma$. Set $\mathcal{Z}_{\mu}(\Gamma) = \bigcup_i T_i$. By definition, we satisfy axiom (Z1).

Face contractions $f : \Gamma \to \Gamma'$ are in bijection with their underlying weighted edge contractions $\Gamma' \to \Gamma$ of underlying graphs. By Lemma \ref{lem:tail_function_weighted_edge_contraction}, $\calZ_{\mu}$ satisfies axiom (Z2). This process is clearly reversible, so we have the desired bijection.
\end{proof}

\begin{corollary} \label{cor:tail_function_is_simplicial_set}
With notation as in Lemma \ref{lem:univ_tail_function_equals_extremal_assignment}, if $g > 1$, tail functions on $\pi_{g,n}$ are in bijection with simplicial complexes $\calK$ on $[n]$. If $g = 0$, tail functions on $\pi_{g,n}$ are in bijection with at least triparted simplicial complexes $\calK$ on $[n]$.
\end{corollary}
\begin{proof}
Compose the bijections of Lemma \ref{lem:univ_tail_function_equals_extremal_assignment} and Lemma \ref{lem:extremal_assignments_biject_simplicial_complexes}.
\end{proof}

\begin{theorem} \label{thm:simplicial_reduction_map_nonintro}
Let $g \geq 0, n \geq 1$ be integers and $\calK$ a simplicial complex on $[n]$. If $g = 0$, assume also that $\calK$ is at least triparted. Then there is a morphism of algebraic stacks
\[
  \rho_{\calK}: \Moduli_{g,n} \to \Moduli_{g,\mathcal{K}}.
\]
restricting to an isomorphism on the dense open substack $\moduli_{g,n}$.
\end{theorem}
\begin{proof}
Let $\mu$ be the tail function on the universal curve $\pi_{g,n} : C_{g,n} \to \Moduli_{g,n}$ associated to $\calK$ by Corollary \ref{cor:tail_function_is_simplicial_set}.
By Proposition \ref{prop:tail_contraction}, there is a contraction morphism
\[
  \begin{tikzcd}
   C_{g,n} \ar[rr,"\tau"] \ar[rd, "\pi_{g,n}"'] & & \ol{C}_{g,n} \ar[ld, "\ol{\pi}"] \\
    & \Moduli_{g,n}
  \end{tikzcd}
\]
associated to $\mu$. By Proposition \ref{prop:tail_contraction} part (i), $\ol{\pi}$ is a flat family of pointed nodal curves. By construction, for each geometric point $s \to \Moduli_{g,n}$, the support of $\mu|_s$ consists of $\calZ_\calK(C_{g,n}|_s)$ where $\calZ_\calK$ is the extremal assignment associated to $\calK$. Therefore, by Proposition \ref{prop:tail_contraction} part (ii) and (iii) $\ol{C}$ is a family of $\calZ_\calK$-stable curves. By Proposition \ref{prop:simplicial_stable_is_ZK_stable}, $\ol{C}$ is a family of $\calK$-stable curves. This defines the morphism $\rho_{\calK}$. The morphism restricts to an isomorphism on $\moduli_{g,n}$ since $\mu = 0$ on smooth curves.
\end{proof}

\subsection{Genus one spaces} \label{ssec:genus_one_univ_contractions}

In this section, we construct a resolution of the rational map $\Moduli_{1,n} \dashrightarrow \Moduli_{1,\mathcal{K}}(Q)$ for each $(Q,\mathcal{K})$-stability condition using a universal contraction datum on $\Moduli_{1,n}^{rad}$. Conversely, the combinatorial data ($Q$, $\mathcal{K}$) emerge from universal contraction data, and we view this as evidence for the naturality of our definitions.

Fix an integer $n$. Let $I$ be the set of isomorphism classes of $n$-pointed basic stable radially aligned tropical curves. For each isomorphism class, fix a representative $\Gamma$ with edge lengths in the strict fs monoid $P_\Gamma$. Recall the following definition:

\begin{definition} (\cite[Definition 4.6]{bkn_qstable})
An $n$-pointed \emphbf{universal radius} $(\rho_\Gamma)_{\Gamma \in I}$ is a choice of radius $\rho \in P_\Gamma$ for each $\Gamma$, compatible with face contractions in the sense that whenever $\Gamma, \Gamma' \in I$ and $\pi : \Gamma \to \Gamma'$ is a face contraction, then $\pi^\sharp(\rho_{\Gamma'}) = \rho_{\Gamma}$.
\end{definition}

By \cite[Lemma 4.5]{bkn_qstable}, the data of a universal radius is equivalent to a radius $\rho \in \Gamma(\Moduli_{1,n}^{rad}, \ol{M}_{\Moduli_{1,n}^{rad}})$ on the universal curve $C$ of $\Moduli_{1,n}^{rad}$.

Recall also that associated to such a radius is a map $\tilde{C} \to C$ of the universal curve inducing the subdivision where $\lambda = \rho$
on the tropicalizations of the geometric fibers of $C$, constructed by taking the log blowup with respect to the log ideal generated by $\lambda$ and $\rho$.

\begin{definition}
Given a universal radius $(\rho_\Gamma)_{\Gamma \in I}$ and $\Gamma \in I$, write $\tilde{\Gamma}$ for the subdivision of $\Gamma$ where $\rho_\Gamma = \lambda$. Write $\Delta_\Gamma$ for the sub-tropical curve of $\tilde{\Gamma}$ where $\lambda < \rho$. Write $\ol{\Delta}_\Gamma$ for the sub-tropical curve of $\tilde{\Gamma}$ where $\lambda \leq \rho$. Observe that a face contraction $\Gamma \to \Gamma'$ induces a face contraction $\tilde{\pi} : \tilde{\Gamma} \to \tilde{\Gamma}'$.

An $n$-pointed \emphbf{universal contraction datum} consists of a contraction datum $(\rho_\Gamma, \mu_\Gamma)$ on each $\Gamma \in I$ such that
\begin{enumerate}
  \item $(\rho_\Gamma)_{\Gamma \in I}$ is a universal radius;
  \item for each face contraction $\pi : \Gamma \to \Gamma'$, we have $\tilde{\pi}^*(\mu_{\Gamma'}) = \mu_{\Gamma}$.
\end{enumerate}
\end{definition}

\begin{proposition}
A contraction datum on the universal curve $C \to \Moduli_{1,n}^{rad}$ is equivalent to the specification of a universal contraction datum.
\end{proposition}
\begin{proof}
The argument is identical to \cite[Lemma 4.5]{bkn_qstable}, with $\mu_\Gamma$ taking the place of $\rho_\Gamma$. Observe that as with radii, the piecewise linear functions $\mu_\Gamma$ are fixed under the automorphisms of all basic stable radially aligned log curves.
\end{proof}

\begin{definition}
A basic stable radially aligned tropical curve $\Gamma$ with edge lengths in $P$ is said to be a \emphbf{$k$-layer tree} if
\begin{enumerate}
    \item The core of $\Gamma$ is a vertex of genus one;
    \item $\Gamma$ has $k$ distinct nonzero radii $\rho_1 < \ldots < \rho_k$.
\end{enumerate}
\end{definition}

\begin{definition}[{{\cite[Definition 3.20]{bkn_qstable}}}] \label{def:partition_type}
Let $\Gamma$ be a radially aligned tropical curve with ordered radii $0 < \rho_1 < \ldots < \rho_k$.

Given a radius $\rho$, we may form a tropical curve $\widetilde{\Gamma}_{\lambda \geq \rho}$ by subdividing the edges and legs of $\Gamma$ where $\lambda = \rho$,
then deleting the locus where $\lambda < \rho$. We define the \emphbf{partition associated to the radius $\rho$} to be the partition of $\{1, \ldots, n\}$ induced by the components of $\widetilde{\Gamma}_{\lambda \geq \rho}$, and we denote it by $\Part(\rho)$.

We say that the resulting strict chain of partitions
\[
  \Part(\rho_1) \prec \Part(\rho_2) \prec \cdots \prec \Part(\rho_k)
\]
is the \emphbf{partition type} of $\Gamma$.
\end{definition}

Given a chain of partitions $\calP : P_1 \prec P_2 \prec \cdots \prec P_k$, there is a unique $k$-layer tree $\Gamma_{\calP}$ with partition type $\calP$. For the idea of the construction, see \cite[Lemma 6.6]{bkn_qstable}.

\begin{proposition}
A universal contraction datum is determined by its values on 1-layer trees.
\end{proposition}

\begin{proof}
The point is that a value in $P_\Gamma$ is a linear combination of free generators, and what the linear combination is can be recovered by face projections onto each of these free generators.

Suppose $\Gamma$ is a basic stable radially aligned tropical curve with monoid $P_\Gamma = \N e_1 \oplus \cdots \oplus \N e_k$. For each $i = 1,\ldots, k$, let $\pi_i : \Gamma_i \to \Gamma$ be the face contraction induced by the projection $\pi_i^* : P_\Gamma \to \N e_i$. Then
\[
  \rho_\Gamma = (\pi_1^*(\rho_\Gamma), \ldots, \pi_k^*(\rho_\Gamma)) = (\rho_{\Gamma_1}, \cdots, \rho_{\Gamma_k}),
\]
so $\rho_\Gamma$ can be recovered from the values $\rho_{\Gamma_1}, \ldots, \rho_{\Gamma_k}$ on 1-layer trees.

The same argument for each vertex $v$ of $\tilde{\Gamma}$ shows that $\mu(v)$ can be recovered from the values $\mu_{\tilde{\Gamma}_i}(\tilde{\pi}_i(v)) = \pi_i^*(\mu(v))$.
\end{proof}

This makes it important to understand contraction data on 1-layer trees. However, 2-layer trees impose compatibility relations on the contraction data on 1-layer trees, from which we deduce the downward closedness of $Q$ and $\calK$. The following straightforward results summarize what contraction data are possible on 1-layer trees (Lemma \ref{lem:1_layer_contraction_datum}, Figure \ref{fig:firstlemma}) and 2-layer trees (Lemma \ref{lem:2_layer_contraction_datum}, Figure \ref{fig:secondlemma}) and describe the compatibility conditions (Lemma \ref{lem:1_layer_tree_relations}, Figure \ref{fig:thirdlemma}).

\begin{definition}
If $P$ is a part of a partition $\mathcal{P} \in \Part([n])$, we say that $P$ is \emphbf{large} if $P$ contains at least two elements, and say $P$ is \emphbf{small} otherwise.

If $\mathcal{P}_1 \preceq \mathcal{P}_2$ is a chain of partitions, we say that a part $P$ of $\mathcal{P}_1$ is \emphbf{refined} if $P \not\in \mathcal{P}_2$.
\end{definition}

\begin{lemma} \label{lem:1_layer_contraction_datum}
Suppose $\Gamma \in I$ is a 1-layer tree with associated non-discrete partition $\mathcal{P}$. The contraction data $(\rho, \mu)$ admitted by $\Gamma$ are the following:
\begin{enumerate}
  \item (Case I) $\rho$ is the unique non-zero radius of $\Gamma$, and $\mu = 0$.
  \item (Case II) $\rho = 0$ and $\tilde{\Gamma} = \Gamma$. In this case, the rational tails of $\tilde{\Gamma}$ are in bijection with the large parts of $\mathcal{P}$, and $\mu_\Gamma$ is supported on a subset of them.
\end{enumerate}
The contraction data on $\Gamma$ in which $\rho = 0$ are therefore in bijection with the power set of the large parts of $\mathcal{P}$. \qed
\end{lemma}

\begin{figure}[h]
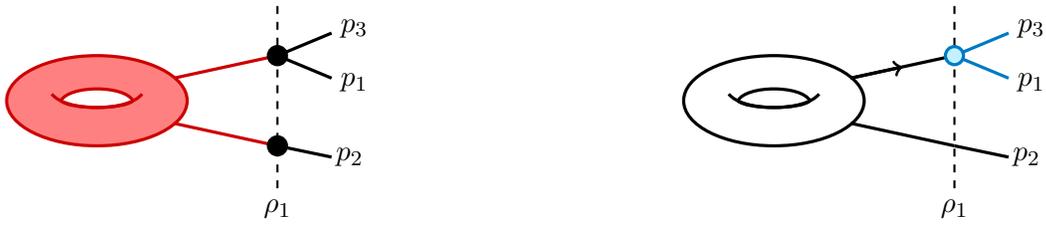

    \FirstLemma
    \caption{Two subdivisions $\tilde{\Gamma}$ of a basic stable radially aligned tropical curve $\Gamma$, a contraction datum in Case I (left) and a contraction datum in Case II (right). Components in {\color{redstroke}red} denote are contracted under $\rho$, and components in {\color{bluestroke}blue} are contracted under $\mu$.}
    \label{fig:firstlemma}
\end{figure}

\begin{lemma} \label{lem:2_layer_contraction_datum}
Suppose $\Gamma \in I$ is a 2-layer tree with associated chain of non-discrete partitions $\mathcal{P}_1 \prec \mathcal{P}_2$. Let $\rho_1 < \rho_2$ be the non-zero radii of $\Gamma$. The contraction data $(\rho, \mu)$ admitted by $\Gamma$ are the following:
\begin{enumerate}
  \item (Case I-I) $\rho_\Gamma = \rho_2$, and $\mu_\Gamma = 0$.
  \item (Case I-II) $\rho_\Gamma = \rho_1$. In this case, the rational tails of $\tilde{\Gamma}$ on which $\mu_\Gamma$ may be supported are in bijection with the large parts of $\mathcal{P}_2$.
  \item (Case II-II) $\rho_\Gamma = 0$. In this case, the rational tails
  of $\tilde{\Gamma}$ on which $\mu_\Gamma$ may be supported are in bijection with the union of the set of refined parts of $\mathcal{P}_1$ with the set of large parts of $\mathcal{P}_2$ that refine some part of $\mathcal{P}_1$. \qed
\end{enumerate}
\end{lemma}

\begin{figure}[h]
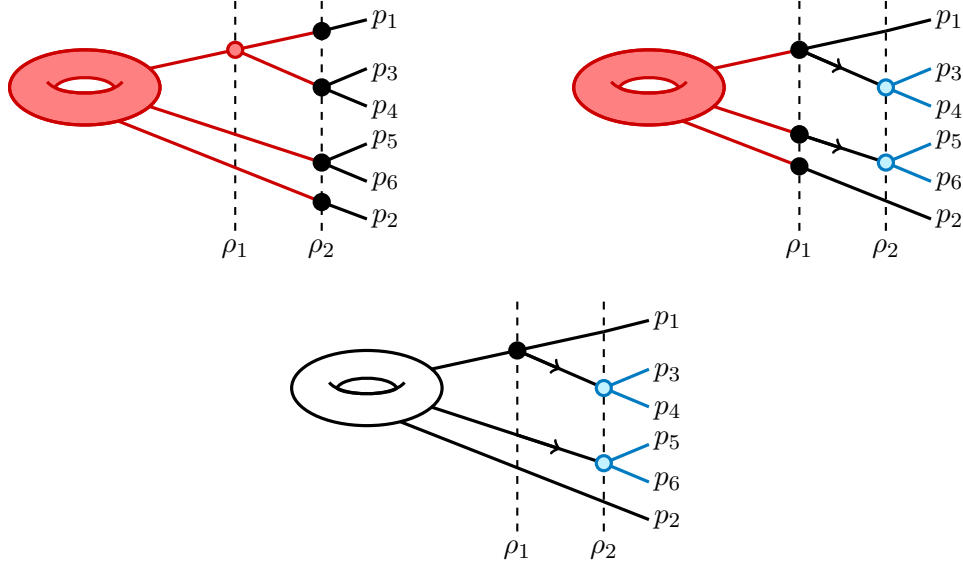

    \SecondLemma
    \caption{Three subdivisions $\tilde{\Gamma}$ of a basic stable radially aligned tropical curve $\Gamma$, with examples of contraction data in Case I-I (top left), Case I-II (top right), and Case II-II (bottom). For Case I-II, there are four possible choices of $\mu$ and for Case II-II there are six possible choices for $\mu$, depending on which rational tails are chosen.}
    \label{fig:secondlemma}
\end{figure}

\begin{lemma} \label{lem:1_layer_tree_relations}
Let $(\rho_\Gamma, \mu_\Gamma)_{\Gamma \in I}$ be a universal contraction datum, let $\mathcal{P}_1 \prec \mathcal{P}_2$ be a chain of non-discrete partitions, and let $\Gamma_1$ and $\Gamma_2$ be the 1-layer trees associated to $\mathcal{P}_1$ and $\mathcal{P}_2$. Then
\begin{enumerate}
  \item If $(\rho_{\Gamma_2}, \mu_{\Gamma_2})$ is in Case I of Lemma \ref{lem:1_layer_contraction_datum}, then so is $(\rho_{\Gamma_1}, \mu_{\Gamma_1})$.
  \item If $(\rho_{\Gamma_1}, \mu_{\Gamma_1})$ is in Case II of Lemma \ref{lem:1_layer_contraction_datum}, then so is $(\rho_{\Gamma_2}, \mu_{\Gamma_2})$. Moreover
  \begin{enumerate}
      \item If $P_1$ is a large part of $\mathcal{P}_1$ so that $\mu_{\Gamma_1}$ is nonzero on the corresponding component of $\Gamma_1$, and $P_2$ is a large part of $\mathcal{P}_2$ contained in $P_1$, then $\mu_{\Gamma_2}$ is nonzero on the corresponding component of $\Gamma_2$.
      \item If $P_2$ is a large part of $\mathcal{P}_2$ so that $\mu_{\Gamma_2}$ is nonzero on the corresponding component of $\Gamma_2$ and $P_2 \in \mathcal{P}_1$ (that is, if $P_2$ is an unrefined part of $\mathcal{P}_1$), then $\mu_{\Gamma_1}$ is nonzero
      on the corresponding component of $\Gamma_1$.
  \end{enumerate}
\end{enumerate}
\end{lemma}

\begin{proof}
This all follows from considering the 2-layer tree $\Gamma_{\mathcal{P}_1 \prec \mathcal{P}_2}$ associated to $\mathcal{P}_1 \prec \mathcal{P}_2$ and the face contractions to $\pi_1 : \Gamma_1 \to \Gamma_{\mathcal{P}_1 \prec \mathcal{P}_2}$ and $\pi_2 : \Gamma_2 \to \Gamma_{\mathcal{P}_1 \prec \mathcal{P}_2}$. See Figure \ref{fig:thirdlemma}.
\end{proof}

\begin{figure}[h]
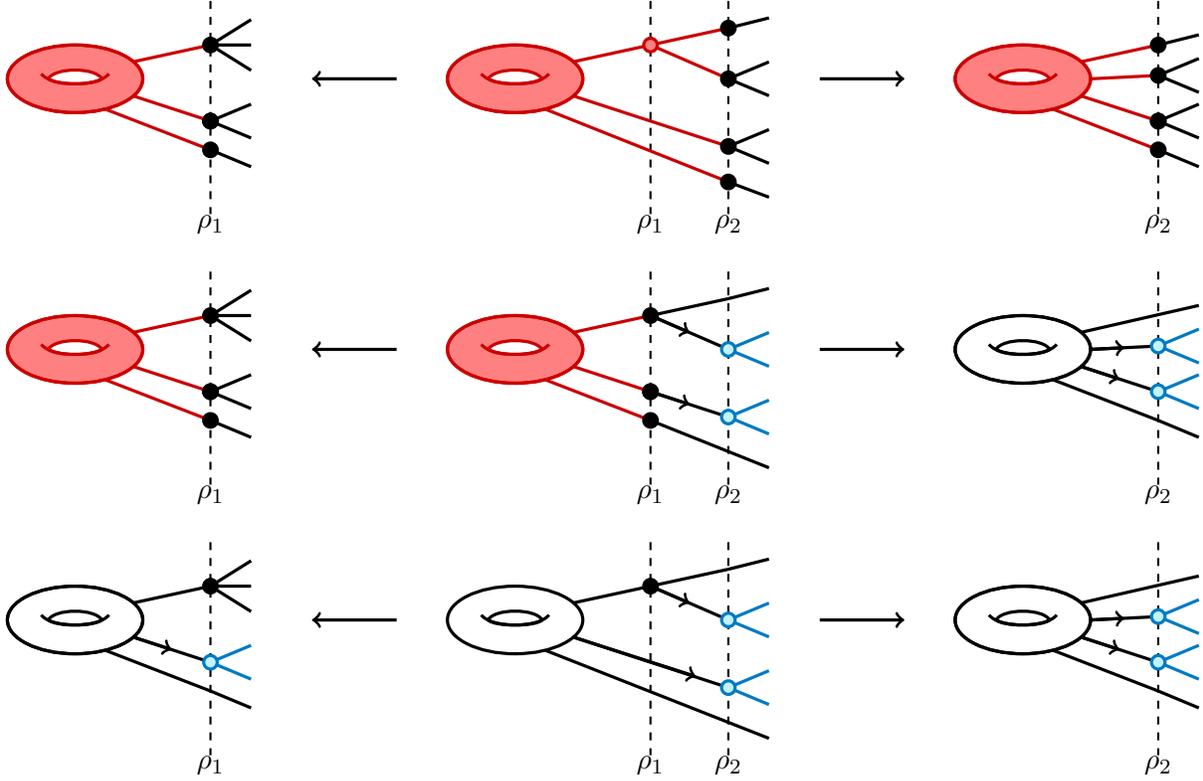

    \ThirdLemma
    \caption{Contraction data on 2-layer trees (middle column) enforce conditions on 1-layer trees (left column and right column). The left column is obtained from the middle column by sending $e_2 = \rho_2 - \rho_1 \to 0$, and the right column is obtained from the middle column by sending $e_1 = \rho_1\to 0$. The middle entry of each row corresponds respectively to Case I-I (top), Case I-II (middle), and Case II-II (bottom) from Lemma \ref{lem:2_layer_contraction_datum}.}
    \label{fig:thirdlemma}
\end{figure}

\begin{proposition}\label{prop:uni_contract_equiv_QK}
  A universal contraction datum is equivalent to the choice of
  \begin{enumerate}
    \item A downward closed proper subset $Q$ of the poset of partitions of $[n]$
    \item A simplicial complex $\calK$ on $[n]$
  \end{enumerate}
  such that $Q$ and $\calK$ do not overlap (Definition \ref{def:do_not_overlap}).
\end{proposition}
\begin{proof}
We construct an explicit one-to-one correspondence.

Given a universal contraction datum $(\rho_\Gamma, \mu_\Gamma)_{\Gamma \in I}$, we associate to it the set $Q$ of partitions $\mathcal{P}$ so that the corresponding 1-layer tree $\Gamma$ has nonzero radius $\rho.$ By \cite[Proposition 4.12]{bkn_qstable}, this set $Q$ is downward closed.

We construct the simplicial complex $\calK$ as follows. If $S$ is a subset of $[n]$ containing at least two elements, let $\mathcal{P}(S)$ be the partition of $[n]$ in which $S$ is the only large part. Consider the corresponding 1-layer tree, $\Gamma_{\mathcal{P}(S)}$. Observe that $\Gamma_{\mathcal{P}(S)}$ has a unique rational component, and that component is marked by $S$. If $\mu_{\Gamma_{\mathcal{P}(S)}}$ is nonzero, set $S \in \calK$. Else, set $S \not\in \calK$. By construction, $\mathcal{P}(S)$ does not belong to $Q$, so $Q$ and $\calK$ do not overlap.

To see that $\calK$ is downward closed, suppose $S$ has at least two elements, $S \subsetneq T$, and $T \in \calK$ according to our construction. Then $\mathcal{P}(T) \prec \mathcal{P}(S)$. By Lemma \ref{lem:1_layer_tree_relations}, since $\mu_{\Gamma_{T}}$ is nonzero on the component corresponding to $T$, $\mu_{\Gamma_S}$ must be nonzero on the component corresponding to $S$. Therefore $S \in \calK$, as required.

For the other half of the correspondence, suppose that a $Q$ and $\calK$ are given as in the statement of the proposition. By \cite[Proposition 4.12]{bkn_qstable}, the choice of $Q$ determines a universal radius $(\rho_\Gamma)$. It remains to choose a tail function $\mu_\Gamma \in \PL(\tilde{\Gamma})$ for each $\Gamma \in I$. By Lemma \ref{lem:tail_function_determined_by_support}, it suffices to identify the rational tails of $\tilde{\Gamma}$ on which we wish $\mu_\Gamma$ to be supported. For each rational tail $T$ of $\tilde{\Gamma}$ not intersecting with $\ol{\Delta}_\Gamma$, we set $T$ to be in the support of $\mu_\Gamma$ precisely when the markings incident to $T$ belong to $\calK$.

Having fixed the data of a universal contraction datum, we must verify that
\begin{enumerate}
    \item for each $\Gamma \in I$, $\ol{\Delta}_\Gamma \cap |\mu_\Gamma| = \emptyset$;
  \item for each face contraction $\pi : \Gamma \to \Gamma'$, $\tilde{\pi}^*(\mu_{\Gamma'}) = \mu_{\Gamma}$.
\end{enumerate}

The first item holds by construction.

For the latter, we know from Lemma \ref{lem:tail_function_weighted_edge_contraction} that $\tilde{\pi}^*(\mu_{\Gamma'})$ is a tail function. By Lemma \ref{lem:tail_function_determined_by_support}, to show $\mu_\Gamma = \tilde{\pi}^*(\mu_{\Gamma'})$ it suffices to show that the supports of $\mu_\Gamma$ and $\tilde{\pi}^*(\mu_{\Gamma'})$ are equal. By Lemma \ref{lem:tail_function_weighted_edge_contraction} again, the support of $\tilde{\pi}^*(\mu_{\Gamma'})$ consists of the vertices $v$ of $\tilde{\Gamma}$ so that $\ul{\tilde{\pi}}^{-1}(v) \subseteq |\mu_{\Gamma'}|$.

By Lemma \ref{lem:subset_of_tails} and the construction of $\mu_{\Gamma'}$, the support of $\tilde{\pi}^*(\mu_{\Gamma'})$ consists equivalently of the vertices $v$ of $\tilde{\Gamma}$ so that $\ul{\tilde{\pi}}^{-1}(v)$ belongs to a rational tail $T'$ of $\Gamma'$ with $\Marks(T') \in \calK$ and $T' \cap \ol{\Delta}_{\Gamma'} = \emptyset$.

Let $v$ be an arbitrary vertex of $\tilde{\Gamma}$. If $v$ belongs to $\ol{\Delta}_{\Gamma}$, then $\ul{\tilde{\pi}}^{-1}(v)$ intersects $\ol{\Delta}_{\Gamma'}$, so neither $\mu_\Gamma$ nor $\tilde{\pi}^*(\mu_{\Gamma'})$ contains $v$ in its support. If $v$ does not belong to $\ol{\Delta}_{\Gamma}$, then $v$ is genus 0 and there is a rational tail $T_v$ of $\tilde{\Gamma}$ rooted at $v$. Write $J$ for the markings on $T_v$. In this case $T_v' \coloneqq \ul{\tilde{\pi}}^{-1}(T_v)$ is also a $J$-marked rational tail, and it does not intersect $\ol{\Delta}_{\Gamma'} \subseteq \ul{\tilde{\pi}}^{-1}(\ol{\Delta}_{\Gamma})$. Observe that $T_v$ is the minimal rational tail containing $v$ and $T_v'$ is the minimal rational tail containing $\ul{\tilde{\pi}}^{-1}(v)$. Then
\begin{align*}
    & \text{$v$ belongs to the support of $\mu_\Gamma$} \\
    \iff & \text{$T_v$ belongs to the support of $\mu_\Gamma$} \\
    \iff & \text{$J \in \calK$} \\
    \iff & \text{$T_v'$ belongs to the support of $\mu_{\Gamma'}$} \\
    \iff & \text{$\ul{\tilde{\pi}}^{-1}(v)$ belongs to the support of $\mu_{\Gamma'}$} \\
    \iff & \text{ $v$ belongs to the support of $\tilde{\pi}^*(\mu_{\Gamma'})$.}
\end{align*}
We conclude that (ii) holds.

The data of $Q$ and $\calK$ are determined by the contraction data on 1-layer trees which we know determines a universal contraction datum, and the constructions we have given are clearly inverses of each other on the restriction to 1-layer trees.
\end{proof}

\begin{theorem} \label{thm:contraction_to_qk_nonintro}
For each $Q, \calK$ as above, there is a diagram of birational morphisms of algebraic stacks
\[
  \begin{tikzcd}
   & \Moduli_{1,n}^{rad} \ar[dl] \ar[dr] & \\
   \Moduli_{1,n} & & \Moduli_{1,\mathcal{K}}(Q)
  \end{tikzcd}
\]
restricting to the identity on $\moduli_{1,n}$.
\end{theorem}
\begin{proof}
Apply Proposition \ref{prop:contract_everything} to the universal curve of $\Moduli_{1,n}^{rad}$, then observe that the resulting curve belongs to $\Moduli_{1,\mathcal{K}}(Q)$.
\end{proof}

\begin{example} \label{ex:different_collision_limits}
    Let $\moduli = \Moduli_{1,\calK}(Q_1)$ and $\moduli' = \Moduli_{1,\calK}(Q_2)$ where
    \begin{align*}
        \calK &= \{ \emptyset, \{1\},\{2\}, \{3\}, \{4\}, \{1,2\} \}, \\
        Q_1 &= \{ 1234, 12/34 \}, \\
        Q_2 &= \{ 1234, 12/34, 1/234, 2/134, 1/2/34 \}.
    \end{align*}
    Let $\Gamma$ be the two layer tree associated to the chain of partitions $12/34 \prec 1/2/34$, and let $\pi : C \to S$ be a 1-parameter smoothing of $C$ in $\Moduli_{1,n}^{rad}$. Apply Theorem \ref{thm:contraction_to_qk_nonintro} for $(Q_1,\calK)$ and $(Q_2,\calK)$ to obtain limits in $\moduli$ and $\moduli'$ of the generic member of $C$. The limit in $\moduli$ is a tacnode with the markings indexed by 1 and 2 colliding on one branch and markings 3 and 4 distinct on the other branch. The limit in $\moduli'$ is a planar triple point with first branch marked by 1, the second branch marked by 2, and the third branch marked by the 3 and 4 (see Figure \ref{fig:different_collision_limits}). Even though the two spaces have the same collision complex, the two limits have different collections of colliding points.

    \begin{figure}[h]
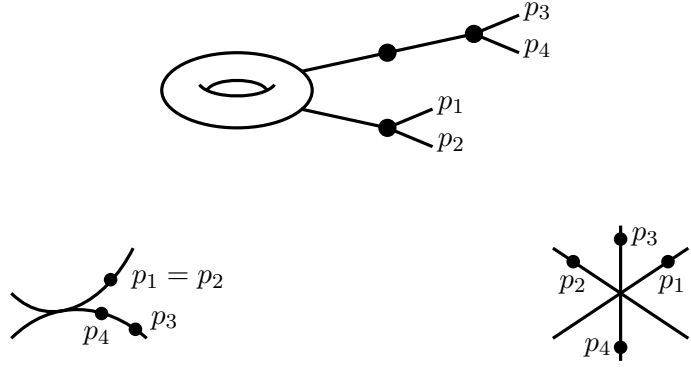
 
    \DifferentCollisionLimits
    \caption{
    The two layer tree associated to the chain of partitions $12/34 \prec 1/2/34$ (top). Below are the respective limits in $\moduli$ (left) and $\moduli'$ (right).
    } \label{fig:different_collision_limits}
\end{figure}
\end{example}

\section{Genus one compactifications admitting collisions}

\subsection{Construction of $(Q,\calK)$-stable spaces}\label{sec:construction}

In this section we prove that the stacks $\Moduli_{1,\calK}(Q)$ are modular compactifications of $\Moduli_{1,n}$. They are by definition substacks of $\mathcal{U}_{1,n}$, so it remains to show that
\begin{enumerate}
  \item $(Q,\calK)$-stable curves are smoothable (so that they lie in $\mathcal{V}_{1,n}$);
  \item $\Moduli_{1,\calK}(Q)$ is an open substack of $\mathcal{V}_{1,n}$;
  \item $\Moduli_{1,\calK}(Q)$ is proper.
\end{enumerate}
We establish each in turn, leveraging the work of \cite{smyth_mstable} and \cite{bkn_qstable}.

We begin by showing that, like simplicially stable spaces, the $(Q,\calK)$-stable spaces have a ``Frankenstein property": each is a regluing of Smyth's $(m,\calA)$-stable spaces \cite[Definition 3.7]{smyth_mstable}. As a base case, each $(m,\calA)$-stable space is a $(Q, \calK)$-stable space for an appropriate choice of $Q$ and $\calK$.

\begin{lemma} \label{lem:ma_to_qk}
For each integer $0 \leq m < n$ and weight vector $\calA = (a_1,\dots,a_n)$ there is a $Q \in \mathfrak{Q}_n$ and simplicial  complex $\calK$ on $[n]$ so that $Q$ and $\calK$ do not overlap and
\[
  \Moduli_{1,\calA}(m) = \Moduli_{1,\calK}(Q).
\]
\end{lemma}
\begin{proof}
Start by setting
\[
  Q = \{ \mathcal{P} \in \PPart([n]) \stc |P| \leq m \}.
\]
If $I \subseteq [n]$, write $\mathcal{P}(I)$ for the partition of $[n]$ given by
\[
  \mathcal{P}(I) = \{ I \} \cup \{ \{ j \} \stc j \in [n] - I \}.
\]
Then set $\calK$ to the set
\[
  \calK = \left\{ I \subseteq [n] \stc \sum_{i \in I} a_i \leq 1 \text{ and } \mathcal{P}(I) \not\in Q \right\}.
\]
It is straightforward to check that these choices result in the desired equality.
\end{proof}

For the remainder of this section fix a positive integer $n$, a $Q \in \mathfrak{Q}_n$, and a simplicial complex $\calK$ on $[n]$ so that $Q$ and $\calK$ do not overlap.

\begin{lemma}\label{lem:qk_to_ma}
For every $(Q,\calK)$-stable curve $(C;p_1,\dots,p_n)$, there exists an integer $m$ with $0 \leq m < n$ and weight data $\calA = (a_1,\dots,a_n)$ such that $(C;p_1,\dots,p_n)$ is $(m,\calA)$-stable in the sense of \cite[Definition 3.7]{smyth_mstable}.
\end{lemma}
\begin{proof}
Suppose $(C;p_1,\dots,p_n)$ is $(Q,\calK)$-stable. Let $Z\subseteq C$ be the minimal subcurve of genus one, and define $m := |\lev(Z)|$. For each $i \in [n]$, define $a_i := \frac{1}{N_i}$ where $N_i = |\{p_j \stc p_j = p_i \in C, j\in [n]\}|$, and set $\calA := (a_1,\dots,a_n)$. It is easily verified that $(C;p_1,\dots,p_n)$ is $(m,\calA)$-stable.
\end{proof}

\begin{theorem} [{{Deformation openness}}] \label{thm:qk_def_open} 
The stack $\Moduli_{1,\calK}(Q)$ is an open substack of $\mathcal{U}_{1,n}$
\end{theorem}
\begin{proof}
The proof is identical to \cite[Theorem 5.2]{bkn_qstable} with the additional observation that axioms (Q1), (Q4), and (Q5) of Definition \ref{def:qk_stable} are open conditions.
\end{proof}

\begin{theorem} [{{Frankenstein property}}] \label{thm:qk_frankenstein} 
There is a finite open cover $\{ U_i \}_{i \in I}$ of $\Moduli_{1,\calK}(Q)$ so that for each $i \in I$, there exists a weight vector $\calA_i$ and integer $m_i$ such that $U_i$ is an open substack of $\Moduli_{1,\calA_i}(m_i)$.
\end{theorem}
\begin{proof}
By Theorem \ref{thm:qk_def_open}, the stack $\Moduli_{1,\calK}(Q)$ and the stacks $\Moduli_{1,\calA_i}(m_i)$ are open substacks of $\mathcal{U}_{1,n}$. By Lemma \ref{lem:qk_to_ma}, each curve of $\Moduli_{1,\calK}(Q)$ belongs to some $\Moduli_{1,\calA}(m)$, so we may argue just as in Theorem \ref{thm:simplicial_frankenstein} to conclude.
\end{proof}

\begin{lemma} \label{lem:qk_smoothable}
The stack $\Moduli_{1,\calK}(Q)$ factors through $\mathcal{V}_{1,n}$
\end{lemma}
\begin{proof}
A curve is smoothable if and only if its singularities are smoothable. Nodes and elliptic $m$-fold points are smoothable, so the curves of $\Moduli_{1,\calK}(Q)$ are smoothable.
\end{proof}


\begin{proposition} \label{prop:qk_univ_closed}
The stack $\Moduli_{1,\calK}(Q)$ is universally closed.
\end{proposition}
\begin{proof}
For any 1-parameter family of curves in $\moduli_{1,n}$, we may find a limit in $\Moduli_{1,n}^{rad}$ and then apply Theorem \ref{thm:contraction_to_qk_nonintro} to produce a $(Q,\mathcal{K})$-stable limit (cf. {{\cite[Theorem 5.3]{bkn_qstable}}}).
%
\end{proof}

\begin{proposition} \label{prop:qk_separated}
The stack $\Moduli_{1, \calK}(Q)$ is separated.
\end{proposition}
\begin{proof}
Let $S$ be the spectrum of a DVR, $s \to S$ a geometric special point, and $\eta \to S$ its
generic point. Suppose that $\pi : C \to S$ and $\pi' : C' \to S$ are families of 
$(Q,\mathcal{K})$-stable curves whose restrictions to $\eta$ are isomorphic and factor through $\moduli_{1,n}$. As in \cite[Section 3.3.2]{smyth_mstable}, it suffices
to show that the isomorphism over generic fibers extends to a regular isomorphism over $S$ after a finite base change.

After a finite base change if necessary, we may construct a semistable curve $C^{ss}$ with regular total space over $S$ dominating $C$ and $C'$:
\[
\begin{tikzcd}
 & C^{ss} \ar[dr, "\phi'"] \ar[dl, "\phi"'] & \\
 C \ar[dr, "\pi"'] & & C' \ar[dl, "\pi'"] \\
  & S
\end{tikzcd}
\]

Let $Z$ be the exceptional locus of $\phi$ and $Z'$ the exceptional locus of $\phi'$. Since $C^{ss}$, $C$, and $C'$ are normal, it suffices to show that $Z = Z'$.
As in \cite[Theorem 5.6]{bkn_qstable}, either
\begin{enumerate}
    \item $C|_s$ and $C'|_s$ are both nodal, or
    \item there exists $p \in C|_s$, $p' \in C'|_s$ so that $p, p'$ are elliptic $l$-fold points with the same level whose preimages $\phi^{-1}(p) = (\phi')^{-1}(p')$ are the same balanced subcurve $E$ of $C^{ss}|_s$.
\end{enumerate}
Again as in \cite[Section 3.3.2]{smyth_mstable} we have
\begin{enumerate}
  \item $Z$ and $Z'$ contain no irreducible component of $C^{ss}|_s$ adjacent to $E$, and
  \item $Z$ and $Z'$ contain each irreducible component of $C^{ss}|_s$ not adjacent to $E$ with only two special points.
\end{enumerate}

It remains to check that each contains the same irreducible components of $C^{ss}$ not adjacent to $E$ with at least 3 special points. Suppose that $K$ is such a component so that $K \subseteq Z$ but $K \not\subseteq Z'$. Then since the only rational singularities that $C|_s$ may have are nodes, $Z$ must contain the entire rational tail $T$ rooted at $K$. With $\Sigma^{ss}$ is the divisor of markings on $C^{ss}$, the image $\phi(\Sigma^{ss} \cap T)$ is a point, so the markings corresponding to $\Sigma^{ss} \cap T$ coincide in $C$. Since $C$ is $(Q, \mathcal{K})$-stable, we must have $\Marks(\phi(T)) \in \calK$.

Since $K \not\subset Z'$, $\phi'(T)$ is a rational tail of $C'|_s$.  Now, $\Marks(\phi(T)) = \Marks(\phi'(T))$ belongs to $\mathcal{K}$. But axiom (Q5) of $(Q, \calK)$-stability (Definition \ref{def:qk_stable}) implies that $\Marks(\phi'(T)) \not\in \calK$, a contradiction. Therefore $Z'$ must also contain $K$ and the result follows.
\end{proof}

\begin{theorem} \label{thm:qk_is_modular_compactification_nonintro}
 The stack $\Moduli_{1,\calK}(Q)$ is a Deligne-Mumford Gorenstein modular compactification of $\moduli_{1,n}$ over $\Z[1/6]$ admitting collisions so that $\calK(\Moduli_{1,\calK}(Q)) = \calK$.
\end{theorem}

\begin{proof}
By Lemma \ref{lem:qk_smoothable} and Theorem \ref{thm:qk_def_open}, it is an open substack of $\mathcal{V}_{1,n}$. By Proposition \ref{prop:qk_univ_closed} and \ref{prop:qk_separated}, $\Moduli_{1,\calK}(Q)$ is proper. Therefore $\Moduli_{1,\calK}(Q)$ is a modular compactification of $\moduli_{1,n}$ over $\Spec \Z[1/6].$

We have that $\Moduli_{1,\calK}(Q)$ is Deligne-Mumford since, by Theorem \ref{thm:qk_frankenstein}, it admits an open cover by Deligne-Mumford stacks \cite[Theorem 3.8]{smyth_mstable}. By definition, the geometric points of $\Moduli_{1,\calK}(Q)$ correspond to Gorenstein curves with smooth markings, and $\calK(\Moduli_{1,\calK}(Q)) = \calK$.
\end{proof}

\subsection{Classification of genus one Gorenstein compactifications}\label{sec:classfication}

We now prove that all modular compactifications of $\moduli_{1,n}$ by Gorenstein curves with possibly colliding markings are classified by the $(Q,\calK)$-stable spaces (Theorem \ref{thm:qk_classification_nonintro}). The method of proof is an adaptation of the argument of \cite{bkn_qstable} to the case of non-colliding markings.

\begin{definition}
Let $\mathcal{W}_{1,n}$ be the moduli stack over $\Spec \Z[1/6]$ parametrizing all families of $n$-pointed Gorenstein curves with smooth, but possibly colliding markings, and no infinitesimal automorphisms, i.e. $H^1\left(C, \omega_C\left(-\sum_{i=1}^n p_i \right)\right) = 0$.
\end{definition}

Let $\moduli$ be a proper open substack of $\mathcal{W}_{1,n}$. Then more precisely, the goal of this section is to show that $\moduli = \Moduli_{1,\mathcal{K}}(Q)$ for some $Q$ and $\mathcal{K}$.

Our strategy is to construct a universal contraction datum inducing a map $\Moduli_{1,n}^{rad} \to \moduli.$ As a first step we break $\mathcal{W}_{1,n}$ into a stratification by ``combinatorial type," and show that $\moduli$ is a union of such strata. Then to decide whether a given curve belongs to $\moduli$, we need only to know its combinatorial type.

\begin{definition} \label{def:combinatorial_type}
Let $(C; p_1,\ldots, p_n)$ be a Gorenstein $n$-pointed curve of arithmetic genus one with smooth markings. The \emphbf{combinatorial type of $C$}, denoted by $\Lambda$, consists of the following data:
\begin{enumerate}
 \item a set $V$ of \emphbf{vertices}, identified with the set of irreducible components of $C$;
 \item a set $E$ of \emphbf{singularities}, identified with the set of singular points of $C$;
 \item a \emphbf{genus function} $g : V \cup E \to \N$ taking an irreducible component of $C$ to the genus of its normalization and taking each singularity of $C$ to its genus as a singularity;
 \item an \emphbf{incidence function} $i : V \times E \to \{ 0, 1 \}$ taking $(v,e) \mapsto 1$ if $e \in v$ and $0$ otherwise;
 \item a \emphbf{marking function} $s$ which assigns to a component $v\in V$ the partition of $\Marks(v)$ given by the subsets $\Marks(x)$ as $x$ varies over the marked points of $v$.
\end{enumerate}
Observe that for a given combinatorial type, the union of the sets $s(v)$ for $v \in V$ is a partition of $\{ 1, \ldots, n \}$.

Two combinatorial types $\Lambda_1 = (V_1, E_1, g_1, i_1, s_1)$ and $\Lambda_2 = (V_2, E_2, g_2, i_2, s_2)$ are \emphbf{isomorphic} if there is a bijection $f : V_1 \cup E_1 \to V_2 \cup E_2$ so that
\begin{enumerate}
  \item $f(V_1) = V_2$ and $f(E_1) = E_2$;
  \item $g_1 = g_2 \circ f$;
  \item $s_1 = s_2 \circ f$;
  \item $i_1(v,e) = i_2(f(v), f(e))$ for all $(v,e) \in V_1 \times E_1$.
\end{enumerate}
\end{definition}

Given a combinatorial type $\Lambda$ there is a natural locally closed substack $\mathcal{Z}_\Lambda$ of $\mathcal{W}_{1,n}$ whose geometric points are the curves with
combinatorial type isomorphic to $\Lambda$.

\begin{lemma}
$\mathcal{W}_{1,n}$ is a union of the open substacks of weighted pointed $m$-stable curves $\Moduli_{1,\mathcal{A}}(m)$.
\end{lemma}
\begin{proof}
Since each of the stacks $\Moduli_{1,\mathcal{A}}(m)$ is an open substack of $\mathcal{W}_{1,n}$, it suffices to show that each curve $(C; p_1,\ldots, p_n)$ in $\mathcal{W}_{1,n}$ belongs to $\Moduli_{1,\calA}(m)$ for some weights $\calA$ and integer $m$. This follows exactly as in 
Lemma \ref{lem:qk_to_ma}.
\end{proof}

\begin{lemma}
Given a combinatorial type $\Lambda$, the locus $\mathcal{Z}_{\Lambda}$ is irreducible.
\end{lemma}
\begin{proof}
The idea is to identify $\mathcal{Z}_\Lambda$ with $\mathcal{Z}_{\ol{\Lambda}}$ where $\ol{\Lambda}$ is a combinatorial type with no colliding markings. More formally,
let $r$ be the number of sets in the partition $\bigcup\limits_{v \in V} s(v)$ of $[n]$. Find a bijection $\eta: \bigcup\limits_{v \in V} s(v) \to \{ \{ 1 \},\dots, \{r\} \}$. 
Let
$\ol{\Lambda}$ be the combinatorial type of $r$-pointed curve obtained from $\Lambda$ with marking function $\eta \circ s$. 
Then clearly, $\mathcal{Z}_{\Lambda} \cong \mathcal{Z}_{\ol{\Lambda}}$. By \cite[Lemma 6.3]{bkn_qstable}, $\mathcal{Z}_\Lambda$ is irreducible.
\end{proof}

\begin{lemma} \label{lem:union_of_loci}
$\moduli$ is a union of $\mathcal{Z}_{\Lambda}$s.
\end{lemma}

\begin{proof}
This follows exactly as in \cite[Lemma 6.4]{bkn_qstable}, with weighted $m$-stable spaces $\Moduli_{1,\mathcal{A}}(m)$ replacing $\Moduli_{1,n}(m)$.
\end{proof}

Our next task is to construct a universal contraction datum inducing a map $\Moduli_{1,n}^{rad} \to \moduli$. We start by constructing contraction data for small slices of $\Moduli_{1,n}^{rad}$, which we call test curves.

\begin{definition} [{{\cite[Definition 6.5]{bkn_qstable}}}]
Let $\Gamma$ be a basic stable radially aligned tropical curve. A \emphbf{$\Gamma$-test curve centered at a geometric point $s$ of $S$} consists of a family of radially aligned curves $\pi : C \to S$ and an isomorphism $\trop(C|_s) \cong \Gamma$ (suppressed in later notation) such that
\begin{enumerate}
  \item $(S,s)$ is an atomic neighborhood for $\pi : C \to S$.
  \item The log structure on $S$ is divisorial; that is, it is the log structure associated to a normal crossings divisor \cite[(1.5)]{kato_log_structures}.
\end{enumerate}
\end{definition}

\begin{lemma} \label{lem:test_curves_exist}
For any basic stable radially aligned tropical curve $\Gamma$, there is a $\Gamma$-test curve.
\end{lemma}
\begin{proof}
Choose a stable radially aligned log curve $C_0$ with tropicalization $\Gamma$. Choose a strict \'etale neighborhood $(S,s)$ of $C_0$ in $\Moduli_{1,n}^{rad}$. Then shrink $S$ if necessary to obtain an atomic neighborhood using Theorem \ref{thm:uniform_charts}. The log structure on $\Moduli_{1,n}^{rad}$ is divisorial, which is an \'etale local property, so the same is true of $(S,s)$.
\end{proof}

The reason we ask that the log structure is divisorial is so that each $\Gamma$-test curve contains test curves for the weighted edge contractions of $\Gamma$ as explained in the lemma below.

\begin{lemma} \label{lem:test_curve_structure}
Let $\pi : C \to S$ be a $\Gamma$-test curve centered at a geometric point $s$ of $S$. Let $\tau : \Gamma' \to \Gamma$ be any face contraction. Then there is an open subscheme $S'$ of $S$ and a geometric point $s'$ of $S'$ so that
\begin{enumerate}
  \item $\pi|_{S'}$ is a $\Gamma'$-test curve centered at $s'$
  \item $\tau$ is the composite of the canonical face contraction $\trop(C|_{s'}) \to \trop(C|_{s})$ with the isomorphisms $\Gamma' \to \trop(C|_{s'})$ and $\trop(C|_s) \to \Gamma$.
\end{enumerate}
\end{lemma}
\begin{proof}
 Write $\bigoplus_{i = 1}^k \N e_i$ for $\ol{M}_{S,s} \cong \Gamma(S, \ol{M}_S)$. Since the log structure on $S$ is divisorial, $S$ possesses a stratification by non-empty locally closed subsets $\{ W_I \}$ indexed by subsets $I \subseteq \{ 1, \ldots, k \}$, where
\[
  W_I = \bigcap_{i \in I} |e_i| \cap \bigcap_{\ell \in I^c} (S - |e_\ell|),
\]
and $|e_i|$ denotes the locus in $S$ where $e_i \neq 0$. 
Let $J$ be the set of indices $j$ such that $e_j$ does not go to zero under $\tau^\sharp$. Choose $s'$ as any point of $W_J$, and let $S' = \bigcup_{I \subseteq J} W_I$.
\end{proof}

\begin{notation} \label{not:notation}
Let $\pi : C \to S$ be a $\Gamma$-test curve centered at $s$. For each contraction datum $D = (\rho, \mu)$ on $\pi$, let $\tilde{C}_D \to S$ be the subdivision of $C$ where $\lambda = \rho$, let $\ol{C}_{D} \to S$ be the contraction of $\tilde{C}_D \to S$ associated to $D$, and let $\Lambda_D$ be the combinatorial type of the fiber of $\ol{C}_D$ over $s$.
\end{notation}

\begin{lemma} \label{lem:contracted_family_vs_contracted_fiber}
With Notation \ref{not:notation},
for each contraction datum $D$ on $\pi$,
\[
  \moduli(s) \text{ contains } \ol{C}_D|_s \, \iff \, \moduli(S) \text{ contains } \ol{C}_D.
\]
\end{lemma}

\begin{proof}
The backward direction is clear. For the forward direction, assume that for some contraction datum $D$, $\ol{C}_D|_s$ belongs to $\moduli$. This follows from the openness of $\moduli$ in $\mathcal{W}_{1,n}$. As in the proof of the previous lemma, write $\bigoplus_{i=1}^k \N e_i$ for $\ol{M}_{S,s} \cong \Gamma(S, \ol{M}_S)$ and let
\[
  W_I = \bigcap_{i \in I} |e_i| \cap \bigcap_{\ell \in I^c} (S - |e_\ell|)
\]
for each $I \subseteq \{1, \ldots, k\}$.
Since $C \to S$ satisfies the conclusions of Theorem \ref{thm:uniform_charts}, for each $I$, the tropicalizations of the fibers of $C|_{W_I} \to W_I$ are constant. It follows that the same is true for the combinatorial types of the fibers of $\ol{C}_D|_{W_I} \to W_I$. Since the log structure of $S$
is divisorial, for each $I \subseteq \{1, \ldots, k \}$, there is a generalization $\eta_{I} \to s$ where $\eta_I \in W_I$. Since $\mathcal{M}$ is open in $\mathcal{W}_{1,n}$, $\moduli$
is closed under generalization. Since $\mathcal{M}$ contains $\ol{C}_D|_s$, $\mathcal{M}$ must also contain $\ol{C}_D|_{\eta_I}$. Then, since
membership in $\mathcal{M}$ is determined by combinatorial type, $\mathcal{M}$ contains all of $\ol{C}_D|_{W_I}$. We conclude that the whole
family $\ol{C}_D \to S$ belongs to $\mathcal{M}$.
\end{proof}

Observe that if $\rho = \rho_i \neq 0$, then $\Lambda_D$ possesses an elliptic $|P_{i+1}|$-fold point (taking $P_{i + 1} = 1/2/\cdots/n$ if $i = k$). On the other hand, if $\rho = 0$, then $\Lambda_D$ possesses no elliptic singularities. Then if $D$ and $D'$ are contraction data with different radii, $\Lambda_D \not\cong \Lambda_{D'}$. Similarly, if $D$ and $D'$ have the same radius, but different tail functions, then without loss of generality $\Lambda_D$ possesses an $I$-marked rational tail where $\Lambda_{D'}$ possesses an $I$-marked leg. We summarize in the following lemma.

\begin{lemma} \label{lem:contracted_curves_distinct}
With Notation \ref{not:notation}, and contraction data $D, D'$ on $C \to S$, we have $\Lambda_D \cong \Lambda_{D'}$ if and only if $D = D'$. \qed
\end{lemma}

\begin{proposition} \label{prop:exactly_one_contraction}
With Notation \ref{not:notation}, $\moduli$ contains exactly one of the families $\ol{C}_D \to S$.
\end{proposition}
\begin{proof}
By Lemma \ref{lem:contracted_curves_distinct} and the separatedness of $\moduli$, $\moduli$ can contain at most one of the families $\ol{C}_D$. By Lemma \ref{lem:contracted_family_vs_contracted_fiber}, we only need to show that there exists a $D$ so that the central geometric fiber of $\ol{C}_D$ belongs to $\moduli$. Note that there are finitely many contraction data $D$ on $\pi$ since there are only finitely many such contraction data on $\trop(C|_s)$.

Choose $t$ to be a point of $S$ generizing $s$ over which $C$ is smooth. Let $X \to S$ be a map from the spectrum of a DVR with special point $x$ mapping to $s$ and generic point $\eta$ mapping to $t$. 

Replacing $X$ by a finite base change if necessary, find the limit curve $C_{\mathcal{M}} \to X$ in $\mathcal{M}$ of the smooth curve $C|_\eta$. After a second base change if necessary, pick a regular family of semistable curves $C^{ss} \to X$ dominating $C_{\mathcal{M}} \to X$, each of the families $\ol{C}_D|_X \to X$, and each of the families $\tilde{C}_D|_X \to X$.
Let $E$ be the exceptional locus of $\phi : C^{ss} \to C_{\mathcal{M}}$ and let $E_D$ be the exceptional locus of $\phi_D : C^{ss} \to \ol{C}_D|_X$ for each $D$. By \cite[Lemma 2.7]{smyth_zstable}, we have 
\begin{enumerate}
  \item $\phi$ (resp. $\phi_D$) is surjective with connected fibers;
  \item $\phi$ (resp. $\phi_D$) is an isomorphism on $C^{ss} - E$ (resp. $C^{ss} - E_D$);
  \item if $Z$ is a connected component of $E$ (resp. $E_D$), then $\phi(Z)$ (resp. $\phi_D$) has singularity genus equal to the arithmetic genus of $Z$ and number of branches equal to $|Z \cap \ol{C - Z}|$.
\end{enumerate}
If $C_\moduli|_x$ possesses an elliptic $m$-fold point $p$, by
\cite[Proposition 2.12]{smyth_mstable}, $\phi^{-1}(p)$ is a balanced subcurve of $C^{ss}|_x$, with
$\phi^{-1}(p)$ consisting of all components of $C^{ss}|_x$ whose nodal distance from the core of $C^{ss}|_x$
is less than some integer $\ell$. In this case, set $E^{bal} = \phi^{-1}(p)$. If $C_\moduli|_x$ has no elliptic $m$-fold points, set $E^{bal} = \emptyset$.

Next, let $T$ be the union of the pre-images $\phi^{-1}(q)$ of the points $q \in C_\moduli|_x$ so that $\Marks(q)$ has size at least two. The connected components of $T$ must be rational tails, since each $q$ has singularity genus 0 and 1 branch.

Since $C_\moduli|_x$ has no infinitesimal automorphisms, $E$ contains all rational irreducible components of $C^{ss}|_x$ with fewer than three special points not adjacent to $E^{bal}$. Write $E^{ss}$ for the union of these components not contained in $T$. Observe that a component adjacent to $T$ and in the exceptional locus must have the same image under $\phi$ as $T$ and is therefore part of one of the pre-images constituting $T$. It follows that the components of $E^{ss}$ are also not adjacent to $T$.

We claim that $E = E^{bal} \cup E^{ss} \cup T$. If $Z$ is some other irreducible component of $C^{ss}$ belonging to $E$, then $Z$ has at least three special points. At least two of these are nodes since $Z$ cannot be a tail (which would imply $Z \subseteq T$), and $Z$ cannot be all of $C^{ss}$. Further, $Z$ is not adjacent to $E^{bal}$ or $T$, since otherwise $Z$ would belong to $E^{bal}$ or $T$, following the argument in the previous paragraph. Let $\hat{Z}$ be the connected component of $E$ to which $Z$ belongs. Then $\phi(\hat{Z})$ is a rational singularity with either at least three branches or at least two branches and a marking. Neither is possible in $C_\moduli$, so $E = E^{bal} \cup E^{ss} \cup T$.

Repeat the construction for each contraction datum $D$ to obtain decompositions $E_D = E^{bal}_D \cup E^{ss}_D \cup T_D$.

Now, arguing as in \cite[Lemma 6.8]{bkn_qstable}, there is a radius $\rho_\moduli$ of $\Gamma$ so that $E^{bal} = E^{bal}_D$ for all contraction data $D = (\rho, \mu)$ with $\rho = \rho_\moduli$. Let $\tilde{C} \to S$ be the subdivision of $C$ at $\lambda = \rho_\moduli.$

Next, we claim that the image of $T$ in $\tilde{C}|_X$ is a disjoint union of stable rational trees. To see this, observe that $\tilde{C}|_X$ has distinct smooth markings and the only semistable components of $\tilde{C}|_X$, if any, are adjacent to the locus where $\lambda < \rho_\moduli$. Let $\mu_\moduli$ be the tail function on $\tilde{C}|_X$ with support the image of $T$ and let $D_\moduli$ be the unique extension of the contraction datum $(\rho_\moduli|_X, \mu_\moduli)$ to $S$.

By construction, the exceptional loci of $\phi_{D_\moduli}$ and $\phi$ are the same. It follows that $C_\moduli|_s \cong \ol{C}_{D_\moduli}|_s$, as required.
\end{proof}

\begin{corollary}
With Notation \ref{not:notation}, $\moduli$ contains exactly one of the loci $\mathcal{Z}_{\Lambda_D}$.
\end{corollary}
\begin{proof}
This is immediate from Proposition \ref{prop:exactly_one_contraction} and Lemma \ref{lem:union_of_loci}.
\end{proof}

\begin{corollary} \label{cor:univ_datum_well_defined}
With Notation \ref{not:notation}, write $D(\pi)$ for the unique contraction datum so that $C_{D(\pi)} \to S$ lies in $\moduli$. Let $\pi' : C' \to S'$ be a second $\Gamma$-test curve, this time centered at $s'$. Then, after identifying $C'|_{s'}$ with $C|_s$ by any isomorphism, $\PL(D(\pi)) = \PL(D(\pi'))$. 
\end{corollary}
\begin{proof}
The combinatorial type $\Lambda_D$ only depends on the corresponding piecewise linear function. Since $\moduli$ contains only one locus $\mathcal{Z}_D$, the piecewise linear functions for both families must be the same.
\end{proof}

We conclude with the complete classification of Gorenstein modular compactifications of $\moduli_{1,n}$ admitting collisions as $(Q,\calK)$-stable spaces.

\begin{theorem} \label{thm:qk_classification_nonintro}
Let $\moduli$ be a Gorenstein modular compactification of $\moduli_{1,n}$ over $\Z[1/6]$ admitting collisions. Then there exists $Q \in \mathfrak{Q}_n$ and $\calK$ a simplicial complex on $[n]$ so that $Q$ and $\calK$ do not overlap and $\moduli = \Moduli_{1,\calK}(Q)$ as open substacks of $\mathcal{V}_{1,n} \times_{\Spec \Z} \Spec \Z[1/6]$.
\end{theorem}

\begin{proof}
By Corollary \ref{cor:univ_datum_well_defined}, there is a well-defined contraction datum $D_\Gamma$ associated to each isomorphism class of basic stable radially aligned tropical curve $\Gamma$ so that for any $\Gamma$-test curve $C \to S$, the contracted curve $\ol{C}_{D_\Gamma} \to S$ belongs to $\moduli$. By Lemma \ref{lem:test_curve_structure}, these contraction data are compatible with edge contractions and hence define a universal contraction datum. Then by Proposition \ref{prop:uni_contract_equiv_QK}, this universal contraction datum is associated to a non-overlapping $Q$ and $\calK$. Moreover, the associated contraction $\ol{C} \to \Moduli_{1,n}^{rad}$ of the universal curve of $\Moduli_{1,n}^{rad}$ belongs to both $\moduli$ and $\Moduli_{1,\calK}(Q)$. The associated morphism $\tau : \Moduli_{1,n}^{rad} \to \mathcal{W}_{1,n}$ induced by $\ol{C} \to \Moduli_{1,n}^{rad}$ factors through both $\moduli$ and $\Moduli_{1,\calK}(Q)$.

We claim that the map $\tau$ surjects onto the geometric points of both $\moduli$ and $\Moduli_{1,\calK}(Q)$. The argument for each case identical, so we give it for $\moduli$. Let $C_0$ be a geometric point of $\moduli$. Since $\moduli_{1,n}$ is dense in $\moduli$, there is a scheme $S$ equal to the spectrum of a discrete valuation ring with geometric special point $s$, generic point $\eta$, and a 1-parameter smoothing $\pi : C \to S$ so that $C_0 = C|_s$ and $C|_\eta$ belongs to $\moduli_{1,n}$. After a finite base change if necessary, there is a limit radially aligned curve $\pi^{rad} : C^{rad} \to S$ in $\Moduli_{1,n}^{rad}$. Then $\tau(\pi^{rad}) : \ol{C^{rad}} \to S$ is a second 1-parameter family of curves in $\moduli$ with generic fiber $C|_\eta$. Since $\moduli$ is separated, we conclude $\ol{C^{rad}}|_s \cong C_0$, i.e., $C_0$ is in the image of $\tau$.

As $\Moduli_{1,\calK}(Q)$ and $\moduli$ have the same geometric points as open substacks of $\calW_{1,n}$, we conclude $\Moduli_{1,\calK}(Q) = \moduli$.
\end{proof}

\bibliographystyle{amsalpha}
\bibliography{biblio}

\providecommand{\bysame}{\leavevmode\hbox to3em{\hrulefill}\thinspace}
\providecommand{\MR}{\relax\ifhmode\unskip\space\fi MR }
\providecommand{\MRhref}[2]{%
  \href{http://www.ams.org/mathscinet-getitem?mr=#1}{#2}
}
\providecommand{\href}[2]{#2}
\begin{thebibliography}{MSvAX18}

\bibitem[ADGH20]{adgh-hassett}
Kenneth Ascher, Connor Dubé, Daniel Gershenson, and Elaine Hou,
  \emph{Enumerating hassett’s wall and chamber decomposition of the moduli
  space of weighted stable curves}, Experimental Mathematics \textbf{29}
  (2020), no.~1, 36--53.

\bibitem[AG08]{alexeevguy}
Valery Alexeev and G.~Michael Guy, \emph{Moduli of weighted stable maps and
  their gravitational descendants}, Journal of the Institute of Mathematics of
  Jussieu \textbf{7} (2008), no.~3, 425--456, MR 2427420.

\bibitem[{Bat}]{battistella_mstable}
Luca {Battistella}, \emph{{Modular compactifications of $\mathcal M_{2,n}$ with
  Gorenstein singularities}}, Algebra \& Number Theory, to appear; preprint
  arXiv:1906.06367.

\bibitem[BC]{battistella_carocci_stable_maps}
Luca Battistella and Francesca Carocci, \emph{A smooth compactification of the
  space of genus two curves in projective space - via logarithmic geometry and
  gorenstein curves}, Geometry \& Topology, to appear; preprint
  arXiv:2008.13506.

\bibitem[BC20]{bc_hassett}
Vance Blankers and Renzo Cavalieri, \emph{{Wall-Crossings for Hassett
  Descendant Potentials}}, International Mathematics Research Notices (2020),
  rnaa077.

\bibitem[BKN]{bkn_qstable}
Sebastian Bozlee, Bob Kuo, and Adrian Neff, \emph{A classification of modular
  compactifications of the space of pointed elliptic curves by {Gorenstein}
  curves}, Algebra and Number Theory, to appear; preprint arXiv:2105.10582.

\bibitem[Boz20]{bozlee_thesis}
Sebastian Bozlee, \emph{An application of logarithmic geometry to moduli of
  curves of genus greater than one}, Ph.D. thesis, University of Colorado,
  2020.

\bibitem[Boz21]{bozlee_contractions}
\bysame, \emph{Contractions of subcurves of families of log curves},
  Communications in Algebra \textbf{49} (2021), no.~11, 4616--4660.

\bibitem[CCUW20]{ccuw}
Renzo Cavalieri, Melody Chan, Martin Ulirsch, and Jonathan Wise, \emph{A moduli
  stack of tropical curves}, Forum of Mathematics, Sigma \textbf{8} (2020),
  e23.

\bibitem[DM69]{deligne_mumford}
Pierre Deligne and David Mumford, \emph{The irreducibility of the space of
  curves of given genus}, Publications Math\'ematiques de l'IH\'ES \textbf{36}
  (1969), 75--109.

\bibitem[Fry21]{fry_thesis}
Andy Fry, \emph{Moduli spaces of rational graphically stable curves}, Ph.D.
  thesis, Colorado State University, 2021.

\bibitem[FS10]{fedorchuk_smyth}
Maksym Fedorchuk and David~Ishii Smyth, \emph{Alternate compactifications of
  moduli spaces of curves}, Handbook of Moduli \textbf{24} (2010), 331--414.

\bibitem[Has03]{hassett_weighted_curves}
Brendon Hassett, \emph{Moduli spaces of weighted pointed stable curves},
  Advances in Mathematics \textbf{173} (2003), 316--352.

\bibitem[Kap93a]{kap1}
Mikhail Kapranov, \emph{Chow quotients of grassmannians i}, I. M. Gelfand
  Seminar, Advanced Soviet Mathematics (Providence, RI) (Sergei Gelfand and
  Simon Gindikin, eds.), vol.~16, Amer. Math. Soc., 1993.

\bibitem[Kap93b]{kap2}
\bysame, \emph{Veronese curves and {G}rothendieck-{K}nudsen moduli space
  {$\ol{M}_{0,n}$}}, Journal of Algebraic Geometry \textbf{2} (1993), no.~2,
  239–262.

\bibitem[Kat89]{kato_log_structures}
Kazuya Kato, \emph{Logarithmic structures of {F}ontaine-{I}llusie}, Algebraic
  Analysis, Geometry, and Number Theory (1989), 191--224.

\bibitem[Kat00]{fkato_deformations}
Fumiharu Kato, \emph{Log smooth deformation and moduli of log smooth curves},
  International Journal of Mathematics \textbf{11} (2000), no.~2, 215--232.

\bibitem[Kee92]{keel}
Sean Keel, \emph{Intersection theory of moduli space of stable n-pointed curves
  of genus zero}, Transactions of the American Mathematical Society
  \textbf{330} (1992), 545--574.

\bibitem[KM76]{knudsen_mumford_projectivity_I}
Finn Knudsen and David Mumford, \emph{The projectivity of the moduli space of
  stable curves. {I}, preliminaries on ``det'' and ``div''}, Math. Scand.
  \textbf{39} (1976), no.~1, 19--55.

\bibitem[Knu83a]{knudsen_projectivity_II}
Finn Knudsen, \emph{The projectivity of the moduli space of stable curves.
  {II}, the stacks {$M_{g,n}$}}, Math. Scand. \textbf{52} (1983), no.~2,
  161--199.

\bibitem[Knu83b]{knudsen_projectivity_III}
\bysame, \emph{The projectivity of the moduli space of stable curves. {III},
  the line bundles on {$M_{g,n}$}, and a proof of the projectivity of
  {$\overline{M}_{g,n}$} in characteristic 0}, Math. Scand. \textbf{52} (1983),
  no.~2, 200--212.

\bibitem[LM00]{losev_manin}
Andrey Losev and Yuri Manin, \emph{{New moduli spaces of pointed curves and
  pencils of flat connections.}}, Michigan Mathematical Journal \textbf{48}
  (2000), no.~1, 443 -- 472.

\bibitem[MSvAX18]{moon}
Han-Bom Moon, Charles Summers, James von Albade, and Ranze Xie,
  \emph{{Birational contractions of $\overline{\mathrm {M}}_{0,n}$ and
  combinatorics of extremal assignments}}, Journal of Algebraic Combinatorics
  \textbf{47} (2018), 51--90.

\bibitem[NV12]{borne_vistoli}
Borne Niels and Angelo Vistoli, \emph{Parabolic sheaves on logarithmic
  schemes}, Advances in Mathematics \textbf{231} (2012), no.~3-4, 1327--1363.

\bibitem[{OEI}]{oeis_simplicial_complexes}
{OEIS Foundation Inc.}, \emph{The on-line encyclopedia of integer sequences},
  \url{https://oeis.org/A307249}, Accessed: August 10, 2022.

\bibitem[Par17]{keli_thesis}
Keli Parker, \emph{Semistable modular compactifications of moduli spaces of
  genus one curves}, Ph.D. thesis, University of Colorado, 2017.

\bibitem[RSW19]{rsw}
Dhruv {Ranganathan}, Keli {Santos-Parker}, and Jonathan {Wise}, \emph{{Moduli
  of stable maps in genus one and logarithmic geometry I}}, Geometry \&
  Topology \textbf{23} (2019), 3315--3366.

\bibitem[Sch91]{schubert_pseudostable}
David Schubert, \emph{A new compactification of the moduli space of curves},
  Compositio Mathematica \textbf{78} (1991), no.~3, 297--313.

\bibitem[Smy11]{smyth_mstable}
David Smyth, \emph{Modular compactifications of the space of pointed elliptic
  curves {I}}, Compositio Mathematica \textbf{147} (2011), no.~3, 877--913.

\bibitem[Smy13]{smyth_zstable}
\bysame, \emph{Towards a classification of modular compactifications of
  $\mathscr{M}_{g,n}$}, Inventiones Mathematicae \textbf{192} (2013), no.~2,
  459--503.

\bibitem[Smy19]{smyth_psi}
\bysame, \emph{Intersections of {$\psi$}-classes on {$\moduli_{1,n}(m)$}},
  Transactions of the American Mathematical Society \textbf{372} (2019),
  8679--8707.

\end{thebibliography}

\end{document}